%% file: paper-mcom.tex
\documentclass{mcom-l}

\usepackage[utf8]{inputenc}
\usepackage[english]{babel}
\usepackage{amsmath, amssymb, gensymb, physics, upgreek, siunitx, bbm}
\usepackage{tikz, graphicx, epstopdf, float, hyperref, xspace}
\usepackage{pgfplots, pgfplotstable}
\usepackage{enumitem, verbatim, listings}
\usepackage{mathtools, stmaryrd}
\usepackage{booktabs, array, ragged2e}
\usepackage[export]{adjustbox}
\usepackage[nameinlink]{cleveref}
\usepackage{todonotes}
\usepackage{csvsimple}
\usepackage{diagbox,multirow}
\usepackage{subcaption}

\usetikzlibrary{trees, cd, babel}
\usetikzlibrary{perspective,positioning,calc}
\graphicspath{{./figures}}
\definecolor{paired1}{HTML}{a6cee3}
\definecolor{paired2}{HTML}{1f78b4}
\definecolor{paired3}{HTML}{b2df8a}
\definecolor{paired4}{HTML}{33a02c}
\definecolor{paired5}{HTML}{fb9a99}
\definecolor{paired6}{HTML}{e31a1c}
\definecolor{paired7}{HTML}{fdbf6f}
\definecolor{paired8}{HTML}{ff7f00}
\definecolor{paired9}{HTML}{cab2d6}
\definecolor{paired10}{HTML}{6a3d9a}

\hypersetup{
	colorlinks,
	linkcolor={red!50!black},
	citecolor={blue!50!black},
	urlcolor={blue!80!black}
}

\pgfplotsset{select coords between index/.style 2 args={
		x filter/.code={
			\ifnum\coordindex<#1\def\pgfmathresult{}\fi
			\ifnum\coordindex>#2\def\pgfmathresult{}\fi
		}
}}

\pgfplotsset{
	log x ticks with fixed point/.style={
		xticklabel={
			\pgfkeys{/pgf/fpu=true}
			\pgfmathparse{exp(\tick)}%
			\pgfmathprintnumber[fixed relative, precision=3]{\pgfmathresult}
			\pgfkeys{/pgf/fpu=false}
		}
	},
	log y ticks with fixed point/.style={
		yticklabel={
			\pgfkeys{/pgf/fpu=true}
			\pgfmathparse{exp(\tick)}%
			\pgfmathprintnumber[fixed relative, precision=3]{\pgfmathresult}
			\pgfkeys{/pgf/fpu=false}
		}
	}
}

\newcommand{\bigo}[1]{\mathcal{O}(#1)}

\let\grad\undefined
\let\curl\undefined
\let\div\undefined

\let\d\undefined
\let\tr\undefined
\DeclareMathOperator{\grad}{grad}
\DeclareMathOperator{\curl}{curl}
\DeclareMathOperator{\div}{div}

\DeclareMathOperator{\d}{d}
\DeclareMathOperator{\tr}{tr}
\DeclareMathOperator{\Exists}{\exists}
\DeclareMathOperator{\Forall}{\forall}

\newcommand{\Hgrad}{H(\grad)}
\newcommand{\Hcurl}{H(\curl)}
\newcommand{\Hdiv}{H(\div)}
\newcommand{\Ltwo}{L^2}

\newcommand{\mesh}{\mathcal{T}_h}

\newcommand{\pullback}{\mathcal{F}}

\renewcommand{\vec}[1]{\mathbf{#1}}

\DeclareMathOperator{\diag}{diag}
\DeclareMathOperator{\spn}{span}

\newcommand{\bu}{\vec{u}}
\newcommand{\bn}{\vec{n}}
\newcommand{\bt}{\vec{t}}

\renewcommand{\bf}{\vec{f}}

\newcommand{\bx}{\vec{x}}

\newcommand{\Khat}{\hat{K}}

\newcommand{\That}{\hat{T}}

\renewcommand{\P}{\mathrm{P}}

\newcommand{\CG}{\mathrm{CG}}
\newcommand{\DG}{\mathrm{DG}}
\newcommand{\Ned}{\mathrm{Ned}^{1}}
\newcommand{\RT}{\mathrm{RT}}
\newcommand{\BDM}{\mathrm{BDM}}
\newcommand{\NedTwo}{\mathrm{Ned}^{2}}

\newcommand{\rmone}{\mathrm{I}}
\newcommand{\rmtwo}{\mathrm{II}}
\newcommand{\interior}{\mathcal{I}}
\newcommand{\interface}{\Gamma}
\newcommand{\pafw}[2]{\mathrm{PAFW}_{#1}({#2})}
\newcommand{\ph}[2]{\mathrm{HT}({#1}, {#2})}
\newcommand{\jacobi}[1]{\mathrm{J}({#1})}

\newcommand{\Nedelec}{Ned\'el\'ec}

\crefname{equation}{}{}
\Crefname{equation}{Equation}{Equations}
\crefname{lemma}{Lemma}{Lemmas}
\crefname{theorem}{Theorem}{Theorems}
\crefname{corollary}{Corollary}{Corollaries}
\crefname{figure}{Figure}{Figures}
\Crefname{figure}{Figure}{Figures}
\crefname{table}{Table}{Tables}
\Crefname{table}{Table}{Tables}
\crefname{remark}{Remark}{Remarks}

\newtheorem{theorem}{Theorem}[section]
\newtheorem{lemma}[theorem]{Lemma}

\theoremstyle{definition}

\theoremstyle{remark}
\newtheorem{remark}[theorem]{Remark}

\theoremstyle{corollary}
\newtheorem{corollary}[theorem]{Corollary}

\numberwithin{equation}{section}

\begin{document}

\title[Fast solvers for simplicial de Rham]{Fast solvers for the 
high-order FEM simplicial de Rham complex}


\author[P. D. Brubeck]{Pablo D.\ Brubeck}
\address{Mathematical Institute,
	University of Oxford,
	Oxford, UK }
\email{brubeckmarti@maths.ox.ac.uk}
\thanks{PBD and PEF were supported by EPSRC grant EP/W026163/1 and 
the UKRI Digital Research Infrastructure Programme through the Science 
	and Technology Facilities Council's Computational Science Centre for 
	Research Communities (CoSeC).}

\author[P.E. Farrell]{Patrick E.\ Farrell}
\address{Mathematical Institute,
	University of Oxford,
	Oxford, UK}
\email{patrick.farrell@maths.ox.ac.uk}
\thanks{PEF was additionally supported by the Donatio Universitatis
	Carolinae Chair ``Mathematical modelling of multicomponent systems''.}

\author[R. C. Kirby]{Robert C.\ Kirby}
\address{Department of Mathematics,
	Baylor University,
	Waco, TX, US}
\email{robert\_kirby@baylor.edu}

\author[C. Parker]{Charles Parker}
\address{Ridgway Scott Foundation,
2544 S Forest Ln, Cedarville, MI}
\email{charles\_parker@alumni.brown.edu}
\thanks{CP acknowledges this material is based upon work supported by the 
National Science Foundation under Award No.~DMS-2201487}


\date{\today}


\begin{abstract}
	We present new finite elements for solving the Riesz maps of the de Rham 
	complex on triangular and tetrahedral meshes at high order. The finite 
	elements discretize the same spaces as usual, but with different basis 
	functions, so that the resulting matrices have desirable properties. 
	These properties mean that we can solve the Riesz maps to a given 
	accuracy in a $p$-robust number of iterations with $\mathcal{O}(p^6)$ 
	flops in three dimensions, rather than the na\"ive $\mathcal{O}(p^9)$ 
	flops.
	
	The degrees of freedom build upon an idea of Demkowicz et al.,
	and consist of integral moments on an equilateral reference simplex with 
	respect to a numerically computed polynomial basis that is orthogonal in 
	two different inner products. As a result, 
	the interior-interface and interior-interior couplings are provably weak, 
	and we devise a preconditioning strategy by neglecting them.
	The combination of this approach with a space decomposition method on 
	vertex and edge star patches allows us to efficiently solve the canonical 
	Riesz maps at high order.
	We apply this to solving the Hodge Laplacians of the de Rham complex with 
	novel augmented Lagrangian preconditioners.
\end{abstract}

\maketitle



\section{Introduction} \label{sec:introduction}

High-order finite element discretizations have very attractive properties, 
such as rapid convergence and high arithmetic intensity. However, their 
implementation requires great care, as na\"ive approaches to operator 
evaluation and linear system solution can scale badly in terms of memory and 
floating-point operations (flops) as the polynomial degree $p$ increases. On 
a single $d$-dimensional cell, the number of degrees of freedom scales like 
$p^d$; computing the $\mathcal{O}(p^{2d})$ entries of a dense stiffness 
matrix with standard Lagrange-type elements would cost $\mathcal{O}(p^{3d})$ 
flops. Making high-order discretizations competitive therefore requires 
better algorithms.

For tensor-product cells (quadrilaterals and hexahedra), a tensor-product 
basis naturally exposes the required structure to decouple multidimensional 
problems into a sequence of one-dimensional problems, reducing computational 
cost and storage. For example, the sum-factorization assembly strategy 
enables fast matrix-free operator evaluation in $\mathcal{O}(p^{d+1})$ 
operations and $\mathcal{O}(p^{d})$ storage~\cite{orszag80}, and the fast 
diagonalization method (FDM) addresses the solution of separable problems on 
structured domains with the same complexities~\cite{lynch64}. The previous 
work of the first two authors employing FDM-inspired finite element bases 
with suitable space decompositions yields the solution of certain separable 
and non-separable problems on unstructured domains with the same 
complexities~\cite{brubeck22,brubeck24}. In particular, these complexities 
are achieved on tensor-product cells for the solution of the following 
partial differential equations (PDEs), posed on a bounded Lipschitz domain 
$\Omega \subset \mathbb{R}^d$ with $d = 3$ for concreteness:
\begin{alignat}{5}
	\beta u - \nabla \cdot \left(\alpha \nabla u\right) &= f 
	\mbox{~in~}\Omega, \quad
	&u &= 0 \mbox{~on~} \Gamma_D, \quad
	&\alpha\nabla u\cdot\bn &= 0 \text{~on~} \Gamma_N; \label{eq:hgrad}\\
	\beta\bu + \nabla\times \left( \alpha \nabla\times \bu\right) &= \bf 
	\mbox{~in~} \Omega,
	&\bu\times\bn &= 0 \mbox{~on~} \Gamma_D, \quad
	&\alpha\nabla \times \bu \times \bn &= 0 \text{~on~} \Gamma_N; 
	\label{eq:hcurl}\\
	\beta\bu - \nabla \left(\alpha \nabla\cdot\bu\right) &= \bf \mbox{~in~} 
	\Omega, \quad
	&\bu\cdot\bn &= 0 \mbox{~on~} \Gamma_D, \quad
	&\alpha\nabla\cdot\bu &= 0 \text{~on~} \Gamma_N. \label{eq:hdiv}
\end{alignat}
Here $\alpha,\beta: \Omega \to \mathbb{R}_+$ are given coefficients, $f: 
\Omega \to \mathbb{R}$ and $\bf: \Omega \to \mathbb{R}^3$ are given source 
functions, $\Gamma_D \subseteq \partial \Omega$, and $\Gamma_N = \partial 
\Omega \setminus \Gamma_D$.
For $\alpha = \beta = 1$, these equations are the so-called \emph{Riesz maps} 
associated with subsets of the spaces
$H(\grad, \Omega) = H^1(\Omega)$, $H(\curl, \Omega)$, and $H(\div, \Omega)$, 
respectively.
For brevity we shall write $\Hgrad = H(\grad, \Omega)$ etc.~where there is no 
potential confusion.
These equations \cref{eq:hgrad,eq:hcurl,eq:hdiv} often arise as subproblems 
in the construction of preconditioners for more complex systems involving 
solution variables in $\Hgrad$, $\Hcurl$, and 
$\Hdiv$~\cite{hiptmair06,kirby10,mardal11,malek14}.

Achieving $\mathcal{O}(p^{d+1})$ operations and $\mathcal{O}(p^{d})$ storage 
on simplices is more challenging.
Bases constructed from collapsed-coordinate mappings~\cite{karniadakis13} 
give one possible approach, allowing a similar sum-factorization evaluation 
strategy.
Perhaps surprisingly, Bernstein--B\'ezier expansions also admit 
sum-factorized matrix-free operator evaluation in $\mathcal{O}(p^{d+1})$ 
operations and $\mathcal{O}(p^{d})$ storage~\cite{ainsworth11, 
kirby2011fast}, and such techniques have been generalized to the entire de 
Rham complex~\cite{ainsworth2018bernstein, kirby2014low}. For the $\Hgrad$ 
mass matrix ($\alpha = 0$) in 2D, there are $\mathcal{O}(p^{3})$ solvers 
available \cite{Ainsworth2019p3}. However, there is no known solver for the 
linear systems arising from FEM discretizations of 
\cref{eq:hgrad,eq:hcurl,eq:hdiv} with these complexities. This manuscript 
makes a step towards addressing this challenge.

The essential idea of our approach is to devise new finite elements for the 
$L^2$ de Rham complex with favorable properties for solvers. With the basis 
functions we propose, the interior and interface degrees of freedom are 
provably weakly coupled. This motivates the use of block preconditioning 
strategies that decouple the interface and interior degrees of freedom. The 
degrees of freedom yielding this weak coupling are designed within the 
framework of Demkowicz et al.~\cite{demkowicz00}, by making a clever choice 
for bases of bubble spaces that arise in the construction of the degrees of 
freedom. These bases for the bubble spaces are constructed by solving a 
handful of offline eigenproblems on the reference cell for each polynomial 
degree $p$. This combination of finite elements and block preconditioning 
strategy solves the Riesz maps of the $L^2$ de Rham complex in 
$\mathcal{O}(p^{3(d-1)})$ operations and $\mathcal{O}(p^{2(d-1)})$ storage, 
worse than the optimal complexities achieved on tensor-product cells, but 
orders better than the na\"ive approach. This is borne out in numerical 
experiments, as shown in \Cref{fig:complexities}.

\begin{figure}
	\centering
	\input{./figures/plot_complexity.tex}
	\caption{
		Flop counts, peak memory usage, and nonzeros in the sparse matrices 
		and patch factors in 
		the solution of the Riesz maps ($\alpha = \beta = 1$) on a 
		Freudenthal mesh with 3 cells in each direction.
	}
	\label{fig:complexities}
\end{figure}

Several other works have addressed the same problem of solving the Riesz maps 
at high order. 
For \cref{eq:hgrad}, \v{S}ol\'{i}n and Vejchodsk\'{y} \cite{Solin08} 
introduced the idea of using generalized eigenfunctions as interior basis 
functions, which we also employ in our construction, but they do not address 
how to construct efficient solvers. 
Also for \cref{eq:hgrad}, Casarin and Sherwin 
\cite{Sherwin2001low} constructed nonoverlapping Schwarz 
methods for the Schur complement system requiring exact interior solves;
however, these interior solves require $\mathcal{O}(p^{3d})$ operations.
Low-order refined preconditioning is provably effective on 
tensor-product cells \cite{pazner22}, but the extension to simplicial cells 
is 
delicate and has only been studied for \cref{eq:hgrad} on triangles 
\cite{chalmers2018}. 
Beuchler and coauthors construct sparse hierarchical 
bases for discretizations of the de Rham complex using Jacobi polynomials 
\cite{beuchler2007sparse,Beuchler12,Beuchler12hdiv,Beuchler13}; this approach 
yields sparser matrices than ours, but the bases here concentrate the nonzero 
entries in the interface block, enabling the efficient interior-interface 
space decomposition described in 
\cref{thm:generic-subspace-decomp-interface-interior-split} below.

A disadvantage of our approach is that the resulting preconditioners are not 
parameter-robust in $\alpha$ and $\beta$. We design elements suitable for the 
stiffness-dominated case, so that the preconditioner is provably robust for 
$\beta/\alpha \in [0, 1)$. This is the more difficult case of wider interest 
(e.g.~arising in augmented Lagrangian preconditioning). We anticipate that it 
is possible to design a basis for the mass-dominated case ($\beta/\alpha \gg 
1$), but with very different properties.

\section{Finite element discretization}

Let $\Omega \subset \mathbb{R}^d$ be a bounded Lipschitz domain with boundary 
$\Gamma \coloneqq \partial \Omega$.
We denote the spaces consisting of $k$-forms as $H(\d^k, \Omega) = \{v\in 
L^2(\Omega) : \d^k v \in L^2(\Omega)\}$, 
where $\d^k$ denotes the exterior derivative on $k$-forms ($\d^0 = \grad$, 
$\d^1 
= \curl$, $\d^2 = \div$, and $\d^3 = 0$).
The trace operator on theses spaces $\tr^k_{\Gamma} : H(\d^k, \Omega) \to 
H^{-1/2}(\Gamma)$ is defined by
\begin{equation}
	\label{eq:tr-k}
	\tr^k_{\Gamma} v \coloneqq \begin{cases}
		v|_{\Gamma} & \text{if } k = 0, \\
		\Pi_{\Gamma} v|_{\Gamma} & \text{if } k = 1, \\
		\bn \cdot v|_{\Gamma} & \text{if } k = 2, 
	\end{cases}
\end{equation}
where $\Pi_\Gamma$ denotes the projection of a
vector onto the tangent space of $\Gamma$. The trace operator is continuous, 
but not surjective if $k < 2$ \cite{boffi13}.

On a shape-regular, conforming simplicial mesh of $\Omega$, denoted by 
$\mesh$,
we consider well-known finite element subcomplexes of the $\Ltwo$ de Rham 
complex.
For concreteness, we will restrict our discussion to the case $d=3$:
\begin{figure}[H] 
	\centering
	\begin{tikzcd}
		\CG_p \arrow[r, "\grad"] & \Ned_p \arrow[r, "\curl"] & 
		\RT_p \arrow[r, "\div"] & \DG_{p-1} \\
		\Hgrad \arrow[r, "\grad"] \arrow[u] \arrow[d]  & \Hcurl \arrow[r, 
		"\curl"]
		\arrow[d] \arrow[u] & \Hdiv \arrow[r, "\div"] \arrow[u] \arrow[d] & 
		\Ltwo \arrow[u] \arrow[d] \\
		\CG_p \arrow[r, "\grad"] & \NedTwo_{p-1} \arrow[r, "\curl"] & 
		\BDM_{p-2} \arrow[r, "\div"] & \DG_{p-3}
	\end{tikzcd}
	\caption{The $\Ltwo$ de Rham complex in three dimensions (middle), and 
	the finite element subcomplexes of the first (above) and second (below) 
	kinds.}
	\label{fig:3DL2deRham}
\end{figure}
\noindent
The top and bottom row of \cref{fig:3DL2deRham} correspond to the 
subcomplexes of first and second kinds, respectively.
In particular, $\CG_p$ and  $\DG_p$  denote the usual spaces of continuous 
and discontinuous piecewise polynomials,
$\Ned_p$~\cite{nedelec80} and $\NedTwo_p$~\cite{ned-second-element} 
correspond to the $\Hcurl$-conforming spaces of the first and second kinds, 
and
$\RT_p$~\cite{rt-element,nedelec80} and 
$\BDM_p$~\cite{brezzi1985,ned-second-element} correspond to the 
$\Hdiv$-conforming spaces of the first and second kinds. One can also form a 
hybrid complex by interchanging kinds as in \cref{fig:3DL2deRham-hybrid}:
\begin{figure}[H]
	\centering
	\begin{tikzcd}
		\CG_p \arrow[r, "\grad"] & \Ned_p \arrow[r, "\curl"] & 
		\BDM_{p-1} \arrow[r, "\div"] & \DG_{p-2} \\
		\CG_p \arrow[r, "\grad"] & \NedTwo_{p-1} \arrow[r, "\curl"] & 
		\RT_{p-1} \arrow[r, "\div"] & \DG_{p-1}
	\end{tikzcd}	
	\caption{Hybrid subcomplexes of the $L^2$ de Rham complex.}
	\label{fig:3DL2deRham-hybrid}		
\end{figure}

To unify our discussion, we employ the same notation $X^k_p(\That)$ to denote 
the 
discrete spaces of $k$-forms of the first or
second kind on a reference equilateral
simplex $\That$. Here, $p$ denotes the degree of the polynomial spaces. In 
particular, for the first kind, we have
\begin{equation}
	\label{eq:first-kind-def}
	X^k_p(\That) \coloneqq \left\{ \begin{array}{rlll}
		\CG_p(\That) &\coloneqq& \P_{p}(\That) & \text{if } k = 0, \\
		\Ned_p(\That) &\coloneqq& \P_{p-1}(\That)^3 + \P_{p-1}(\That)^3\times 
		\bx 
		& \text{if } k = 1, \\
		\RT_p(\That) &\coloneqq& \P_{p-1}(\That)^3 + \P_{p-1}(\That) \bx & 
		\text{if } k = 2, \\
		\DG_{p}(\That) &\coloneqq& \P_{p}(\That) & \text{if } k = 3, 
	\end{array}\right.
\end{equation}
while for the second kind we have
\begin{equation}
	\label{eq:second-kind-def}
	X^k_p(\That) \coloneqq \left\{ \begin{array}{rlll}
		\CG_p(\That) &\coloneqq& \P_{p}(\That) & \text{if } k = 0, \\
		\NedTwo_{p}(\That) &\coloneqq& \P_{p}(\That)^3 & \text{if } k = 1, \\
		\BDM_{p}(\That) &\coloneqq& \P_{p}(\That)^3 & \text{if } k = 2, \\
		\DG_{p}(\That) &\coloneqq& \P_{p}(\That) & \text{if } k = 3.
	\end{array}\right.
\end{equation}
We define the global finite element space $X^k_p(\mesh)$ 
in terms of the space on the reference cell $X^k_p(\That)$ via the 
standard pullback:
\begin{equation}
	X^k_p(\mesh) \coloneqq \{ v \in H(\d^k, \Omega): \Forall T \in \mesh 
	\ \Exists \hat{v} \in X^k_p(\That) \text{ s.t.} \left.v\right|_T = 
	\pullback^k_T(\hat{v}) \},
\end{equation}
where for a cell $T \in \mesh$, the pullback $\pullback^k_T : H(\d^k, \That) 
\to 
H(\d^k, T)$ is defined by:
\begin{equation} \label{eq:pullback-3d-def}
	\pullback^k_{T}(\hat{v}) \coloneqq \begin{cases}
		\hat{v} \circ F_T^{-1} & \text{if } k = 0, \\
		J_T^{-\top} \hat{v} \circ F_T^{-1} & \text{if } k = 1, \\
		(\det J_T)^{-1} J_T \hat{v} \circ F_T^{-1} & \text{if } k = 2, \\
		(\det J_T)^{-1} \hat{v} \circ F_T^{-1} & \text{if } k = 3,
	\end{cases}
\end{equation}
where $F_T : \That \to T$ is a fixed bijective mapping from the reference 
cell 
$\That$ to the physical cell $T$
and $J_T$ denotes its Jacobian matrix. 
Denoting by $\Delta_l(\mesh)$ the set of $l$-dimensional subsimplices of 
$\mesh$, 
we note that a discrete $k$-form
$v \in X^k_p(\mesh)$ and any subset $S \subset \Delta_l(\mesh)$ with $l = 
k:d-1$, 
the trace $\tr^k_{S} v$, given by formula \cref{eq:tr-k} with $\Gamma$
replaced by $S$ is well-defined (see e.g.\ \cite[Lemma 5.1]{Arnold2006}). We 
also define the trivial trace 
$\tr^k_S$ for a $d$-dimensional
subset $S \subset \mesh$ by $\tr^k_S v = v|_{S}$.

With these spaces, the discrete weak formulation of problems 
\cref{eq:hgrad,eq:hcurl,eq:hdiv} is to find $u \in X^k_{p,D}(\mesh)$ such that
\begin{equation} \label{eq:weak-form}
	a^k(u, v) \coloneqq (u, \beta v) + (\d^k u, \alpha \d^k v) = F(v) \qquad
	\Forall v \in X^k_{p,D}(\mesh),
\end{equation}
where $X^k_{p,D}(\mesh)$ is the subset of $X^k_p(\mesh)$ with zero trace on 
$\Gamma_D$:
\begin{equation}
	X^k_{p,D}(\mesh) \coloneqq \{u \in X^k_p(\mesh) : \tr_{\Gamma_D}^k u = 0 
	\}.
\end{equation}
For standard bases for $X^k_{p,D}(\mesh)$, the matrix corresponding to 
\cref{eq:weak-form} will have dense $\bigo{p^d} \times \bigo{p^d}$ 
blocks.

\section{Orthogonality-promoting discretizations} \label{sec:dofs}

We wish to construct bases for $X^k_{p}(\mesh)$ that enable fast solvers for 
problem 
\cref{eq:weak-form} at high order. In particular,
we wish to promote orthogonality of the basis in the $\Ltwo$ and $H(\d^k)$ 
inner 
products. However, constructing such 
bases would involve the
solution of global eigenproblems, which is not appealing. Instead, we devise 
bases so that the \emph{interior} basis functions are orthogonal in
these inner products on the reference cell. This will ensure that the 
interior-interior block of the mass and stiffness matrices associated with
the $\Ltwo(\That)$ and $H(\d^k, \That)$ inner products are diagonal on the 
reference cell.

To form bases with this interior orthogonality property, we follow the 
definition of interpolation degrees of freedom from
Demkowicz et al.~\cite{demkowicz00}. Demkowicz et al.~employed these degrees 
of freedom to define interpolation operators with commuting diagram 
properties; here, we will use them to define finite elements, through the 
Ciarlet
dual basis construction~\cite{ciarlet1978}. In particular, given a set of 
degrees 
of freedom $\{ \ell_j \}$, the basis $\{ \phi_j \}$ is constructed so that 
$\ell_j(\phi_i) = \delta_{ij}$. Although a basis for 
$p$-version methods is typically constructed 
directly and without reference to a set of degrees of freedom, having an 
explicit set of degrees of freedom facilitates the implementation of 
the method in some finite element packages such as Firedrake \cite{firedrake}.
In \cref{rem:harmonic-extension-basis}, we briefly describe an equivalent 
approach that is more in line with a $p$-version construction 
and show that the basis is recursive in dimension.

For the remainder of the manuscript, we select one of the discrete 
subcomplexes in \cref{fig:3DL2deRham,fig:3DL2deRham-hybrid} and drop the 
subscript ``$p$" so that the complexes read
\begin{equation}\label{eq:abstract-complex}
	\begin{tikzcd}[ampersand replacement=\&]
		X^0(\That)  \arrow[r, "\d^{0}"] \& X^{1}(\That) \arrow[r, "\d^1"] \& 
		\cdots \arrow[r, "\d^{d-1}"] \& X^{d}(\That) \arrow[r] \& 0.
	\end{tikzcd}
\end{equation}
In particular, $X^k(\That)$ may be a first or second kind space provided that
\cref{eq:abstract-complex} is exact. Although we will first proceed in this 
abstract 
setting, the degrees of freedom for each of the spaces in
\cref{eq:first-kind-def,eq:second-kind-def} and the corresponding spaces in 2D
are listed explicitly in \cref{sec:concrete-dofs}. 

\subsection{Degrees of freedom with commuting diagram properties}

Let $k \in 0:d-1$ be fixed. The bubble spaces 
on a cell subentity $S \in \Delta_l(\That)$ 
with $l \in k+1:d$, which play a key role in the degrees of freedom 
in~\cite{demkowicz00}, are defined by
\begin{equation}
	\mathring{X}^k(S) \coloneqq \left\{\tr_S^k v : v \in X^k(\That) \text{ 
	and } 
	\tr_{\partial S}^{k} v  = 0
	\right\}.
\end{equation}
We will later construct bases for these bubble spaces so as to promote 
sparsity.
We also define differential operators on the trace subspace $\tr_S^k 
X^k(\That)$ 
for $S \in \Delta_l(\That)$ with $l \in k+1:d-1$, as follows:
\begin{align}
	\d_S^0 v &\coloneqq \grad_S v \quad \text{and} \quad \d_S^1 v \coloneqq 
	\curl_S v,
\end{align}
where $\grad_S$ denotes the tangential surface gradient taking values in 
$\mathbb{R}^l$  
and $\curl_S$ is the surface curl/rot operator taking values in $\mathbb{R}$.
Crucially, these operators
are defined so that
\begin{alignat}{2}
	\label{eq:trace-ds-commute}
	\d_S^k \tr_S^k v &= \tr_S^{k+1} \d^k v \qquad & &\Forall v \in 
	X^k(\That).
\end{alignat}	

The commuting interpolation procedure given in \cite{demkowicz00,zaglmayr06} 
is based on the following linear functionals
\begin{subequations} 
	\begin{alignat}{3}
		\label{eq:demkowicz-whitney-dofs}
		&(q, \tr^k_S v)_{S}, \qquad & &\Forall S \in \Delta_k(\That), \ & 
		&\Forall q 
		\in \tr^k_S X^k(\That),\\
		\label{eq:demkowicz-type-I-dofs}
		&(\d_S^k q, \d^k_S \tr^k_S v)_{S} \qquad & &\Forall S \in 
		\bigcup_{l=k+1}^d 
		\Delta_l(\That), \ & &\Forall q \in  \mathring{X}^k(S), \\
		\label{eq:demkowicz-type-II-dofs}
		&(\d_S^{k-1} q, \tr^k_S v)_{S} \qquad & &\Forall S \in 
		\bigcup_{l=k+1}^d 
		\Delta_l(\That), \ & &\Forall q \in  \mathring{X}^{k-1}(S),
	\end{alignat}
\end{subequations}
which we will use to construct a set of degrees of freedom by fixing bases 
for each of the subspaces above. 

\subsection{Local space decomposition}

We observe that in the linear functionals
\cref{eq:demkowicz-type-I-dofs,eq:demkowicz-type-II-dofs},
only $\d^k_S q$ and $\d^{k-1}_{S} q$ appear, and so the construction of a 
unisolvent 
set of degrees of freedom requires finding minimal sets $\{q \}$ such 
that $\{\d^k_S q\}$ forms a basis for $\d^k_S \mathring{X}^{k}(S)$. Since 
$\d^k_S$ may have a nontrivial kernel, it will be
convenient to characterize this kernel. This characterization can be readily 
achieved by
considering the bubble subcomplex for 
$S \in \Delta_l(\That)$ with $l \in k+1:d$: 
\begin{equation}\label{eq:bubble-subcomplex}
	\begin{tikzcd}[ampersand replacement=\&]
		\mathring{X}^{k-1}(S) \arrow[r, "\d^{k-1}_S"] \& \mathring{X}^{k}(S) 
		\arrow[r, "\d^k_S"] \& 
		\d^k_S \mathring{X}^{k}(S) \arrow[r] \& 0.
	\end{tikzcd}
\end{equation}
We define the subspaces
\begin{subequations}
	\begin{alignat}{2}
		\label{eq:type-I-bubble}
		\mathring{X}^{k, \rmone}(S) &\coloneqq \left\{ v \in 
		\mathring{X}^{k}(S) : (v, w)_{S} = 0 \ \Forall w \in \ker(\d_S^{k}; 
		\mathring{X}^{k}) \right\}, \\
		\label{eq:type-II-bubble}
		\mathring{X}^{k, \rmtwo}(S) &\coloneqq 
		\ker(\d_S^{k}; \mathring{X}^{k}),
	\end{alignat} 	
\end{subequations}
Since the complex \cref{eq:bubble-subcomplex} is exact (see 
e.g.~\cite[Corollary 2]{Licht2017complexes}), the kernel of $\d_S^k$ 
acting on $\mathring{X}^{k}(S)$ is simply $\d^{k-1}_S \mathring{X}^{k-1}(S)$ 
and 
so $\mathring{X}^{k, \rmtwo}(S) = \d^{k-1}_S \mathring{X}^{k-1}(S)$.
In particular, we have the following $L^2(S)$-orthogonal decomposition:
\begin{equation}\label{eq:typeI-typeII-orthog-decomp}
	\mathring{X}^{k}(S)
	= \mathring{X}^{k, \rmone}(S) \oplus \mathring{X}^{k, \rmtwo}(S) = 
	\mathring{X}^{k, \rmone}(S) \oplus \d_S^{k-1} \mathring{X}^{k-1, 
	\rmone}(S).
\end{equation}
Therefore, we may select the functions in 
\cref{eq:demkowicz-type-I-dofs,eq:demkowicz-type-II-dofs}
from $\mathring{X}^{k, \rmone}(S)$ and $\mathring{X}^{k-1, \rmone}(S)$, 
respectively.

We now perform a similar decomposition of the trace space $\tr^k_S 
X^k(\That)$, $S \in \Delta_k(\That)$, to select the functions in 
\cref{eq:demkowicz-whitney-dofs}.
We adopt the convention that for $E \in \Delta_1(\That)$ and $v \in 
X^1(\That)$, the projection onto the tangent space of $E$ is scalar-valued, 
and thus $\tr^k_S X^k(\That)$, $S \in \Delta_k(\Khat)$, is scalar-valued for 
$k \in 0:d-1$.
Standard techniques also show that we have the following 
$L^2(S)$-orthogonal decomposition of the 
trace space 
$\tr^k_S X^k(\That)$:
\begin{equation}\label{eq:ktrace-decomp}
	\tr^k_S X^{k}(\That) 
	= \tr^k_S W^k(\That) \oplus \d_S^{k-1} \mathring{X}^{k-1, \rmone}(S) 
	= \mathbb{R} \oplus \d_S^{k-1} \mathring{X}^{k-1, \rmone}(S), 
\end{equation}
where $W^k(\That)$ denotes the space of Whitney forms, 
defined as the lowest-order finite element space of the first kind; e.g., if 
$d=3$, then 
\begin{equation}
	W^k(\That) \coloneqq \begin{cases}
		\CG_1(\That) & \text{if } k = 0, \\
		\Ned_1(\That) & \text{if } k = 1, \\
		\RT_1(\That) & \text{if } k = 2, \\
		\DG_0(\That) & \text{if } k = 3. 
	\end{cases}
\end{equation}
In contrast to \cite{demkowicz00}, we further decompose 
\cref{eq:demkowicz-whitney-dofs} using \cref{eq:ktrace-decomp} by choosing 
$q$ in \cref{eq:demkowicz-whitney-dofs}
to be the constant 1 and of the form $\d_S^{k-1} \tilde{q}$, 
where $\tilde{q} \in 
\mathring{X}^{k-1, \rmone}_p(S)$. As we will see in 
\cref{lem:basis-contains-whitney}, this choice ensures that the 
canonical basis for the Whitney forms is contained in the high-order basis.
We will achieve an interior orthogonality property through a careful and 
novel choice 
of basis for the bubble spaces $\mathring{X}^{k, \rmone}(S)$.

\subsection{Abstract construction of orthogonality-promoting basis}

We define the finite elements via a Ciarlet triple $(\That, X^k(\That), 
\mathcal{L} )$, 
where the set $\mathcal{L} = \{\ell_j^k\}$ is a basis for the dual space 
$(X^k(\That))^*$,
with degrees of freedom $\ell_j^k$ that map a $k$-form $v \in X^k(\That)$ 
to a real number $\ell_j^k(v)$. Let $N^k_S \coloneqq \dim \mathring{X}^{k, 
	\rmone}(S)$ denote the dimension of 
the type-$\rmone$ bubble spaces for $S \in \Delta_l(\That)$, $l \in k+1:d$. 
For notational convenience, we also define $N^{k}_S \coloneqq \dim \tr^k_S 
W^k(\That) = 1$ for $S \in \Delta_{k}(\That)$.
Thanks to \cref{eq:typeI-typeII-orthog-decomp,eq:ktrace-decomp}, there holds 
\begin{subequations}
	\label{eq:bubble-trace-dim-counts}
	\begin{alignat}{2}
		\dim \mathring{X}^k(S) &= N^k_S + N^{k-1}_S \qquad & &\Forall S \in 
		\bigcup_{l=k+1}^{d} \Delta_l(\That), \\
		\dim \tr_S^k X^k(\That) &= 1 + N^{k-1}_S \qquad & &\Forall S \in  
		\Delta_k(\That). 
	\end{alignat} 	
\end{subequations}

For $k \in 0:d-1$ and $l \in k+1 : d$ and each subsimplex $S \in 
\Delta_l(\That)$,
we compute an eigenbasis of the bubbles 
$ \{ \psi_{S, j}^{k} : j \in 1:N^k_S \} 
\subseteq \mathring{X}^{k, \rmone}(S)$
that solves the following generalized eigenvalue problem:
\begin{equation}
	\label{eq:ref-eval-problem-x}
	(\d_S^{k} \psi_{S, j}^{k}, \d_S^k \psi_{S, i}^{k})_{S} = \delta_{ij} 
	\quad \text{and} \quad 
	(\psi_{S, j}^{k}, \psi_{S, i}^{k})_{S} = \lambda_{S, j} \delta_{ij} 
	\qquad 
	\Forall~ i, j \in 1:N^k_S.
\end{equation}
The eigenbases $\{\psi_{S, j}^{k}\}$ are numerically computed offline and 
only 
once for each relevant reference subsimplex. Then, by construction, the set 
$\{ 
\psi_{S, j}^{k} \}$ forms a basis for $\mathring{X}^{k, \rmone}(S)$, and the 
set 
$\{ \d_S^k \psi_{S, j}^{k} \}$ forms a basis for $\d_S^k \mathring{X}^{k}(S)$.

We choose $\That$ to be an equilateral simplex
so that only one facet, edge, etc.~needs to be computed,
and the rest are constructed by the appropriate pullback.
We compute the solution to \cref{eq:ref-eval-problem-x} 
by first solving the following augmented eigenvalue problem on the whole 
bubble 
space $\mathring{X}^{k}$: 
\begin{equation}
	\label{eq:ref-eval-problem-x-aug}
	(\d_S^{k} \hat{\psi}_{S, j}^{k}, \d_S^k \hat{\psi}_{S, i}^{k})_{S} = 
	\mu_{S, j}\delta_{ij} \quad \text{and} \quad 
	(\hat{\psi}_{S, j}^{k}, \hat{\psi}_{S, i}^{k})_{S} = \delta_{ij} \qquad 
	\Forall~ i, j \in 1:\dim \mathring{X}^{k}.
\end{equation}
Then, we renormalize the subset of eigenfunctions for which $\mu_{S, j} \neq 
0$ to
form the solution to \cref{eq:ref-eval-problem-x}.

Our degrees of freedom for the discrete spaces in the $\Ltwo$-de Rham complex
$X^k(\That)$ are defined as follows.
On each subsimplex we prescribe two types of degrees of freedom if $k>0$.
On each $k$-simplex $S \in \Delta_k(\That)$,
we augment the degrees of freedom for Whitney forms \cref{eq:whitney-dofs} 
with integral moments against the exterior derivative of the basis for the 
bubble 
space of $(k-1)$-forms $\mathring{X}^{k-1}(S)$.
On the remaining sub-simplices of dimension $l \in k+1:d$,
$S \in \Delta_l(\That)$, we prescribe 
integral moments of $\d^k_S\tr^k_S v$ against the eigenbasis for $\d^{k}_S 
\mathring{X}^{k}$,
and integral moments of $\tr^k_S v$ against the eigenbasis for $\d^{k-1}_S 
\mathring{X}^{k-1}(S)$:
\begin{subequations} \label{eq:abstract-dofs}
	\begin{alignat}{2}
		\label{eq:whitney-dofs}
		\ell^{k, \rmone}_{S, 1}(v) &\coloneqq (1, \tr^k_S v)_{S}, \qquad & 
		&\Forall S 
		\in \Delta_k(\That),\\
		\label{eq:type-I-dofs}
		\ell^{k, \rmone}_{S, j}(v) &\coloneqq (\d_S^k \psi_{S, j}^{k}, \d_S^k 
		\tr^k_S 
		v)_{S} \qquad & &\Forall j \in 1:N^k_S, \ 
			\Forall S \in \bigcup_{l=k+1}^d \Delta_l(\That), \\
		\label{eq:type-II-dofs}
		\ell^{k, \rmtwo}_{S, j}(v) &\coloneqq (\d_S^{k-1} \psi_{S, j}^{k-1}, 
		\tr^k_S 
		v)_{S} \qquad & &\Forall j \in 1:N^{k-1}_S, \ 
			\Forall S \in \bigcup_{l=k}^d \Delta_l(\That). 
	\end{alignat}
\end{subequations}
We collect the degrees of freedom into a single set
\begin{equation} \label{eq:abstract-dofs-stiff}
	\mathcal{L} :=   \bigcup_{l = k}^{d} \bigcup_{S \in 
		\Delta_{l}(\That)} \{ \ell^{k, \rmone}_{S, j} : j \in 1:N^k_S \} \cup 
		\{ 
	\ell^{k, \rmtwo}_{S, j} : j \in 1:N^{k-1}_S \}. 
\end{equation}

\noindent The unisolvence of the degrees of freedom 
$\mathcal{L}$ follows 
from standard arguments.
\begin{lemma} \label{lem:abstract-dofs-stiff-unisolvant}
	The degrees of freedom \cref{eq:abstract-dofs-stiff} are unisolvent on 
	$X^k(\That)$.
\end{lemma}
\begin{proof}
	We first show that $\dim\mathcal{L} = \dim X^k(\That)$. 
	Note 
	that
	\begin{align*}
		\dim\mathcal{L} &= \sum_{S \in \Delta_k(\That)} \left( 
		1 + 
		N_S^{k-1} \right) + \sum_{l=k+1}^{d} \sum_{S \in \Delta_l(\That)} 
		\left( 
		N_S^k + N_{S}^{k-1} \right)  \\
		&= \sum_{S \in \Delta_k(\That)} \dim \tr_S^k 
		X^k(\That) + \sum_{l=k+1}^{d} \sum_{S \in \Delta_l(\That)} \dim 
		\mathring{X}^k(S),
	\end{align*}
	where we used \cref{eq:bubble-trace-dim-counts}. Thanks to 
	\cite[Theorem 5.5 \& Lemma 5.6]{Arnold2010}, we conclude that $\dim 
	\mathcal{L} = \dim X^k(\That)$.
	
	Now suppose that $v \in X^k(\That)$ satisfies $\ell(v) = 0$ for all $\ell 
	\in \mathcal{L}$. Since \cref{eq:whitney-dofs} vanishes, the 
	decomposition \cref{eq:ktrace-decomp} shows that for all $S \in 
	\Delta_k(\That)$, $\tr_S^k v = \d_S^{k-1} w_S$ for some $w_S \in 
	\mathring{X}^{k-1,\rmone}(S)$. However, the degrees of freedom 
	\cref{eq:type-II-dofs} vanish, and so $\tr_S^k v \equiv 0$.
	
	We may now proceed inductively on $l \in k+1:d$. For $S \in 
	\Delta_l(\That)$ 
	with $l \in k+1:d$, $\tr_S^k v \in \mathring{X}^k(S)$, and so $\d_S^k 
	\tr_S^k 
	v \equiv 0$ by \cref{eq:type-I-dofs}. By the exactness of 
	\cref{eq:bubble-subcomplex},  $\tr_S^k v = \d_S^{k-1} w_S$ for some $w_S 
	\in 
	\mathring{X}^{k-1,\rmone}(S)$. The degrees of freedom 
	\cref{eq:type-II-dofs} then 
	give $w_S \equiv 0$. Since this holds with $S = \That$, $v \equiv 0$. 
	The unisolvence of $\mathcal{L}$ now follows.
\end{proof}

\subsection{Concrete definition of orthogonality-promoting degrees of 
freedom} 
\label{sec:concrete-dofs}

For the sake of readability, we have simplified our notation slightly by
dropping the $k$-form superscript
and explicitly writing out the trace operators.
In particular, the the simplified notation in this section is limited 
to this section only, as it clashes with notation used elsewhere.

\subsubsection{Degrees of freedom for $\Hgrad$} \label{sec:hgrad}

We discretize $H(\grad, \That)$ with the space $\CG_p(\That)$.
For each subsimplex $S\in \Delta_l(\That)$ with $l\in 1:d$,
we compute the eigenbasis $\{\psi_{S,j}\}$ for
$\mathring{X}^{0,\rmone}(S) = \mathring{X}^0(S)$ satisfying
\begin{equation} \label{eq:hgrad-fdm}
	(\grad_S\psi_{S, j}, \grad_S\psi_{S, i})_{S} = \delta_{ij}, \quad
	(\psi_{S, j}, \psi_{S, i})_{S} = \lambda_j\delta_{ij},
\end{equation}
where we recall that $\grad_S$ is the tangential gradient on $S$. The degrees 
of freedom for $X^0(\That)$ in \cref{eq:abstract-dofs-stiff} are 
point evaluations at vertices 
and integral moments of surface gradients
on each higher-dimensional subsimplex:
\begin{subequations}
	\label{eq:hgrad-dofs}
	\begin{alignat}{2}
		\label{eq:hgrad-dofs-low-order} 
		\ell_V(v) &= v(V)  \qquad & &\Forall V \in \Delta_0(\That),  \\
		\label{eq:hgrad-dofs-grad}
		\ell_{S, j}(v) &= (\grad_S\psi_{S, j}, \grad_S v)_{S} \qquad & 
		&\Forall j \in 1:N_S^0, \ \Forall S 
		\in \bigcup_{l=1}^d \Delta_l(\That).
	\end{alignat}
\end{subequations}

\subsubsection{Degrees of freedom for $\Hcurl$} \label{sec:hcurl}

We discretize $H(\curl, \That)$ with $X^1(\That)$, 
which we recall is either $\Ned_p(\That)$ or $\NedTwo_{p}(\That)$,
so that the corresponding space $X^0(\That)$ is either 
$\CG_p(\That)$ or $\CG_{p+1}(\That)$, respectively. We define a basis for the 
dual of $X^1(\That)$ by 
using the eigenbasis $\{\psi_{S,j}\}$ constructed in \cref{eq:hgrad-fdm}
for the $\Hgrad$ bubbles $\mathring{X}^0(S)$. 
For each subsimplex $S\in \Delta_l(\That)$ with $l\in 2:d$,
we compute the eigenbasis $\{\Psi_{S,j}\}$ for $\mathring{X}^{1,\rmone}(S)$ 
satisfying
\begin{equation} \label{eq:hcurl-fdm}
	(\curl_S\Psi_{S, j}, \curl_S\Psi_{S, i})_{S} = \delta_{ij}, \quad
	(\Psi_{S, j}, \Psi_{S, i})_{S} = \lambda_j\delta_{ij},
\end{equation}
where we recall that $\curl_S$ is the tangential curl on $S$.
As mentioned above, $\{\curl_S\Psi_{S,j}\}$ is a basis for $\curl_S 
\mathring{X}^{1}(S)$.

The degrees of freedom for $X^1(\That)$ in \cref{eq:abstract-dofs-stiff} 
are tangential moments along edges,
moments of curl against curls
and moments against gradients
on each higher-dimensional subsimplex:
\begin{subequations}
	\label{eq:hcurl-dofs}
	\begin{alignat}{2}
		\label{eq:hcurl-dofs-low-order}
		\ell^{\rmone}_{E, 1}(v) &= (1, \bt\cdot v)_E  \qquad & &\Forall E \in 
		\Delta_1(\That),  \\
		\label{eq:hcurl-dofs-no-grad}
		\ell^{\rmone}_{S, j}(v) &= (\curl_S\Psi_{S, j}, \curl_S \Pi_S v)_{S}  
		\qquad & 
		&\Forall j \in 1:N_S^1, \ 
			\Forall S \in \bigcup_{l=2}^d \Delta_l(\That), \\
		\label{eq:hcurl-dofs-grad}
		\ell^{\rmtwo}_{S, j}(v) &= (\grad_S\psi_{S, j}, \Pi_S v)_{S}  \qquad 
		& 
		&\Forall j \in 1:N_S^0, \ 
			\Forall S \in \bigcup_{l=1}^d \Delta_l(\That),      
	\end{alignat}
\end{subequations}
where we recall that $\Pi_S$ is the projection operator onto the tangent 
space of $S$. In particular, the only difference between $\Ned_p(\That)$ and 
$\NedTwo_{p}(\That)$ is in \cref{eq:hcurl-dofs-grad}. For $\Ned_p(\That)$, 
the 
eigenfunctions $\psi_{S, j}$ defined in \cref{eq:hgrad-fdm} and appearing in 
\cref{eq:hcurl-dofs-grad} are degree $p$, while 
for $\NedTwo_{p}(\That)$, they are degree $p+1$.

\subsubsection{Degrees of freedom for $\Hdiv$} \label{sec:hdiv}

We discretize $H(\div, \That)$ with $X^2(\That)$, which we recall 
is either $\RT_p(\That)$ or $\BDM_{p}(\That)$. If $X^2(\That) = \RT_p(\That)$,
then we can choose $X^1(\That)$ as either $\Ned_p$ or $\NedTwo_p$, 
while if $X^2(\That) = \BDM_p(\That)$, we can choose  $X^1_p(\That)$ 
as either $\Ned_{p+1}$ or $\NedTwo_{p+1}$. The resulting basis will depend on 
this
particular choice, but the properties of the basis in the remainder of the 
manuscript are independent of this choice. For the implementation, we always 
choose $X^1(\That)$ to be the corresponding space of the second kind.

We define a basis for the dual of $X^2(\That)$ by 
using the eigenbasis $\{\Psi_{F,j}\}$ constructed in \cref{eq:hcurl-fdm}
for the $\Hcurl$ bubbles $\mathring{X}^{1, \rmone}(F)$. 
On the interior of $\That$,
we compute the eigenbasis $\{\Phi_{C,j}\}$ for 
$\mathring{X}^{2,\rmone}(\That)$ 
satisfying
\begin{equation} \label{eq:hdiv-fdm}
	(\div\Phi_{C, j}, \div\Phi_{C, i})_{\That} = \delta_{ij}, \quad
	(\Phi_{C, j}, \Phi_{C, i})_{\That} = \lambda_j\delta_{ij}.
\end{equation}
As mentioned above, $\{\div\Psi_{C,j}\}$ is a basis for $\div 
\mathring{X}^{2}(\That)$.

The degrees of freedom for $X^2(\That)$ in \cref{eq:abstract-dofs-stiff} 
are normal moments on faces,
moments of divergence against divergences
and moments against curls:
\begin{subequations} 
	\label{eq:hdiv-dofs}
	\begin{alignat}{2}
		\label{eq:hdiv-dofs-low-order}
		\ell^{\rmone}_{F, 1}(v) &= (1, \bn\cdot v)_F  \qquad & &\Forall F \in 
		\Delta_{d-1}(\That),  \\
		\label{eq:hdiv-dofs-face}
		\ell^{\rmtwo}_{F, j}(v) &= (\curl_F\Psi_{F, j}, \bn\cdot v)_{F} 
		\qquad & 
		&\Forall j \in 1:N_F^2, \ \Forall F \in \Delta_{d-1}(\That), \\
		\label{eq:hdiv-dofs-no-curl}
		\ell^{\rmone}_{C, j}(v) &= (\div\Phi_{C, j}, \div v)_{\That}  \qquad 
		& &\Forall j \in 1:N_C^2, \\
		\label{eq:hdiv-dofs-curl}
		\ell^{\rmtwo}_{C, j}(v) &= (\curl\Psi_{C, j}, v)_{\That} \qquad
		& &\Forall j \in 1:N_C^1. 
	\end{alignat}
\end{subequations}
Note that the only difference between $\RT_p(\That)$ or $\BDM_{p}(\That)$ are 
the eigenfunctions $\Psi_{F, j}$ and $\Psi_{C, j}$ defined in 
\cref{eq:hcurl-fdm} appearing in \cref{eq:hdiv-dofs-face,eq:hdiv-dofs-curl}. 
For $\RT_p(\That)$, these eigenfunctions are $\Ned_p$ or $\NedTwo_p$ 
functions (or their trace), while for $\BDM_p(\That)$, they are 
$\Ned_{p+1}(\That)$ or $\NedTwo_{p+1}(\That)$ functions (or their trace).

\begin{remark}
	The degrees of freedom in \cref{eq:hgrad-dofs,eq:hcurl-dofs} for 
	discretizations of $\Hgrad$ and $\Hcurl$ are well-defined for $d=2$. To 
	discretize $\Hdiv$ in 2D, which corresponds to $k=1$, we choose 
	$X^0(\That)$ 
	to be $\CG_{p}$ if $X^1(\That) = \RT_p$ or $\CG_{p+1}$ if $X^1(\That) = 
	\BDM_{p}$ and the operator $\curl_F$ in \cref{eq:hdiv-dofs-face} is the 
	tangential derivative operator and $\curl$ in \cref{eq:hdiv-dofs-curl} 
	corresponds to a $\pi/2$ rotation of the gradient operator.
\end{remark}

\section{Properties of the basis on the reference cell}

For $k \in 0:d-1$, let
\begin{align}
	\label{eq:reference-basis}
	\left\{ \phi_{S, j}^{k, \rmone}, \phi_{S, n}^{k, \rmtwo} : j = 1:N^k_S, \ 
	n = 
	1:N^{k-1}_S, \ S \in \Delta_{l}(\That), \ l \in k:d \right\}	
\end{align}
be the basis dual to the degrees of freedom \cref{eq:abstract-dofs-stiff}:
\begin{alignat*}{2}
	\ell^{k, \rmone}_{S, i}(\phi_{S, j}^{k, \rmone}) &= \delta_{ij}, \qquad & 
	\ell^{k, \rmtwo}_{S, m}(\phi_{S, j}^{k, \rmone}) &= 0, \\
	\ell^{k, \rmone}_{S, i}(\phi_{S, n}^{k, \rmtwo}) &= 0, \qquad & \ell^{k, 
		\rmtwo}_{S, m}(\phi_{S, n}^{k, \rmtwo}) &= \delta_{mn},
\end{alignat*}	
for all $i, j = 1:N^k_{S}$, $m, n = 1:N^{k-1}_S$, and $S \in 
\Delta_{l}(\That)$ 
with $l=k:d$.

\subsection{Relations to Whitney forms and eigenfunctions}

The first result shows that this basis contains the canonical basis for the 
Whitney forms.
\begin{lemma} \label{lem:basis-contains-whitney}
	For $k=0:d-1$ and $S \in \Delta_k(\That)$, $\phi_{S , 1}^{k, \rmone}$ 
	coincides with the 
	Whitney form associated to $S$.
\end{lemma}
\begin{proof}
	We prove the result for $d=3$; the case 
	$d=2$ is similar. Let $k \in 0:d-1$ and, for $S \in 
	\Delta_k(\That)$, let $w_{S}^k \in W^k(\That)$ denote the Whitney form 
	associated to $S$: $(1, \tr_{S'}^k w_S^k) = \delta_{S S'}$ for all $S' 
	\in 
	\Delta_k(\That)$.
	
	\noindent \textbf{Step 1: Type I degrees of freedom \cref{eq:type-I-dofs} 
		vanish. } For 
	$S' \in \cup_{l=k+1}^{d} \Delta_k(\That)$ and $j \in 1:N_{S'}^k$, there 
	holds 
	\begin{align*}
		\ell^{k, \rmone}_{S', j}(w_S^k) = (\d_{S'}^k \psi_{S', j}^{k}, 
		\d_{S'}^k 
		\tr^k_{S'} w_S^k)_{S'} &= \begin{cases}
			-(\psi_{S', j}^{0}, \div_{S'} \grad_{S'} \tr^k_{S'} w_S^0)_{S'} & 
			\text{if } k=0, \\
			(\psi_{S', j}^{1}, \curl_{S'} \curl_{S'} \tr^k_{S'} w_S^1 )& 
			\text{if } k = 1, \\
			-(\psi_{S', j}^{2}, \grad_{S'} \div_{S'} \tr^k_{S'} w_S^2)_{S'} & 
			\text{if } k = 2,
		\end{cases}
	\end{align*}
	where $\div_{S'}$ denotes surface divergence. Since $w_S^k \in 
	\P_1(\That)$, 
	$\ell^{k, \rmone}_{S', j}(w_S^k) = 0$. 
	
	\noindent \textbf{Step 2: Type II degrees of freedom 
	\cref{eq:type-II-dofs} 
		vanish. } 
	Let $S' \in \cup_{l=k}^{d} \Delta_l(\That)$ and $j \in 1:N_{S'}^{k-1}$. 
	Then, 
	there holds
	\begin{align*}
		\ell^{k, \rmtwo}_{S', j}(w_S^k) = (\d_{S'}^{k-1} \psi_{S', j}^{k-1}, 
		\tr^k_{S'} w_S^k)_{S'} &= \begin{cases}
			(\psi_{S', j}^{1}, \div_{S'} \tr^k_{S'} w_S^1)& 
			\text{if } k = 1, \\
			-(\psi_{S', j}^{2}, \curl_{S'} \tr^k_{S'} w_S^2)_{S'} & 
			\text{if } k = 2,
		\end{cases}.
	\end{align*}
	If $S' \in \Delta_{k}(\That)$, then $\tr_{S'}^k w_S^k \in \mathbb{R}$, 
	and so 
	$\ell^{k, \rmtwo}_{S', j}(w_S^k) = 0$. 
	
	Now suppose that $S' \in \Delta_l(\That)$ with $l \in k+1:d$. For $k=1$, 
	we 
	have
	\begin{align*}
		\tr_{S'} w_S^1 \in \begin{cases}
			\P_{0}(S')^{2} + \P_{0}(S) (\Pi_{S'} \bx)^{\perp} & \text{if } l 
			= 
			2, 
			\\
			\P_{0}(S')^{3} + \P_{0}(S)^3 \times \bx & \text{if } l = 3,
		\end{cases} 	
	\end{align*}
	and direct computation shows that $\tr_{S'} w_S^1 = 0$. E.g. for $l=3$, 
	we 
	have
	\begin{align*}
		\div (\vec{a} + \vec{b} \times \bx) = \div (\vec{b} \times \bx) = 
		-\vec{b} \cdot \curl \bx = 0 \qquad \Forall \vec{a}, \vec{b} \in 
		\mathbb{R}^3.
	\end{align*}
	Similarly, for $k=2$, we have
	\begin{align*}
		\curl(\vec{a} + b \bx) = b \curl \bx = 0 \qquad \Forall \vec{a} \in 
		\mathbb{R}^3, \ \Forall b \in \mathbb{R}.
	\end{align*}
	Thus, \cref{eq:type-II-dofs} vanish. By the unisolvence of 
	\cref{eq:abstract-dofs-stiff}, $\phi_{S , 1}^{k, \rmone} = w_{S}^k$.
\end{proof}

The next result shows that the traces of the type-$\rmone$ basis functions 
coincide with the associated eigenfunctions, while the traces of the 
type-$\rmtwo$ 
basis functions coincide with the exterior derivative of the $(k-1)$-form 
eigenfunctions.
\begin{lemma} \label{lem:trace-basis-eigenfunctions}
	For $k \in 0:d-1$, there holds
	\begin{subequations}
		\label{eq:trace-of-basis}
		\begin{alignat}{3}
			\label{eq:trace-of-basis-I}
			\tr^k_S \phi_{S, j}^{k, \rmone} &= \psi_{S, j}^{k} \qquad 
			& &\Forall j \in 1:N_S^k,\
			& &\Forall S \in \bigcup_{l=k+1}^d \Delta_l(\That), \\
			\label{eq:trace-of-basis-II}
			\tr^k_S \phi_{S, n}^{k, \rmtwo} &= \d_S^{k-1} \psi_{S, n}^{k-1} 
			\qquad 
			& &\Forall n \in 1:N_S^{k-1}, \ 
			& &\Forall S \in \bigcup_{l=k}^d \Delta_l(\That),
		\end{alignat}	
	\end{subequations}
	where $\{ \psi_{S, j}^{k} \}$ is defined in \cref{eq:ref-eval-problem-x}.
\end{lemma}
\begin{proof}
	Let $k \in 0:d-1$, $S \in \Delta_l(\That)$ for some $l \in k+1:d$, and $j 
	\in 
	1:N_S^k$. Arguing as in the proof of 
	\cref{lem:abstract-dofs-stiff-unisolvant}, we may show that $\tr_S^k 
	\phi_{S, 
		j}^{k, \rmone} \in \mathring{X}^{k, \rmone}(S)$, and so 
	\cref{eq:trace-of-basis-I} follows by the uniqueness of the 
	eigendecomposition of $\mathring{X}^k(S)$ \cref{eq:ref-eval-problem-x}. 
	\Cref{eq:trace-of-basis-II} follows from a similar argument.
\end{proof}

In fact, we can improve \cref{eq:trace-of-basis-II} and show that each 
type-$\rmtwo$ basis functions is the exterior derivative of a type-$\rmone$ 
basis 
function for $(k-1)$-forms.
\begin{lemma} \label{lem:d-preserves-basis-functions}
	For $k \in 1:d-1$, there holds
	\begin{align}
		\label{eq:d-preserves-basis-functions}
		\phi_{S, n}^{k, \rmtwo} = \d^{k-1} \phi_{S, n}^{k-1, \rmone} \qquad 
		\Forall n \in 1:N_S^{k-1}, \ 
		\Forall S \in \bigcup_{l=k}^{d} \Delta_l(\That).
	\end{align}
\end{lemma}
\begin{proof}
	Let $S' \in \Delta_k(\That)$. If $S \neq S'$, then $\tr^{k-1}_{S'} 
	\phi_{S, 
		n}^{k-1, \rmone} \equiv 0$ by the same arguments in the proof of 
		\cref{lem:abstract-dofs-stiff-unisolvant}. Moreover, if $S = S'$, 
		then 
	$\phi_{S, n}^{k-1, \rmone} \in \mathring{X}^{k-1, \rmone}$ and
	\cref{eq:ktrace-decomp} shows that $(1, \d_{S}^{k-1} \tr^{k-1}_{S}  
	\phi_{S, 
		n}^{k-1, \rmone})_S = 0$, and so \cref{eq:trace-ds-commute} gives
	\begin{align*}
		(1, \tr^k_{S} \d^{k-1} \phi_{S, n}^{k-1, \rmone})_{S} = (1, 
		\d_{S}^{k-1} 
		\tr^{k-1}_{S}  \phi_{S, n}^{k-1, \rmone})_{S} = 0.
	\end{align*}
	
	Now let $S' \in \cup_{l=k+1}^d \Delta_l(\That)$ and  $j \in 1:N^k_{S'}$. 
	\Cref{eq:trace-ds-commute} again gives
	\begin{align*}
		(\d_{S'}^k \psi_{S', j}^{k}, \d_{S'}^k \tr_{S'}^k  \d^{k-1} \phi_{S, 
			n}^{k-1, \rmone} )_{S'} = (\d_{S'}^k \psi_{S', j}^{k}, \d_{S'}^k 
		\d_{S'}^{k-1} \tr_{S'}^{k-1} \phi_{S, n}^{k-1, \rmone} )_{S'} = 0
	\end{align*}
	by the complex property $\d_{S'}^k \circ \d_{S'}^{k-1} = 0$. Thus, the 
	degrees of freedom \cref{eq:whitney-dofs,eq:type-I-dofs} vanish.
	
	Finally, let $S' \in \cup_{l=k}^d \Delta_l(\That)$ and $m \in 
	1:N^{k-1}_S$. 
	By virtue of $\phi_{S, n}^{k-1, \rmone}$ being a basis function of 
	$X^{k-1}_p(\That)$ dual to the degrees of freedom 
	$\mathcal{L}$, there holds
	\begin{align*}
		(\d_{S'}^{k-1} \psi_{S', j}^{k-1}, \tr^k_{S'}  \d^{k-1} \phi_{S, 
			n}^{k-1, \rmone})_{S'} = (\d_{S'}^{k-1} \psi_{S', j}^{k-1}, 
		\d_{S'}^{k-1} 
		\tr^{k-1}_{S'} \phi_{S, 
			n}^{k-1, \rmone})_{S'} = \delta_{S S'} \delta_{mn}.
	\end{align*}
	\Cref{eq:d-preserves-basis-functions} now follows from the unisolvence of 
	$\mathcal{L}$ on $X^{k}(\That)$.
\end{proof}

	\begin{remark}
		\label{rem:harmonic-extension-basis}
		We briefly outline how 
		\Cref{lem:basis-contains-whitney,lem:trace-basis-eigenfunctions,%
		lem:d-preserves-basis-functions}
		show that the basis \cref{eq:reference-basis} may be constructed 
		in a more typical $p$-version fashion as follows. For 
		$k \in 0:d-1$, $l \in k+1:d$, $S \in \Delta_l(\That)$, and 
		$j \in 1:N_S^k$, we extend 
		$\psi_{S, j}^k \in \mathring{X}^{k, \rmone}(S)$
		defined in \cref{eq:ref-eval-problem-x} to $X^k(\That)$
		by recursively taking harmonic extensions onto subsimplices 
		of one dimension higher. More precisely, 
		let $E_{l}^k \psi_{S, j}^k \in \prod_{\substack{S' \in \Delta_l(\That)}}
		\Tr_{S'}^k X^k(\That)$ be the zero extension of $\psi_{S, j}^k$:
		\begin{align*}
			E_{l}^k \psi_{S, j}^k|_{S'} = \psi_{S, j}^k \delta_{S S'}
			\quad \forall S' \in \Delta_l(\That).
		\end{align*}
		For $m \in l+1:d$, we recusively define
		$E_m^k  \psi_{S, j}^k \in \prod_{\substack{S' \in \Delta_m(\That)}}
		\Tr_{S'}^k X^k(\That)$
		on $S' \in \Delta_m(\That)$ as the 
		solution to the following variational problem:
		\begin{alignat*}{2}
			(\d_{S'}^k E_{m}^k \psi_{S, j}^k, \d_{S'}^k v)_{S'} &= 0 
			\qquad & & \forall v \in \mathring{X}^{k, \rmone}(S'), \\
			(E_m^k \psi_{S, j}^k, v)_{S'} &= 0 
			\qquad & & \forall v \in \d^{k-1}_{S'} 
				\mathring{X}^{k-1, \rmone}(S'), \\
			\Tr_F^k E_m^k \psi_{S, j}^k &= E_{m-1} \psi_{S, j}^k 
			\qquad & &\forall F \in \Delta_{m-1}(S').
		\end{alignat*}	
		One may readily verify that $E_m \psi_{S, j}$ is well-defined as the
		traces are constructed to be compatible. Then, the basis 
		\cref{eq:reference-basis} which is dual to the 
		degrees of freedom \cref{eq:abstract-dofs} 
		is the Whitney forms augmented with
		\begin{alignat*}{3}
			\phi_{S, j}^{k, \rmone} &= E_d^k \psi_{S, j}^k \qquad 
			& &\Forall j \in 1:N_S^k, \
			& &\Forall S \in \bigcup_{l=k+1}^d \Delta_l(\That), \\
			\phi_{S, n}^{k, \rmtwo} &= \d^{k-1} E_d^k \psi_{S, n}^{k-1} 
			\qquad 
			& &\Forall n \in 1:N_S^{k-1}, \ 
			& &\Forall S \in \bigcup_{l=k}^d \Delta_l(\That).
		\end{alignat*}
		Also note that this construction is hierarchical in dimension 
		in the sense that the trace of the 3D basis functions associated to $S$, 
		$ \{ \tr_{S} \phi_{S, j}^{k, \rmone}, \tr_{S} \phi_{S, n}^{k, \rmtwo}  
		\}$, 
		coincides with the 2D basis construction on $S$.
	\end{remark}

\subsection{Algebraic properties} \label{sec:reference-matrix-properties}

We now turn to properties of the resulting mass and stiffness matrices on the 
reference cell. In 
particular, let $\{ \phi_j^k \}$ denote the basis for $X^{k}(\That)$ in 
\cref{eq:reference-basis} and let the mass matrix $\hat{M} = (\hat{M}_{ij})$ 
and 
stiffness matrix $\hat{K} = (\hat{K}_{ij})$ be given by
\begin{align*}
	\hat{M}_{ij} = (\phi_i^k, \phi_j^k)_{\That} \quad \text{and} \quad 
	\hat{K}_{ij} = (\d^k \phi_i^k, \d^k \phi_j^k)_{\That}.
\end{align*}
We partition $\hat{M}$ and $\hat{K}$ into the contribution from the interior 
and 
interface basis functions:
\begin{align*}
	\hat{M} = \begin{bmatrix}
		\hat{M}_{\interior \interior} & \hat{M}_{\interior \interface} \\
		\hat{M}_{\interface \interior} & \hat{M}_{\interface \interface}
	\end{bmatrix} 
	\quad \text{and} \quad 
	\hat{K} = \begin{bmatrix}
		\hat{K}_{\interior \interior} & \hat{K}_{\interior \interface} \\
		\hat{K}_{\interface \interior} & \hat{K}_{\interface \interface}
	\end{bmatrix},
\end{align*}
where the subscript ``$\interior \interior$" denotes the contribution from 
the 
interior basis functions ($\phi_{S, j}^{k, \rmone}, \phi_{S, j}^{k, \rmtwo}$ 
with 
$S = \That$), ``$\interface\interface$" the contribution from the remaining 
(interface) functions, and ``$\interior \interface$" and ``$\interface 
\interior$" the interaction between the interior and interface functions. We 
further partition each block into contributions 
from the type-$\rmone$ and type-$\rmtwo$ basis functions, using the 
superscripts ``$\rmone,\rmone$", ``$\rmone,\rmtwo$", etc.

\begin{lemma} \label{lem:ref-element-matrix-properties}
	$\hat{M}$ and $\hat{K}$ have the following structure:
	\begin{align*}
		\hat{M} = \left[
		\begin{array}{cc|cc}
			\Lambda & & \bullet & \bullet \\
			& I & & \\ 
			\midrule
			\bullet & & \bullet & \bullet \\
			\bullet & & \bullet & \bullet
		\end{array} \right]
		\quad \text{and} \quad 
		\hat{K} = \left[
		\begin{array}{cc|cc}
			I & & & \\
			& & & \\
			\midrule
			& & \bullet & \\
			& & &
		\end{array} \right],
	\end{align*}
	where $\Lambda$ is a diagonal matrix whose entries are the eigenvalues  
	$\{\lambda_{T,j}\}_{j=1}^{N_T^k}$ \cref{eq:ref-eval-problem-x}.
\end{lemma}
\begin{proof}
	First note that the degrees of freedom \cref{eq:type-I-dofs} with 
	$S = \That$ mean that $(\d^k \phi, \d^k v)_{\That} = 0$ for any interface 
	basis function $\phi$ and any interior bubble $v \in 
	\mathring{X}^k(\That)$, and so $\hat{K}_{\interior \interface} = 0$ and 
	$\hat{K}_{\interface \interior} = 0$. 
	Similarly, \cref{eq:d-preserves-basis-functions} shows that $\d^{k} \phi 
	= 0$ for any type-$\rmtwo$ basis function, and so the only nonzero blocks 
	of $\hat{K}$ are $\hat{K}_{\interior \interior}^{\rmone, \rmone}$ and 
	$\hat{K}_{\interface,\interface}^{\rmone, \rmone}$. Thanks to 
	\cref{eq:trace-of-basis-I,eq:ref-eval-problem-x}, $\hat{K}_{\interior 
		\interior}^{\rmone, \rmone} = I$.
	
	Using \cref{eq:trace-of-basis,eq:ref-eval-problem-x} once again gives that
	$\hat{M}_{\interior \interior}$ is diagonal. More precisely, 
	\Cref{eq:trace-of-basis-I,eq:ref-eval-problem-x} mean that 
	$\hat{M}_{\interior \interior}^{\rmone, \rmone} = \Lambda$. Additionally, 
	\cref{eq:d-preserves-basis-functions} and the choice of degrees of 
	freedom \cref{eq:type-II-dofs} show that 
	(i) $\hat{M}_{\interior \interface}^{\rmtwo, \rmone} = 0$ and 
	$\hat{M}_{\interior \interface}^{\rmtwo, \rmtwo} = 0$ and
	(ii) the type$-\rmtwo$/type-$\rmtwo$ 
	block of the mass matrix of a $k$-form is the type$-\rmone$/type-$\rmone$ 
	block of the stiffness matrix of a $(k-1)$-form. Thus, 
	$\hat{M}_{\interior \interior}^{\rmtwo, \rmtwo} = I$.
\end{proof}

\subsection{Comparison to existing bases}

In \Cref{fig:conditioning} we compare the $\Hgrad$ interior basis functions 
of this
construction on the equilateral reference simplex to the hierarchical basis
proposed by Sherwin \& Karniadakis~\cite{sherwin1995new}. 
The basis proposed by this work has been constructed to produce diagonal
stiffness and mass interior submatrices $K_{\interior\interior}$ and 
$M_{\interior\interior}$.
Thus, diagonal scaling will ensure a condition number of one.
On the other hand, for the hierarchical basis these condition numbers grow 
with
$p$, indicating that a point-Jacobi preconditioner is not a suitable approach 
for this basis.  A
sparse-direct method would be required to eliminate the interiors.

We also consider the number of nonzeros per row of the stiffness matrix.  In
our construction, the number of nonzeros per row of the interior mass and
stiffness blocks is exactly one.  For the hierarchical basis
of~\cite{sherwin1995new}, Beuchler \& Pillwein~\cite{beuchler2007sparse} 
proved
that the number of nonzeros per row asymptotes to a constant of approximately 
8 in two
dimensions, and is proven to be bounded above by 105 in three dimensions. 

\begin{figure}
	\centering
	\input{./figures/plot_conditioning.tex}
	\caption{
		Condition number of the diagonally-scaled
		stiffness and mass interior submatrices 
		for $\Hgrad$ on an equilateral reference simplex. 
		The proposed basis in this work yields diagonal submatrices, 
		while the sparsity-optimized hierarchical
		bases studied in \cite{beuchler2007sparse}
		are less suitable for diagonal scaling.
	}
	\label{fig:conditioning}
\end{figure}

\section{Properties of the global basis functions}

Continuing with the notation from the previous sections, we let $X^k$ denote 
the 
global discrete space $X^{k}_{p, D}(\mesh)$.
The global basis for $X^k$ is constructed from the local basis on 
$X^k(\That)$ as 
in \cite{Rognes2010efficient}. In particular, for each element $T \in \mesh$, 
the 
bijective mapping $F_T : \That \to T$  is chosen to preserve fixed global 
orientations of 
subsimplices, and each global basis function supported on $T$ is of the form 
$\pullback_T^k(\hat{\phi})$, where $\{ \hat{\phi} \}$ is a basis for 
$X^k(\That)$ and 
$\pullback_T^k$ is defined in \cref{eq:pullback-3d-def}. For the remainder of 
the 
manuscript, we will assume that $F_T$ is chosen in this way. Boundary 
conditions 
are then enforced by removing all functions associated to subsimplices $S \in 
\Delta_{l}(\mesh)$, $l \in k:d-1$, lying on $\Gamma_D$.

\subsection{Key subspaces}

For $k \in 0:d-1$ and each subsimplex $S \in \Delta_l(\mesh)$, $l \in k:d$, 
let 
$\{ \phi_{S, j}^{k, \rmone} \}$ and $\{ \phi_{S, n}^{k, \rmtwo} \}$ denote 
the 
corresponding global basis functions associated to $S$. We collect a number 
of 
subspaces related to the basis functions. First, we separate the subspaces 
corresponding to the type-$\rmone$ basis functions into the Whitney forms, 
the 
high-order interface bubble functions, and the interior functions supported 
on 
one element: 
\begin{subequations}
	\label{eq:global-spaces-typeI}
	\begin{align}
		W^k &:= \spn\left\{ \phi_{S, 1}^{k, \rmone} : \ S \in 
		\Delta_k(\mesh), \ \tr_{\Gamma_D}^k \phi_{S, 1}^{k, \rmone} = 0 
		\right\}, 
		\\
		\mathring{X}^{k, \rmone}_{\interface} &:= \spn\left\{ \phi_{S, 
			j}^{k, 
			\rmone} : S \in \bigcup_{l=k+1}^{d-1} 
		\Delta_l(\mesh), \ j \in 1:N^k_{S}, \ \tr_{\Gamma_D}^k \phi_{S, 
			j}^{k, \rmone} = 0 	\right\}, \\
		\mathring{X}^{k, \rmone}_{\interior} &:= \spn\left\{ \phi_{T, 
			j}^{k, 
			\rmone} : \ \forall T \in \Delta_d(\mesh), \ j \in 1:N^k_{T} 
		\right\}, \\
		\label{eq:global-space-typeI}
		X^{k, \rmone} &:= W^k \oplus \mathring{X}^{k, \rmone}_{\interface} 
		\oplus 
		\mathring{X}^{k, \rmone}_{\interior},
	\end{align}
\end{subequations}
where spaces with the ``$\mathring{\phantom{n}}$" overset vanish on $S \in 
\Delta_k(\mesh)$.
Similarly, we separate the type-$\rmtwo$ subspaces:
\begin{subequations}
	\label{eq:global-spaces-typeII}
	\begin{align}
		X^{k, \rmtwo}_{\interface} &:= \spn\left\{ \phi_{S, 
			j}^{k, \rmtwo} : S \in \bigcup_{l=k}^{d-1} \Delta_l(\mesh), \ j 
			\in 
		1:N^{k-1}_{S}, \tr_{\Gamma_D}^k \phi_{S, j}^{k, \rmtwo} = 0 \right\}, 
		\\
		\mathring{X}^{k, \rmtwo}_{\interior} &:= \spn\left\{ \phi_{T, 
			j}^{k, \rmtwo} : \forall T \in \Delta_d(\mesh), \ j \in 
			1:N^{k-1}_{T} 
		\right\}, \\
		X^{k, \rmtwo} &:= X^{k, \rmtwo}_{\interface} \oplus 
		\mathring{X}^{k, \rmtwo}_{\interior}.
	\end{align}
\end{subequations}

By construction, the type-$\rmone$ and type-$\rmtwo$ spaces decompose $X^k$:
\begin{align} \label{eq:typeI-typeII-direct-sum}
	X^{k} = X^{k, \rmone} \oplus X^{k, \rmtwo}.
\end{align}
Moreover, since $\d^k \mathcal{F}_T^k(\phi) = \mathcal{F}_T^{k+1}(\d^k 
\phi)$, we 
also have the global analogue of \cref{lem:d-preserves-basis-functions} for 
$k \in 1:d-1$:
\begin{align}
	\label{eq:d-preserves-basis-functions-global}
	\phi_{S, n}^{k, \rmtwo} = \d^{k-1} \phi_{S, n}^{k-1, \rmone} \qquad 
	\Forall S \in \bigcup_{l=k}^{d} \Delta_l(\mesh), \ \Forall n \in 
	1:N_S^{k-1}.
\end{align}

\subsection{Decomposing the range and kernel of $\d^k$}

We now decompose the range of kernel of $\d^k$ using the subspaces defined 
above. 
The first result shows that the image of $\d^k$ acting on the high-order 
type-$\rmone$ spaces is disjoint from $\d^k W^k$.

\begin{lemma}
	\label{lem:image-type-one-bubble-disjoint-whitney}
	For $k \in 0:d-1$, if $u \in \mathring{X}^{k, \rmone}_{\interface} \oplus 
	\mathring{X}^{k, \rmone}_{\interior}$ satisfies $\d^k u \in \d^k 
	W^k$, then $u \equiv 0$. Consequently, $\d^k (\mathring{X}^{k, 
		\rmone}_{\interface} \oplus \mathring{X}^{k, \rmone}_{\interior}) 
		\cap \d^k 
	W^k 
	= \{0\}$.
\end{lemma}
\begin{proof}
	Let $v \in \mathring{X}^{k, \rmone}_{\interface} \oplus \mathring{X}^{k, 
		\rmone}_{\interior}$ and $w \in W^k$ satisfy $\d^k v = \d^k w$.
	Let $T \in \mesh$, $\hat{v} \in \mathring{X}^{k, 
		\rmone}_{\interface}(\That)$, and 
	$\hat{w} \in W^{k}(\That)$ be such that $v|_{T} = \pullback^k_T(\hat{v})$
	and $w|_{T} = \pullback^k_T(\hat{w})$, where $\mathring{X}^{k, 
		\rmone}_{\interface}(\That)$, etc. denotes the spans of the 
		corresponding 
	basis functions on the reference cell. 
	Since 
	\begin{align*}
		\pullback^{k+1}_T (\d^k \hat{w}) = \d^k w|_{T} = \d^k v|_{T} =  
		\pullback^{k+1}_T (\d^k \hat{v})
	\end{align*}
	and $\pullback^{k+1}_T$ is invertible, there holds $\d^k \hat{w} = \d^k 
	\hat{v}$.
	However, for $S \in \Delta_l(\That)$, $l \in k+1:d$, and $j \in 1:N_S^k$, 
	there holds
	\begin{align*}
		(\d^k_S, \psi_{S, j}, \d^k_S \tr_S^k \hat{v})_S = (\d^k_S, \psi_{S, 
		j}, 
		\tr_S^{k+1} \d^k \hat{v})_S &= (\d^k_S, \psi_{S, j}, \tr_S^{k+1} \d^k 
		\hat{w})_S \\
		&= (\d^k_S, \psi_{S, j}, \d^k_S \tr_S^k \hat{w})_S = 0,
	\end{align*}
	where we used \cref{eq:trace-ds-commute} for the first and 
	third equalities and \cref{lem:basis-contains-whitney} for the final 
	equality. Thus,	
	\begin{align*}
		\hat{v} &= \sum_{l=k+1}^{d} \sum_{S \in \Delta_l(\That)} 
		\sum_{j=1}^{N_S^k} (\d^k_S \psi_{S, j}, \d^k_S \tr_S^k \hat{v})_S 
		\phi_{S, 
			j}^{k, \rmone} = 0 \implies v \equiv 0.
	\end{align*}
\end{proof}

The next result shows the kernel of $\d^k$ consists of a low-order 
component from the Whitney forms plus the (high-order) type-$\rmtwo$ space, 
while 
the image of $\d^k$ also contains a low-order component from the Whitney 
forms 
and a component from the high-order type-$\rmone$ space.
\begin{lemma} \label{lem:kernel-image-dk-global-splitting}
	For $k \in 0:d-1$, there holds
	\begin{align}
		\label{eq:dk-kernel-decomp}
		\begin{aligned}
			\ker(\d^k; X^k) &= X^{k, \rmtwo}_{\interface} \oplus 
			\mathring{X}^{k, \rmtwo}_{\interior} \oplus 
			\ker(\d^k; W^k) \\
			&= \d^{k-1} \mathring{X}^{k-1, \rmone}_{\interface} 
			\oplus 
			\d^{k-1} \mathring{X}^{k-1, \rmone}_{\interior} \oplus \d^{k-1} 
			W^{k-1} \oplus \mathfrak{H}^{k},
		\end{aligned}
	\end{align}
	where $\mathfrak{H}^{k}$ denote the space of lowest-order harmonic forms:
	\begin{align} \label{eq:harmonic-forms}
		\mathfrak{H}^{k} := \{ \mathfrak{h} \in \ker(\d^k; W^k) : 
		(\mathfrak{h}, 
		\d^{k-1} w) = 0 \ \Forall w \in W^{k-1} \}.
	\end{align}
	Moreover, there holds
	\begin{align} \label{eq:dk-image-decomp}
		\d^k X^k &= \d^k \mathring{X}^{k, \rmone}_{\interface} \oplus \d^k 
		\mathring{X}^{k, \rmone}_{\interior} \oplus \d^k W^k.
	\end{align}
\end{lemma}
\begin{proof}
	\noindent \textbf{Step 1: \cref{eq:dk-kernel-decomp}. } Thanks to 
	\cref{eq:d-preserves-basis-functions-global}, $X^{k, 
		\rmtwo}_{\interface} = \d^{k-1} \mathring{X}^{k-1, 
		\rmone}_{\interface}$ 
	and $\mathring{X}^{k, \rmtwo}_{\interior} = \d^{k-1} \mathring{X}^{k-1, 
		\rmone}_{\interior}$, and so $X^{k, \rmtwo}_{\interface} 
	\oplus 
	\mathring{X}^{k, \rmtwo}_{\interior} \subset \ker(\d^k; X^k)$. 
	
	Now suppose that $u \in 
	\ker(\d^k; X^k)$. Then, we decompose 
	$u = u_{\rmone} + u_{\rmtwo} + w$, where 
	$u_{\rmone} \in \mathring{X}^{k, \rmone}_{\interface} \oplus 
	\mathring{X}^{k, \rmone}_{\interior}$, 
	$u_{\rmtwo} \in X^{k, \rmtwo}_{\interface} \oplus 
	\mathring{X}^{k, \rmtwo}_{\interior}$, and $w \in W^k$. 
	Then,
	\begin{align*}
		0 = \d^k u = \d^k u_{\rmone} + \d^k u_{\rmtwo} + 
		\d^k w = \d^k u_{\rmone} + \d^k w,
	\end{align*}
	and so $\d^k u_{\rmone} \in \d^k W^k$.  By 
	\cref{lem:image-type-one-bubble-disjoint-whitney}, $u_{\rmone} \equiv 0$, 
	and so $\d^k w \equiv 0$. \Cref{eq:dk-kernel-decomp} now follows.
	
	\noindent \textbf{Step 2: \cref{eq:dk-image-decomp}. } Thanks to Step 1 
	and 
	\cref{lem:image-type-one-bubble-disjoint-whitney}, we have $\d^k X^k = 
	\d^k 
	\mathring{X}^{k, \rmone}_{\interface} \oplus \d^k \mathring{X}^{k, 
		\rmone}_{\interior} \oplus \d^k W^k$, and so it suffices to show that 
		$\d^k 
	\mathring{X}^{k, \rmone}_{\interface} \cap \d^k \mathring{X}^{k, 
		\rmone}_{\interior} 
	= \{0\}$. 
	Let  $u_{\interface} \in \mathring{X}^{k, \rmone}_{\interface}$ and 
	$u_{\interior} \in   
	\mathring{X}^{k, \rmone}_{\interior}$ satisfy $\d^k (u_{\interface} -
	u_{\interior}) = 0$. By 
	\cref{lem:image-type-one-bubble-disjoint-whitney}, 
	$u_{\interface} = u_{\interior}$, and so $u_{\interface} = 
	u_{\interior} = 0$.
\end{proof}

\subsection{Matrix properties on a physical cell}

Let $T \in \mesh$ and consider the local matrix $A$ corresponding to the 
restriction of $a^k(\cdot,\cdot)$ \cref{eq:weak-form} to the cell $T$:
\begin{align*}
	A_{ij} = (\phi_i, \beta \phi_j)_T + (\d^k \phi_i, \alpha \d^k \phi_j)_T, 
	\qquad i, j \in 1:\dim X^k(T),
\end{align*}
where $\{ \phi_i \}$ consists of all global basis functions supported on $T$. 
As 
in \cref{sec:reference-matrix-properties}, we partition $A$ into the 
contribution 
from interior and interface basis functions:
\begin{align*}
	A = \begin{bmatrix}
		A_{\interior \interior} & A_{\interior \interface} \\
		A_{\interface \interior} & A_{\interface \interface}
	\end{bmatrix}.
\end{align*}
If $T$ is an equilateral simplex, then 
\cref{lem:ref-element-matrix-properties} 
shows that $A_{\interior \interior}$ is diagonal and that the $A_{\interior 
	\interface}$ and $A_{\interface \interior}$ blocks only contain a mass 
	term. This 
structure is lost if $T$ is not an equilateral simplex due to the mapping 
$\pullback_T^k$. However, both the coupling and the mass terms only introduce 
weak coupling in the sense that if we define
\begin{align*}
	P := \begin{bmatrix}
		\diag(A_{\interior \interior}) & 0 \\
		0 & A_{\interface \interface}
	\end{bmatrix},
\end{align*} 
then $P$ is spectrally equivalent to $A$, as the following 
result quantifies.
\begin{lemma} \label{lem:physical-element-decouple}
	For $k \in 0:d-1$ and $T \in \mesh$, there exists a constant $C > 0$ 
	depending only on shape regularity, $d$, $k$, and $\mathrm{diam}(\Omega)$ 
	such that
	\begin{align}
		\label{eq:physical-element-decouple}
		\kappa_2\left(P^{-1} A \right) \leq C \max\left\{ 1, 
		\frac{\beta}{\alpha} h_T^2\right\}
	\end{align}
	where $\kappa_2(\cdot)$ denotes the condition number with respect to the 
	2-norm.
\end{lemma}
\noindent The proof of \cref{lem:physical-element-decouple} is given in
\cref{sec:proof-physical-element-decouple}.

Provided that $\beta/\alpha$ is not too large, 
\cref{eq:physical-element-decouple} suggests that we may construct efficient 
preconditioners from existing ones by dropping the interior-interface 
coupling 
and by using a point-Jacobi (diagonal scaling) preconditioner on the interior 
basis functions. We combine this strategy with space decomposition 
preconditioners in the next section.

\section{Preconditioning} \label{sec:preconditioning}

In this section, we propose preconditioners for the Riesz maps  
\cref{eq:hgrad,eq:hcurl,eq:hdiv}. We describe our preconditioners in terms of 
space decompositions~\cite{ToWi06,xu92}. To capitalize on the weak coupling 
between 
interior and interface functions described in 
\cref{lem:physical-element-decouple}, we group the global interior and global 
interface functions together in their own subspaces:
\begin{align}
	\label{eq:global-space-interior-interface}
	X^{k}_{\interior} &:= \mathring{X}^{k, \rmone}_{\interior} 
	\oplus \mathring{X}^{k, \rmtwo}_{\interior} \quad \text{and} \quad 	
	X^k_{\interface} := W^k \oplus \mathring{X}^{k, \rmone}_{\interface} 
	\oplus 
	X^{k, \rmtwo}_{\interface}.
\end{align}
\noindent We also collect all $N^k_{\interior} := \dim 
X^k_{\interior}$ interior basis functions $\{ 
\phi_{\interior, \ell}^{k} :  \ell \in 1:N^k_{\interior} \}$.

To illustrate how the interior-interface decoupling strategy 
in \cref{lem:physical-element-decouple} may be applied to arbitrary space 
decompositions of $X^k$, we present the corresponding result for Schwarz 
methods with exact solvers. To this end, let $\{X_m^k\}_{m=1}^{M}$ be a 
subspace 
decomposition of $X^k$:
\begin{align} \label{eq:abstract-space-decomp}
	X^k = \sum_{m=1}^{M} X_m^k, \qquad \text{where} \quad  X_m \subseteq X^k 
	\quad 
	\forall m \in 1:M,
\end{align}
where each space is equipped with the $a^k(\cdot,\cdot)$ inner product. 
Then, we may decouple the interior and interface functions to arrive at the 
finer 
decomposition
\begin{align} \label{eq:abstract-space-decomp-split-all-interior}
	X^k = \sum_{m=1}^{M} X^k_m \cap X^k_{\interface} + X^k_{\interior},
\end{align}
where $X^k_m \cap X^k_{\interface}$ is equipped with $a^k(\cdot,\cdot)$ for 
$m 
\in 1:M$ and $X^k_{\interior}$ is equipped with a point-Jacobi solver
\begin{align*}
	a_{\interior}^k(u,v) := \sum_{\ell=1}^{N^k_{\interior}} a^k( u_{\ell}, 
	v_{\ell} ) 
	\qquad \forall u, v \in X^k_{\interior},
\end{align*}
where $u_{\interior} = \sum_{\ell=1}^{N^k_{\ell}} u_{\ell}$ 
with $u_{\ell} \in \spn\{\phi^k_{\ell} \}$ and similar notation for $v$. 
Note that the decomposition 
\cref{eq:abstract-space-decomp-split-all-interior} and associated bilinear 
forms 
is equivalent to the decomposition
\begin{align} \label{eq:abstract-space-decomp-split}
	X^k = \sum_{m=1}^{M} X^k_m \cap X^k_{\interface} + 
	\sum_{\ell=1}^{N^k_{\interior}} \spn\{ \phi^k_{\ell} \},
\end{align}
where each space is equipped with $a^k(\cdot,\cdot)$; we will use 
\cref{eq:abstract-space-decomp-split-all-interior} and 
\cref{eq:abstract-space-decomp-split} interchangeably. Then, 
standard arguments from the domain decomposition literature
(see e.g.~\cite[Remark on p.~178]{GO95} combined with 
\cref{cor:stable-decomp-interior-interface} and 
\cref{lem:stable-decomp-physical-interior} show
that the interior-interface splitting is, up to a constant 
independent of $h$, $p$, $\alpha$, and $\beta$, as stable as the original 
decomposition \cref{eq:abstract-space-decomp} provided that $\beta h^2 / 
\alpha 
\leq 1$:
\begin{theorem} \label{thm:generic-subspace-decomp-interface-interior-split}
	Suppose that the decomposition \cref{eq:abstract-space-decomp} is energy 
	stable: For every $u \in X^k$, 
	there exists $u_m \in X^k_m$, $m \in 1:M$ such that
	\begin{align} \label{eq:stable-decomposition-assumption}
		\sum_{m=1}^{M} u_m = u \quad \text{and} \quad \sum_{m=1}^{M} 
		a^k(u_m, u_m) 
		&\leq C_{\mathrm{stable}}(\alpha, \beta) a^k(u,u),
	\end{align}
	where $C_{\mathrm{stable}}(\alpha, \beta) \geq 1$ is independent of $u$.
	Then, the decomposition \cref{eq:abstract-space-decomp-split} is also 
	energy 
	stable: For every $u \in X^k$, there exists $u_m 
	\in X^k_m \cap X^k_{\interface}$, $m \in 1:M$, and $u_{\interior} \in 
	X^k_{\interior}$, such that
	\begin{subequations} \label{eq:abstract-space-decomp-split-stable}
		\begin{align}
			\sum_{m=1}^{M} u_m + u_{\interior} &= u, \\
			\sum_{m=1}^{M} 
			a^k(u_m, u_m) + a^k_{\interior}(u_{\interior}, u_{\interior})  
			&\leq C C_{\mathrm{stable}}(\alpha, \beta) \max\left\{ 1, 
			\frac{\beta}{\alpha} h^2 \right\} a(u, u),
		\end{align}
	\end{subequations}
	where $C$ depends only on shape regularity, $d$, and $k$. Moreover, 
	$a^k_{\interior}(\cdot,\cdot)$ is locally stable:
	\begin{align} \label{eq:abstract-space-decomp-interior-stability}
		a^k(v, v) \leq C a_{\interior}^k(v, v)  \qquad \forall v \in 
		X^k_{\interior}.
	\end{align}
\end{theorem}

One can apply \cref{thm:generic-subspace-decomp-interface-interior-split} to 
obtain condition number bounds on Schwarz methods associated to the subspace 
decomposition \cref{eq:abstract-space-decomp-split}, where the spaces are 
equipped with the solvers described above. For example, if the subspaces 
$\{X_m^k\}_{m=1}^{M}$ can be colored with $N_{\mathrm{color}}$ 
colors, then $\{X_m^k \cap X^k_{\interface}\}_{m=1}^{M}$ and 
$X^k_{\interior}$ can be be colored with $N_{\mathrm{color}}+1$ 
colors. Thanks to \cref{eq:abstract-space-decomp-split-stable}, 
\cref{eq:abstract-space-decomp-interior-stability}, and standard theory 
\cite[Theorem 2.7]{ToWi06}, the additive Schwarz preconditioner $P^{-1}$ 
associated to the subspace decomposition 
\cref{eq:abstract-space-decomp-split} satisfies
\begin{align*} 
	\kappa_2( P^{-1} A ) \leq C C_{\mathrm{stable}}(\alpha, \beta) 
	(N_{\mathrm{color}} + 1) 
	\max\left\{1, \frac{\beta}{\alpha} h^2 \right\},
\end{align*}
where $C$ depends only on shape regularity, $d$, and $k$. Similar results can 
be 
obtained for multiplicative and hybrid methods.

Note that a Schwarz method associated to the decomposition 
\cref{eq:abstract-space-decomp-split} is substantially cheaper than the 
Schwarz 
method associated to the original decomposition 
\cref{eq:abstract-space-decomp}, 
as the interior and interface functions are completely decoupled and a 
point-Jacobi preconditioner is employed on the interior basis functions. 
Decoupling the interior and interface functions is reminiscent of static 
condensation in which the $\mathcal{O}(p^3)$ interior degrees of freedom are 
exactly eliminated from the global system at a cost of $\mathcal{O}(p^9)$ 
operations, and problem is reformulated only on the interface degrees of 
freedom. 
The key advantage of the decomposition \cref{eq:abstract-space-decomp-split} 
is 
that we retain the decoupling of the interior and interface functions 
\textit{without actually performing static condensation}, and we only require 
a 
point-Jacobi preconditioner for the interior degrees of freedom. 

We now turn to 
specific space decompositions.

\subsection{Pavarino--Arnold--Falk--Winther decompositions}

To start, we consider the Pavarino--Arnold--Falk--Winther (PAFW) 
decomposition~\cite{pavarino93,arnold00,zaglmayr06,Schoberl08,brubeck24}. The
decomposition for $k$-forms is given by
\begin{equation} \label{eq:pafw}
	X^k = W^k
	+ \sum_{V\in \Delta_0(\mesh)} \left. X^k \right|_{\star V},
\end{equation}
which consists of two components. The first summand is the lowest-degree 
space and plays the role of the coarse grid, propagating coarse-scale 
information 
globally. The second summand restricts the fine space to small subspaces, 
each 
supported on the \emph{star} of a vertex $\star V$, where $V \in 
\Delta_0(\mesh)$ 
is a vertex of the mesh. The star operation is well-known in algebraic 
topology~\cite[\S 2]{munkres84}\cite{egorova09,pcpatch}; it gathers all 
entities 
that contain the given entity as a subentity. In particular, the star of a 
vertex 
returns all cells, faces, and edges that are adjacent to this vertex, as well 
as 
the vertex itself, and functions in $\left. X^k \right|_{\star V}$ have 
vanishing 
trace on $\partial (\star V)$. This second summand plays the role of a 
multigrid 
relaxation, or the subdomain solves in a domain decomposition solver. 

We denote the decomposition \cref{eq:pafw} by $\pafw{0}{X^k}$ to indicate 
that 
the 
space in the summation is $X^k$ restricted to the star of a 0-dimensional 
subsimplex. The decomposition satisfies the hypothesis of 
\cref{thm:generic-subspace-decomp-interface-interior-split} with 
$C_{\mathrm{stable}}$ independent of $\alpha$, $\beta$, 
$h$, and $p$ \cite{Falk2025local}. However, on a regular (Freudenthal 
\cite{Bey00}) mesh, an interior vertex star contains 14 edges, 36 faces, and  
24 cells, and the vertex-star patch problems become very expensive to solve, 
particularly at high order. Thus, we employ 
the subspace splitting in \cref{eq:abstract-space-decomp-split}:
\begin{align}\label{eq:approx-sc-pafw}
	X^k = W^k + \sum_{V\in \Delta_0(\mesh)} \left. X^k_{\interface} 
	\right|_{\star V} 
	+   
	\sum_{\ell=1}^{N^k_{\interior}} \spn\left\{ \phi_{\ell}^{k} 
	\right\},
\end{align}
which we refer to as $\pafw{0}{X^k_{\interface}}+\jacobi{X^k_{\interior}}$, 
where 
$\jacobi{X^k_{\interior}}$ indicates that a point-Jacobi preconditioner is 
employed on 
the space of interior functions.
\cref{thm:generic-subspace-decomp-interface-interior-split} shows that the 
decomposition \cref{eq:approx-sc-pafw} is also uniformly stable with respect 
to $h$ and $p$ 
and stable in $\alpha$ and $\beta$ provided that the ratio $\beta h^2/\alpha$ 
is bounded. 
Crucially, the dimension of the vertex-star problems scales like 
$\mathcal{O}(p^{d-1})$ for \cref{eq:approx-sc-pafw}, while it scales like  
$\mathcal{O}(p^d)$ for \cref{eq:pafw}. 
For example, at degree $p=4$ and $p=10$, 
the vertex star $\left. X^{0}\right|_{\star V}$ has dimension 175 and 3439, 
respectively, while $\left. X_{\interface}^{0}\right|_{\star V}$ has 
dimension 
151 and 1423, respectively.

\subsection{Hiptmair--Toselli space decompositions}

For $k > 0$, finer space decompositions are possible and more efficient. 
Arnold, 
Falk \& Winther \cite{arnold00} prove that one can employ stars over entities 
of 
dimension $k-1$ 
instead of dimension 0~, with the resulting patches being much 
smaller for $k > 1$. Furthermore, as observed by Hiptmair and 
Toselli~\cite{hiptmair97,hiptmair98,hiptmair00}, one can even use stars over 
entities of dimension $k$ by 
introducing a potential space. This approach splits $X^k$ as the sum of 
the 
image of the exterior derivative $\d^{k-1}$ on a potential space $Y^{k-1}$, 
and 
some other space $Z^k \subset X^k$:
\begin{equation} \label{eq:hiptmair-abstract}
	X^k = \d^{k-1} Y^{k-1} + Z^k.
\end{equation}
The challenge here is to choose $Y^{k-1}$ and $Z^k$ 
as small as possible;
e.g. excluding as much of the kernel of $\d^{k-1}$ as possible from $Y^{k-1}$ 
and 
as much of the kernel of $\d^k$ from $Z^k$. 

\subsubsection{Standard Hiptmair--Toselli decomposition}

The largest choice $Y^{k-1} = X^{k-1}$ and $Z^k = X^k$ motivates
the Hiptmair--Toselli space decomposition
\begin{equation} \label{eq:pavarino_hiptmair}
	X^k = W^k + \sum_{J \in \Delta_{k-1}(\mesh)} \left.\d^{k-1} 
	X^{k-1}\right|_{\star J} + \sum_{L \in \Delta_{k}(\mesh)} \left. 
	X^{k} \right|_{\star L},
\end{equation}
where introducing the potential space in \cref{eq:hiptmair-abstract} enables 
the 
splitting over stars of dimension $k-1$ and $k$, respectively. We denote this 
decomposition by $\ph{X^{k-1}}{X^k}$ to indicate that potential space is 
$X^{k-1}$ and that the $k$-dimensional stars are taken over $X^k$. As shown 
in \cite{hiptmair97,hiptmair98,hiptmair00}, the $\ph{X^{k-1}}{X^k}$ 
decomposition satisfies the 
hypothesis of \cref{thm:generic-subspace-decomp-interface-interior-split} 
with 
$C_{\mathrm{stable}}$ independent of $h$, $\alpha$, and $\beta$. As before, 
the 
$\ph{X^{k-1}}{X^k}$  
decomposition can be combined with 
\cref{thm:generic-subspace-decomp-interface-interior-split}, 
yielding the 
$\ph{X^{k-1}_{\interface}}{X^k_{\interface}}+\jacobi{X^k_{\interior}}$ 
decomposition.

Note that the $\ph{X^{k-1}}{X^k}$ decomposition involves an auxiliary
problem on the local potential space $X^{k-1}|_{\star J}$:
find $\psi \in X^{k-1}|_{\star J}$ such that
\begin{equation} \label{eq:local-potential-problem}
	a^k(\d^{k-1} \phi, \d^{k-1} \psi) = 
	(\d^{k-1} \phi, \beta\d^{k-1} \psi)_{\star J} = F(\d^{k-1} \phi) 
	\quad \Forall \phi \in \left.X^{k-1}\right|_{\star J},
\end{equation}
since $\d^k \circ \d^{k-1} = 0$. For $k > 1$, problem
\cref{eq:local-potential-problem} 
is singular since $\d^{k-2} X^{k-2} \subset \ker(\d^{k-1}; X^{k-1})$. To 
address this,
a pragmatic choice is to solve the regularized problem
\begin{equation} \label{eq:local-potential-problem-added-mass}
	(\phi, \epsilon \psi)_{\star J} +
	(\d^{k-1} \phi, \beta\d^{k-1} \psi)_{\star J} = F(\d^{k-1} \phi) 
	\quad \Forall \phi \in \left.X^{k-1}\right|_{\star J},
\end{equation}
for small $\epsilon$; in this work $\epsilon = 10^{-8} \beta$ is used. 
This regularization can be avoided by posing 
\cref{eq:local-potential-problem} on a reduced space, 
as discussed in the next section.

\subsubsection{Type-$\rmone$ based decomposition}

The finest possible choice for $Y^{k-1}$ and $Z^k$ in 
\cref{eq:hiptmair-abstract} is to ensure
\begin{align*}
	Y^{k-1} \cap \ker(\d^{k-1}; X^{k-1}) = \{0\} \quad \text{and} \quad Z^k 
	\cap 
	\d^{k-1} X^{k-1} = 
	\{0\}.
\end{align*}
However, this is not practical. As shown in \cref{eq:dk-kernel-decomp},
one component of $\ker(\d^{k-1}, X^{k-1})$ is the space of low-order harmonic 
forms $\mathfrak{H}^k$ \cref{eq:harmonic-forms} which depends on the topology 
of 
$\Omega$, and forms in $\mathfrak{H}^k$ have global support owing to the 
orthogonality condition in \cref{eq:harmonic-forms}.
Instead, a pragmatic intermediate choice motivated by 
\cref{lem:kernel-image-dk-global-splitting} is to employ
\begin{equation} \label{eq:hiptmair-reduction}
	Y^{k-1} = X^{k-1, \rmone} \quad \text{and} \quad Z^k = X^{k, \rmone}, 
\end{equation}
i.e.~to use only the span of the type-I basis functions. This choice excludes 
all 
type-$\rmtwo$ basis functions in the potential space $Y^{k-1}$ and in the 
space 
$Z^{k}$, and
this choice is simple to construct because our basis is already split in this 
fashion. Since $X^{0, \rmone} = X^0$, we only obtain a 
reduction compared to the largest choice of the potential space $Y^{k-1} = 
X^{k-1}$ if $k > 1$. We then arrive at the
$\ph{X^{k-1, \rmone}}{X^{k, \rmone}}$ decomposition:
\begin{equation} \label{eq:pavarino_hiptmair-novel}
	X^k = W^k + \sum_{J \in \Delta_{k-1}(\mesh)} \left.\d^{k-1} 
	X^{k-1, \rmone} \right|_{\star J} + \sum_{L \in \Delta_{k}(\mesh)} \left. 
	X^{k, \rmone} \right|_{\star L},
\end{equation}
which we believe to be novel.  

The local potential problems for the $\ph{X^{k-1, \rmone}}{X^{k, \rmone}}$ 
decomposition take the following form for $J \in \Delta_{k-1}(\mesh)$: Find 
$\psi 
\in X^{k-1, \rmone}|_{\star J}$ such that
\begin{equation} \label{eq:local-potential-problem-type-I}
	a^k(\d^{k-1} \phi, \d^{k-1} \psi) = 
	(\d^{k-1} \phi, \beta\d^{k-1} \psi)_{\star J} = F(\d^{k-1} \phi) 
	\quad \Forall \phi \in \left.X^{k-1, \rmone}\right|_{\star J}.
\end{equation}
Note that if $w \in W^{k-1}|_{\star J}$ satisfies $\d^{k-1} w = 0$, then $w 
\equiv 0$, and so problem \cref{eq:local-potential-problem-type-I} is 
invertible
thanks to 
\cref{lem:image-type-one-bubble-disjoint-whitney,%
	lem:kernel-image-dk-global-splitting}. In particular, problem 
\cref{eq:local-potential-problem-type-I} involves fewer degrees of freedom 
than 
problem \cref{eq:local-potential-problem-added-mass}. A similar approach was 
used 
by Zaglmayr 
\cite[\S 6.3]{zaglmayr06} to solve 
the curl-curl problem on the whole domain $\Omega$ by posing the problem only 
the type-$\rmone$ space $X^{1, \rmone}$.

\begin{remark}
	Since the type-$\rmone$ space $X^{k-1, \rmone}$ captures the image of the 
	exterior derivative $\d^{k-1}$ (i.e. $\d^{k-1} X^{k-1} = \d^{k-1} X^{k-1, 
		\rmone}$), the space decompositions $\ph{X^{k-1}}{X^{k}}$ and 
		$\ph{X^{k-1, 
			\rmone}}{X^{k}}$ are identical and thus satisfy the hypotheses of 
	\cref{thm:generic-subspace-decomp-interface-interior-split} with the same 
	stability constant. The only distinction we make in the 
	notation is the form of the potential problems 
	\cref{eq:local-potential-problem-added-mass} and 
	\cref{eq:local-potential-problem-type-I}. So, one could always use the 
	type-$\rmone$ space in the potential space and avoid solving the 
	singular problem \cref{eq:local-potential-problem}.
\end{remark}

Combining \cref{eq:pavarino_hiptmair-novel} with 
\cref{thm:generic-subspace-decomp-interface-interior-split} yields the 
$\ph{X^{k-1, \rmone}_{\interface}}{X^{k, 
		\rmone}_{\interface}}+\jacobi{X^k_{\interior}}$ decomposition:
\begin{equation} \label{eq:hcurl-decomp}
	X^k =  W^k 
	+ \sum_{J\in \Delta_{k-1}(\mesh)} \left. \d^{k-1} X^{k-1, 
	\rmone}_{\interface} 
	\right|_{\star J}
	+ \sum_{L\in \Delta_k(\mesh)} \left. X^{k, \rmone}_{\interface} 
	\right|_{\star 
		L} 
	+ 
	\sum_{\ell=1}^{N^k_{\interior}} \spn\left\{ \phi_{ 
		\ell}^{k} 
	\right\}.
\end{equation}

\noindent In \Cref{tab:Hcurl-patch-size} we record the dimension of the 
largest 
subspace on 
vertices and edges in the proposed space decompositions for $\Hcurl$. 
We again observe that applying 
\cref{thm:generic-subspace-decomp-interface-interior-split} reduces patch 
size, 
that $\ph{X^0}{X^1}$ and 
$\ph{X^0_{\interface}}{X^1_{\interface}}+\jacobi{X^1_{\interior}}$ involve 
much smaller patches than 
$\pafw{0}{X^1}$ and 
$\pafw{0}{X^1_{\interface}}+\jacobi{X^1_{\interior}}$, 
respectively, and that using only the type-$\rmone$ spaces as in 
$\ph{X^0}{X^{1, \rmone}}$ and 
$\ph{X^0_{\interface}}{X^{1,\rmone}_{\interface}}+\jacobi{X^1_{\interior}}$  
yields a further substantial 
improvement to the size of the edge problems solved.

\begin{table}[tbhp]
	\caption{
		Maximum subspace sizes for solving the $\Hcurl$ Riesz map with the 
		$\pafw{0}{X^1}$ \cref{eq:pafw}, 
		$\pafw{0}{X^1_{\interface}}+\jacobi{X^1_{\interior}}$ 
		\cref{eq:approx-sc-pafw}, 
		$\ph{X^0}{X^1}$ \cref{eq:pavarino_hiptmair}, 
		$\ph{X^0_{\interface}}{X^1_{\interface}}+\jacobi{X^1_{\interior}}$, 
		$\ph{X^0}{X^{1, \rmone}}$ \cref{eq:pavarino_hiptmair-novel}, and
		$\ph{X^0_{\interface}}{X^{1,\rmone}_{\interface}}+\jacobi{X^1_{\interior}}$
		\cref{eq:hcurl-decomp}
		decompositions on a regular mesh for 
		the \Nedelec{} space of the first kind. We distinguish between the 
		maximal vertex-star patch and maximal edge-star patch. The edge patch 
		includes 6 cell interiors, 6 faces, and one edge.
	}
		\label{tab:Hcurl-patch-size}
		\begin{center}
			\setlength{\tabcolsep}{4pt}
			\begin{tabular}{rlrr}
				\toprule
				decomposition & $X^1$ & max.~vertex dim & max.~edge dim\\
				\midrule
				$\pafw{0}{X^1}$ & $\Ned_{4}$ & 776 & -\\
				$\pafw{0}{X^1_{\interface}}+\jacobi{X^1_{\interior}}$ & 
				$\Ned_{4}$ & 
				488 & -\\
				$\ph{X^0}{X^1}$ & $\Ned_{4}$ & 175 & 148\\
				$\ph{X^0_{\interface}}{X^1_{\interface}}+\jacobi{X^1_{\interior}}$
				
				& $\Ned_{4}$ & 151 & 76\\
				$\ph{X^0}{X^{1, \rmone}}$ & $\Ned_{4}$ 
				& 175 & 121 \\
				PH($X_{\interface}^0$, 
				$X_{\interface}^{1,\rmone}$)$+$J($X^1_{\interior}$) & 
				$\Ned_{4}$ 
				& 151 & 55\\
				\midrule
				$\pafw{0}{X^1}$ & $\Ned_{10}$& 12020 & -\\
				$\pafw{0}{X^1_{\interface}}+\jacobi{X^1_{\interior}}$ & 
				$\Ned_{10}$ 
				& 3380 & -\\
				$\ph{X^0}{X^1}$ & $\Ned_{10}$ & 3439 & 2710\\
				$\ph{X^0_{\interface}}{X^1_{\interface}}+\jacobi{X^1_{\interior}}$
				
				& $\Ned_{10}$ & 1423 & 550\\
				$\ph{X^0}{X^{1, \rmone}}$ & $\Ned_{10}$ & 3439 & 1981 \\
				PH($X_{\interface}^0$, 
				$X_{\interface}^{1,\rmone}$)$+$J($X^1_{\interior}$) & 
				$\Ned_{10}$ 
				& 1423 & 325\\
				\bottomrule
			\end{tabular}
		\end{center}
\end{table}

\subsubsection{Stability of $\ph{X^{k-1,\rmone}}{X^{k,\rmone}}$}

We arrived at the $\ph{X^{k-1,\rmone}}{X^{k,\rmone}}$ decomposition by an 
algebraic decomposition of $\d^{k-1} X^{k-1}$ and $X^{k}$, which does not 
necessarily imply that $\ph{X^{k-1,\rmone}}{X^{k,\rmone}}$ is a stable 
decomposition of $X^k$. The next result shows that 
$\ph{X^{k-1,\rmone}}{X^{k,\rmone}}$ is indeed a stable decomposition of $X^k$ 
with respect to the mesh size $h$.
\begin{lemma} \label{lem:ph-typeI-typeII-stable}
	Let $k \in 0:d-1$ and $C_{\mathrm{HT}}$ denote the smallest constant 
	independent of $\alpha$ and $\beta$ such that the $\ph{X^{k-1}}{X^{k}}$ 
	decomposition \cref{eq:pavarino_hiptmair} satisfies the stability 
	hypothesis of 
	\cref{thm:generic-subspace-decomp-interface-interior-split}. Then, there 
	exists a constant $C_{\rmone/\rmtwo}$ independent of $h$, $\alpha$, and 
	$\beta$ such that	
	the $\ph{X^{k-1,\rmone}}{X^{k,\rmone}}$ decomposition 
	\cref{eq:pavarino_hiptmair-novel} satisfies the stability hypothesis of 
	\cref{thm:generic-subspace-decomp-interface-interior-split} with 
	$C_{\mathrm{stable}}(\alpha, \beta) \leq C_{\rmone/\rmtwo} 
	C_{\mathrm{HT}}$.
\end{lemma}
\noindent The proof of \cref{lem:ph-typeI-typeII-stable} is given in 
\cref{sec:proof-ph-typeI-typeII-stable}.

The constant $C_{\rmone,\rmtwo}$ in \cref{lem:ph-typeI-typeII-stable} may 
depend 
on the polynomial degree. Indeed, the numerical experiments below in 
\cref{sec:results} suggest that using the Hiptmair--Toselli  
$\ph{X^{k-1,\rmone}}{X^{k,\rmone}}$ 
decomposition to construct preconditioners does not provide $p$-independent 
iteration counts, which in turn suggests that $C_{\rmone,\rmtwo}$ depends on  
$p$. A more detailed discussion of $C_{\rmone,\rmtwo}$ is given in 
\cref{sec:typeI-typeII-constant-p}, where further experiments suggest that
$C_{\rmone,\rmtwo} \to \infty$ as $p \to \infty$.

\subsection{Decomposition for $\Hdiv$}
\label{sec:hdiv-decompositions}

For $d=3$ and $k=2$, the  $\ph{X^{1, \rmone}}{X^{2, \rmone}}$ decomposition 
\cref{eq:pavarino_hiptmair-novel} simplifies as 
\cref{eq:hdiv-dofs-no-curl,eq:global-space-typeI} 
show that $X^{2,\rmone}_{\interface} = W^{2}$. The last term in 
\cref{eq:hcurl-decomp} then becomes a splitting of $W^{2}$, 
which is not necessary as $W^2$ is already included in 
\cref{eq:pavarino_hiptmair-novel}. Combining this fact with 
\cref{thm:generic-subspace-decomp-interface-interior-split} leads to the 
following decomposition:
\begin{align} \label{eq:hdiv-ph-reduced}
	X^2 = W^2 + \sum_{E \in \Delta_{1}(\mesh)} \left. \curl 
	X_{\interface}^{1, \rmone} \right|_{\star E} + 
	\sum_{\ell=1}^{N^2_{\interior}} \spn\left\{ \phi_{\ell}^{2} \right\}.
\end{align}
Moreover, the interface potential problem on the edge star $E \in 
\Delta_1(\mesh)$
\begin{align} \label{eq:hdiv-potential-problem-edge}
	\phi \in \left. X_{\interface}^{1, \rmone} \right|_{\star E} : \qquad 
	(\curl 
	\phi, \beta 
	\curl \psi) = F(\curl \psi) \qquad \forall \psi \in \left. 
	X_{\interface}^{1, 
		\rmone} 
	\right|_{\star E}
\end{align}
and the full interface problem on the edge star
\begin{align} \label{eq:hdiv-problem-edge}
	u \in \left. X_{\interface}^{2} \right|_{\star E} : \qquad (u, \beta v) + 
	(\div u, \alpha 
	\div v) = F(v) \qquad \forall v \in \left. X_{\interface}^{2} 
	\right|_{\star 
		E}
\end{align}
have a comparable number of interface unknowns. In particular, 
\cref{eq:d-preserves-basis-functions-global,eq:dk-image-decomp} give
\begin{align*}
	\left. X_{\interface}^{2} \right|_{\star E} = \left. X_{\interface}^{2, 
		\rmone} \right|_{\star E} \oplus \left. X_{\interface}^{2, \rmtwo} 
	\right|_{\star E} = \left. W^2 \right|_{\star E} + \left. \curl 
	X_{\interface}^{1, \rmone} \right|_{\star E} = \left. W^2 \right|_{\star 
	E} 
	\oplus \left. \curl 
	\mathring{X}_{\interface}^{1, \rmone} \right|_{\star E},
\end{align*}
and so \cref{eq:hdiv-problem-edge} has only $|\Delta_2(\star E)|-1$ more 
unknowns
than \cref{eq:hdiv-potential-problem-edge}. We can thus avoid using the 
potential space $X^{1, \rmone}_{\interface}$ altogether without substantially 
increasing the patch size by replacing $\left. \curl X_{\interface}^{1, 
\rmone} 
\right|_{\star E}$ with $X_{\interface}^{2}$ to 
obtain
\begin{equation} \label{eq:hdiv-decomp}
	X^{2} = W^{2} +
	\sum_{E\in \Delta_{1}(\mesh)} \left. X^{2}_{\interface} \right|_{\star E}
	+ \ \sum_{\ell=1}^{N^{2}_{\interior}} \spn\{ \phi_{\ell}^{2} \}.
\end{equation}
We denote this decomposition by 
$\pafw{1}{X^2_{\interface}}+\jacobi{X^2_{\interior}}$, as it is of the form 
\cref{eq:approx-sc-pafw}, with 
the only difference being that edge stars are used instead of vertex stars. 
As 
shown in \cite{arnold00}, $\pafw{1}{X^2}$ satisfies the assumptions of 
\cref{thm:generic-subspace-decomp-interface-interior-split} with 
$C_{\mathrm{stable}}$ independent of $h$, $\alpha$, and $\beta$, and so 
\cref{eq:hdiv-decomp} is uniformly stable 
with respect to $h$, $\alpha$, and $\beta$ provided that $\beta h^2/\alpha$ 
is 
bounded. We remark that the $p$-stability of $\pafw{1}{X^2}$, and 
$\pafw{k-1}{X^k}$ more generally, remains an open problem for $k > 1$.

In \Cref{tab:Hdiv-patch-size} we record the dimension of the largest subspace 
on 
vertices, edges, and faces in the proposed space decompositions for $\Hdiv$. 
We 
observe that applying 
\cref{thm:generic-subspace-decomp-interface-interior-split} drastically 
reduces 
the patch size for each decomposition, while 
$\ph{X^{1,\rmone}_{\interface}}{X^{2,\rmone}_{\interface}} + 
\jacobi{X^2_{\interior}}$
and 	
$\pafw{1}{X^2_{\interface}}+\jacobi{X^2_{\interior}}$ have the smallest 
subproblems 
of about the same size.

\begin{table}[tbhp]
	{
		\caption{
			Maximum subspace sizes for solving the $\Hdiv$ Riesz map with the 
			$\pafw{0}{X^2}$ \cref{eq:pafw}, 
			$\pafw{0}{X^2_{\interface}}+\jacobi{X^2_{\interior}}$ 
			\cref{eq:approx-sc-pafw}, 
			$\ph{X^1}{X^2}$ \cref{eq:pavarino_hiptmair}, 
			$\ph{X^{1,\rmone}_{\interface}}{X^{2,\rmone}_{\interface}} + 
			\jacobi{X^2_{\interior}}$ 
			\cref{eq:hdiv-ph-reduced}, and 
			$\pafw{1}{X^2_{\interface}}+\jacobi{X^2_{\interior}}$ 
			\cref{eq:hdiv-decomp} 
			decompositions on a regular mesh for the Raviart--Thomas space. 
			The 
			face patch includes 2 
			cell interiors, and one face. 
		}
		\label{tab:Hdiv-patch-size}
		\begin{center}
			\setlength{\tabcolsep}{3.7pt}
			\begin{tabular}{rlrrr}
				\toprule
				decomposition & $X^2$ & max.~vertex dim & max.~edge dim & 
				max.~face dim\\
				\midrule
				$\pafw{0}{X^2}$ & $\RT_{4}$ & 1080 & - & -\\
				$\pafw{0}{X^2_{\interface}}+\jacobi{X^2_{\interior}}$ & 
				$\RT_{4}$ 
				& 
				360 & - & - \\
				$\ph{X^1}{X^2}$ & $\RT_{4}$ & - & 148 & 70 \\
				$\ph{X^{1,\rmone}_{\interface}}{X^{2,\rmone}_{\interface}} + 
				\jacobi{X^2_{\interior}}$
				& $\RT_{4}$ & - & 55 & - \\
				$\pafw{1}{X^2_{\interface}}+\jacobi{X^2_{\interior}}$  & 
				$\RT_{4}$ & - & 60 & -\\
				\midrule
				$\pafw{0}{X^2}$ & $\RT_{10}$& 13860 & - & -\\
				$\pafw{0}{X^2_{\interface}}+\jacobi{X^2_{\interior}}$ & 
				$\RT_{10}$ & 
				1980 & - & -\\
				$\ph{X^1}{X^2}$ & $\RT_{10}$ & - & 2710 & 1045 \\
				$\ph{X^{1,\rmone}_{\interface}}{X^{2,\rmone}_{\interface}} + 
				\jacobi{X^2_{\interior}}$
				& $\RT_{10}$ & - & 325 & - \\
				$\pafw{1}{X^2_{\interface}}+\jacobi{X^2_{\interior}}$  & 
				$\RT_{10}$ & - & 330 & - \\
				\bottomrule
			\end{tabular}
		\end{center}
	}
\end{table}

\subsection{Computational cost}

Our solver is globally matrix-free, but requires assembly of interface
submatrices over patches of mesh entities of dimension at most $d-1$.
In order to avoid computing and storing the matrix entries $a^k(\phi^k_j,
\phi^k_i)$, we instead employ a Krylov method only requiring the computation 
of
$a^k(u, \phi^k_i)$ for a given $u\in X^k$.  
Given a quadrature rule with $\bigo{p^d}$ quadrature points,
and without sum-factorization, this can
be done in $\bigo{p^{2d}}$ operations and $\bigo{p^d}$ storage per cell.
However, to form the basis there is an offline cost of $\bigo{p^{3d}}$ flops
and $\bigo{p^{2d}}$ storage to solve the eigenproblem
\cref{eq:ref-eval-problem-x-aug} and tabulate the basis at quadrature points.
Nevertheless, the tabulation is precomputed and stored, so the $\bigo{p^9}$ 
flops
are only incurred once on the reference cell, and the $\bigo{p^6}$ memory 
requirement does not scale with the number of cells.

The interface submatrices cost $\bigo{p^{3d-2}}$ flops to assemble and have 
$\bigo{p^{2(d-1)}}$ entries. The solution of each interface problem involves 
the 
factorization of the submatrix, which requires $\bigo{p^{3(d-1)}}$ flops and 
$\bigo{p^{2(d-1)}}$ storage per patch. 
At each Krylov iteration, solving the interface 
subproblems incurs $\bigo{p^{2(d-1)}}$ flops per patch. 
In three dimensions, the  costs of operator application and preconditioner 
setup
and application amount to $\bigo{p^6}$ flops and $\bigo{p^4}$ storage.

This is illustrated in \Cref{fig:complexities}, where we record the total
number of floating point operations, memory storage, and matrix nonzeros
required to solve the Riesz maps using the space decompositions with and
without the interior-interface decomposition.  
For $\Hgrad$ we use $\pafw{0}{X^0_\interface} + \jacobi{X^0_\interior}$,
for $\Hcurl$ we use $\ph{X^0_\interface}{X^{1,\rmone}_\interface} + 
\jacobi{X^2_\interior}$, 
and for $\Hdiv$ we use $\pafw{1}{X^2_\interface} + \jacobi{X^2_\interior}$.  
The analogous decompositions before interior-interface splitting
result in $\bigo{p^9}$ flops and $\bigo{p^6}$ storage and nonzeros.

\section{Numerical results} \label{sec:results}

All of the above bases and preconditioners have been implemented in Firedrake 
\cite{firedrake}, which we use to perform the numerical experiments. For 
brevity we only present results for spaces of the first kind.

\subsection{Riesz maps} \label{sec:results-riesz}

We first discretize each of the Riesz maps \cref{eq:weak-form} on an initial
Freudenthal mesh of $\Omega=(0, 1)^3$ with three elements along each edge 
(see \cref{fig:freudenthal} for the case of one element per edge) with pure 
Neumann boundary conditions ($\Gamma_N = \partial \Omega$). For
each Riesz map, we fix $\beta=1$, vary $\alpha \in \{10^{3}, 1, 10^{-3}\}$ and
$p \in 3:7$, and perform two levels of mesh refinement following
\cite{Bey00}. Note that the meshes resulting from refinement have a smaller
shape regularity constant compared to the original Freudenthal mesh, but the
shape regularity constant remains bounded away from zero as more refinements
are performed. We take the right hand side of the resulting linear system to 
be
a random vector and solve with the preconditioned conjugate gradient (PCG) 
method 
with a relative tolerance of $10^{-8}$.

\begin{figure}[htb]
	\centering
	\begin{subfigure}{0.45\linewidth}
		\centering
		\begin{tikzpicture}[3d view={120}{10}, scale=2]
			\coordinate (v0) at (-1, -1, -1) {};
			\coordinate (v1) at (1, -1, -1) {};
			\coordinate (v2) at (-1, 1, -1) {};
			\coordinate (v3) at (1, 1, -1) {};
			\coordinate (v4) at (-1, -1, 1) {};
			\coordinate (v5) at (1, -1, 1) {};
			\coordinate (v6) at (-1, 1, 1) {};
			\coordinate (v7) at (1, 1, 1) {};
			\coordinate (v1v5) at ($(v1)!0.5!(v5)$);
			\coordinate (v3v6) at ($(v3)!0.5!(v6)$);
			
			\draw (v0) -- (v1) (v0) -- (v3) (v0) -- (v7);
			\draw (v1) -- (v3) (v1) -- (v7);
			\draw (v3) -- (v7);
			
			\draw (v0) -- (v2) (v0) -- (v3) (v0) -- (v7);
			\draw (v2) -- (v3) (v2) -- (v7);
			\draw (v3) -- (v7);
			
			\draw (v0) -- (v1) (v0) -- (v5) (v0) -- (v7);
			\draw (v1) -- (v5) (v1) -- (v7);
			\draw (v5) -- (v7);
			
			\draw (v0) -- (v2) (v0) -- (v6) (v0) -- (v7);
			\draw (v2) -- (v6) (v2) -- (v7);
			\draw (v6) -- (v7);
			
			\draw (v0) -- (v4) (v0) -- (v5) (v0) -- (v7);
			\draw (v4) -- (v5) (v4) -- (v7);
			\draw (v5) -- (v7);
			
			\draw (v0) -- (v4) (v0) -- (v6) (v0) -- (v7);
			\draw (v4) -- (v6) (v4) -- (v7);
			\draw (v6) -- (v7);
		\end{tikzpicture}
		\caption{}
		\label{fig:freudenthal}
	\end{subfigure}
	\hfill
	\begin{subfigure}{0.45\linewidth}
		\centering
		\includegraphics[width=0.9\linewidth]{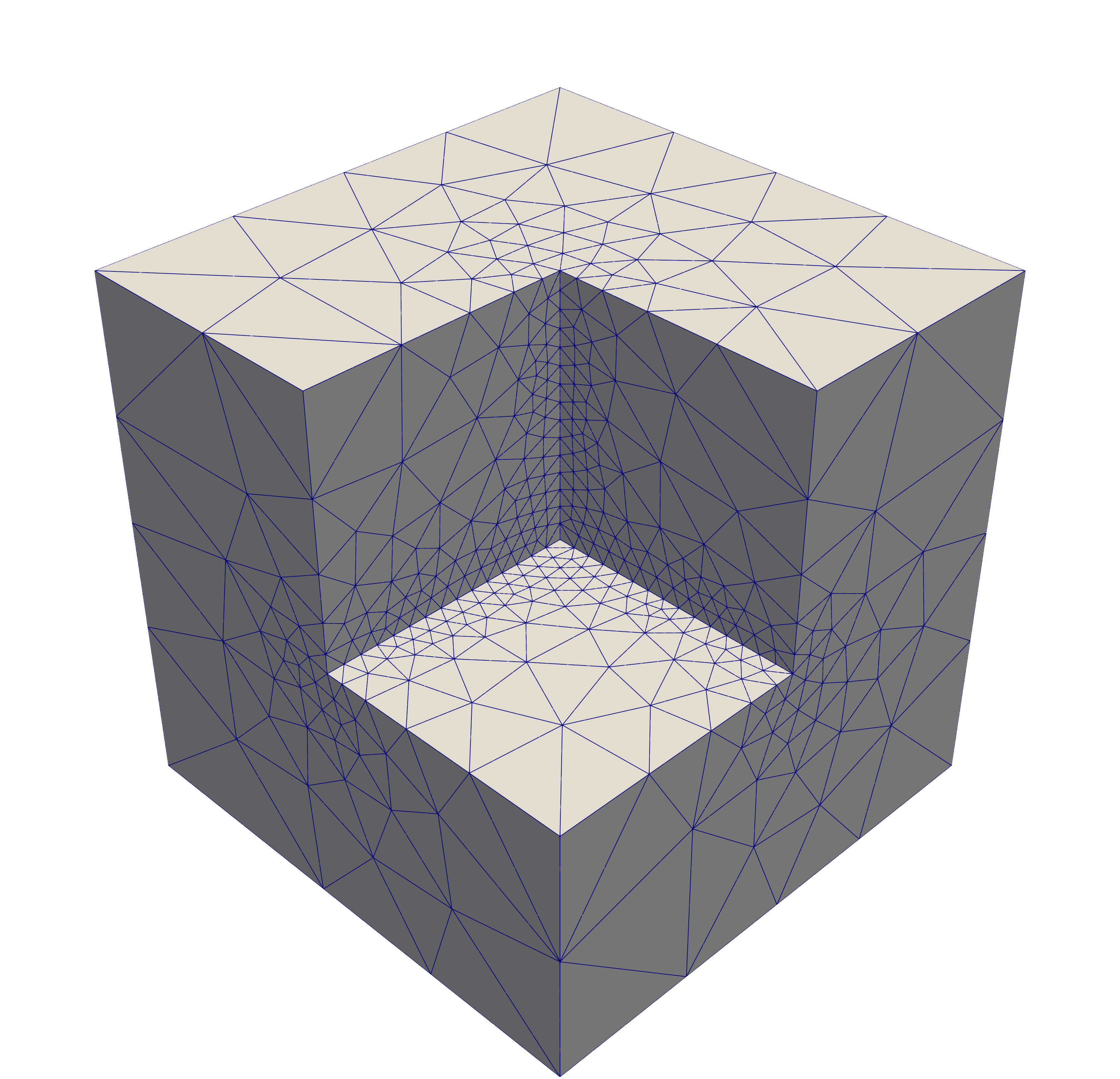}
		\caption{}
		\label{fig:fichera}	
	\end{subfigure}
	\caption{(a) Freudenthal subdivision of a cube and (b) mesh of the 
	Fichera 
		corner.}
\end{figure}

\subsubsection{$\Hgrad$ solvers}

We consider the subspace decomposition $\pafw{0}{X^0}$ \cref{eq:pafw} and the 
corresponding interior-interface splitting 
$\pafw{0}{X^0_{\interface}}+\jacobi{X^0_{\interior}}$ 
\cref{eq:approx-sc-pafw}. 
Here, and in the following examples, we use a symmetric hybrid Schwarz method 
preconditioner. In particular, the solvers from each group of spaces 
separated by 
a ``$+$" in the decompositions
\begin{align} \label{eq:hgrad-subspace-decomps}
	X^k = W^k + \sum_{V\in \Delta_0(\mesh)} \left. X^k \right|_{\star V}  
	\quad \text{and} \quad X^k = W^k + \sum_{V\in \Delta_0(\mesh)} \left. 
	X^k_{\interface} \right|_{\star V} 
	+   
	\sum_{\ell=1}^{N^k_{\interior}} \spn\left\{ \phi_{\ell}^{k} 
	\right\}
\end{align}
are weighted using estimated extremal eigenvalues and combined 
multiplicatively from right to left to right for symmetry. We equip $W^0$ 
with a geometric multigrid V-cycle with vertex patch
relaxation (equivalently point-Jacobi) and a direct Cholesky solve on the
coarsest mesh, while the remaining subspaces are equipped with exact solvers.

To more precisely describe the multiplicative splitting and the computation 
of the weights, we rewrite the hybrid methods in 
\cref{eq:hgrad-subspace-decomps} as a symmetric multiplicative Schwarz method 
with inexact solvers:  
\begin{align}
	\label{eq:generic-subspace-splitting-mult}
	X^k = \sum_{m=1}^{M} X^k_m,
\end{align}
where each space $X^k_m$ is equipped with a weight $\rho_m > 0$ and the 
inexact 
solver $\rho_m \tilde{a}_m(\cdot,\cdot)$, and the multiplicative sweep is 
done $X^k_1, \ldots, X^k_m, \ldots X^k_1$. For example, the second 
decomposition in \cref{eq:hgrad-subspace-decomps} satisfies
\begin{align*}
	X^0_1 = X^0_{\interior}, \quad X^0_2 = X^0_{\interface}, \quad \text{and} 
	\quad X^0_3 = W^0,
\end{align*}
where $\tilde{a}_1(\cdot,\cdot)$ is a point-Jacobi solver, 
$\tilde{a}_2(\cdot,\cdot)$ 
is the additive Schwarz method with decomposition $\sum_{V\in 
\Delta_0(\mesh)} \left. 
X^0_{\interface} \right|_{\star V}$ and exact solvers (additive patch solver 
on 0-stars), and $\tilde{a}_3(\cdot,\cdot)$ is a geometric multigrid V-cycle 
as above. The first decomposition in \cref{eq:hgrad-subspace-decomps} can be 
described similarly. The weight $\rho_m$ is taken to be 
\begin{equation}
	\rho_m = \frac{\tilde{\lambda}_{\min}+3\tilde{\lambda}_{\max}}{4}
\end{equation}
where $\tilde{\lambda}_{\min}$ and $\tilde{\lambda}_{\max}$ are estimates of 
the extreme generalized eigenvalues of
\begin{align*}
	u \in X^0_m, \lambda \in \mathbb{R} : \qquad a(u, v) = \lambda 
	\tilde{a}_m(u, v) \qquad \forall v \in X^0_m
\end{align*}  
computed from 10 conjugate gradient iterations with a randomized right hand 
side 
\cite{Lanczos51}. With this language, the 
$\pafw{0}{X^0_{\interface}}+\jacobi{X^0_{\interior}}$ solver is summarized in 
\cref{fig:solver-diagram-PAFW}.

\begin{figure}[htbp]
	\footnotesize
	\centering
	\begin{tikzpicture}[%
		every node/.style={draw=black, thick, anchor=west},
		grow via three points={one child at (-0.5,-0.7) and
			two children at (-0.7,-0.7) and (-0.7,-1.4)},
		edge from parent path={($(\tikzparentnode.south west)+(0.5, 0.0)$) |- 
			(\tikzchildnode.west)}]
		\node {KSP: conjugate gradients}
		child {node {PC: symmetric multiplicative Schwarz}
			child {node {$X^k_\interior$: point-Jacobi on cell interiors}}
			child {node {$X^k_\interface$: interface additive patches on 
			$l$-stars}}
			child {node {$W^k$: geometric multigrid V-cycle}
				child {node {Relaxation: patches on $l$-stars}}
				child {node {$h$-coarse: Cholesky}}
			}
		};
	\end{tikzpicture}
	\caption{Solver diagram for 
		$\pafw{l}{X^k_{\interface}}+\jacobi{X^k_{\interior}}$ used for 
		$\Hgrad$ and $\Hcurl$ 
		($l=0$) and $\Hdiv$ ($l=1$). 
	}
	\label{fig:solver-diagram-PAFW}
\end{figure}

The iteration counts for the two decompositions in 
\cref{eq:hgrad-subspace-decomps} with $p \in 5:7$ are displayed in 
\cref{tab:hgrad-unit-cube}. Here, and in the tables that follow, each column 
corresponds to a different value of $\alpha$, and the number of iterations 
for 
the solver with interior-interface splitting are displayed on the left 
followed 
by the iteration counts for the corresponding solver without the 
interior-interface splitting in brackets. We observe that across all values 
of 
$\alpha$ and $p$, the computationally cheaper 
$\pafw{0}{X^0_{\interface}}+\jacobi{X^0_{\interior}}$ solver performs within 
3 
iterations of the full $\pafw{0}{X^0_{\interface}}$ solver, and all of these 
iteration counts are robust with respect to the $\alpha$, $p$, and $h$.
\begin{table}[htbp]
	\centering
	\caption{PCG iteration counts for the $\Hgrad$ Riesz map, where $Y^0 = 
		X^0_{\interface} \ (Y^0 = X^0)$. \label{tab:hgrad-unit-cube}}
	\begin{filecontents*}{results/CG.3d.demkowicz.afw.combined.csv}
level,degree,dofs,beta,nonzeros,memory,Aalpha,Aiterations,Acondest,Atime,Balpha,Biterations,Bcondest,Btime,Calpha,Citerations,Ccondest,Ctime,Fnonzeros,Fmemory,FAiterations,FAcondest,FAtime,FBiterations,FBcondest,FBtime,FCiterations,FCcondest,FCtime
0,2,343,1.00E+00,8421,0,1.00E+03,9,1.6266687890913198,0.17890501022338867,1.00E+00,9,1.6187695132073057,0.0257723331451416,1.00E-03,10,1.6984267048469335,0.0273592472076416,8421,0,9,1.6266687890920877,0.19553661346435547,9,1.6187695132073052,0.028241395950317383,10,1.6984267048469304,0.03077244758605957
0,3,1000,1.00E+00,65544,0,1.00E+03,9,1.709906688644814,0.11077427864074707,1.00E+00,10,1.704300324352925,0.028568029403686523,1.00E-03,10,1.7279393513682078,0.02843642234802246,65544,0,9,1.7099066886448133,0.10849356651306152,10,1.7043003243529256,0.029605627059936523,10,1.7279393513682066,0.028661489486694336
0,4,2197,1.00E+00,308744,0,1.00E+03,10,1.7905078293973657,0.12526369094848633,1.00E+00,10,1.7873578178248481,0.036118268966674805,1.00E-03,10,1.9423579785492435,0.035558462142944336,268469,0,10,1.725093128180418,0.533437967300415,10,1.7158197179966708,0.059522390365600586,10,1.755236184576473,0.05786752700805664
0,5,4096,1.00E+00,1035239,0,1.00E+03,10,1.7906826345615448,0.159515380859375,1.00E+00,10,1.7917290599590183,0.05875253677368164,1.00E-03,9,1.904872512764637,0.056723833084106445,771007,0,10,1.7464990208209012,6.7566845417022705,10,1.74322778136158,0.08123946189880371,9,1.6909862707323038,0.07438921928405762
0,6,6859,1.00E+00,2815463,0,1.00E+03,10,1.8020377947948611,0.2405688762664795,1.00E+00,10,1.8021483286275546,0.11841058731079102,1.00E-03,9,1.9141080691465198,0.11038517951965332,1775833,0,10,1.7581934285999226,10.327884912490845,10,1.74997386735443,0.13228654861450195,10,1.721416298117014,0.13406133651733398
0,7,10648,1.00E+00,6566456,0,1.00E+03,10,1.790593021086871,0.4743003845214844,1.00E+00,10,1.7920018983876467,0.30901336669921875,1.00E-03,9,1.908047227413905,0.28374600410461426,3534596,0,10,1.749807792937544,13.494743585586548,10,1.7458257563754709,0.28357481956481934,10,1.7456664310850802,0.2801220417022705
1,2,2197,1.00E+00,49728,0,1.00E+03,11,2.3361457103453165,0.4822208881378174,1.00E+00,12,2.3195731451503905,0.06231069564819336,1.00E-03,8,1.4746878078138745,0.04496455192565918,49728,0,11,2.3361457103462007,0.4951822757720947,12,2.3195731451503914,0.061919450759887695,8,1.4746878078138745,0.045229196548461914
1,3,6859,1.00E+00,501506,0,1.00E+03,12,2.4052546586941483,0.14010024070739746,1.00E+00,12,2.383283872654464,0.07475757598876953,1.00E-03,9,1.7082501561874117,0.06150650978088379,501506,0,12,2.4052546586958043,0.13998770713806152,12,2.383283872654465,0.07421183586120605,9,1.7082501561874115,0.06046605110168457
1,4,15625,1.00E+00,2455820,0,1.00E+03,12,2.493867001381121,0.3498990535736084,1.00E+00,12,2.4731854950762298,0.12362408638000488,1.00E-03,10,1.9485540116848319,0.10727858543395996,2147228,0,12,2.4872940953016025,0.7441198825836182,12,2.4664519062114545,0.18968796730041504,10,1.7672077736053189,0.16531753540039062
1,5,29791,1.00E+00,8328023,0,1.00E+03,12,2.491245197657437,0.37167882919311523,1.00E+00,12,2.473575001022573,0.26645517349243164,1.00E-03,10,1.931139030796086,0.24522757530212402,6246055,0,12,2.5291875384869806,0.5225627422332764,12,2.5103851885681085,0.3352038860321045,10,1.7876731238856276,0.2966790199279785
1,6,50653,1.00E+00,22582081,0,1.00E+03,12,2.4666258025648085,0.8243286609649658,1.00E+00,12,2.4471732117327236,0.6700465679168701,1.00E-03,10,1.9238021782694796,0.6159937381744385,14485666,0,13,2.6115826996451363,0.9781503677368164,13,2.608346997473739,0.748387336730957,11,1.976381516648063,0.6752364635467529
1,7,79507,1.00E+00,52664501,0,1.00E+03,12,2.484360500255093,1.9906220436096191,1.00E+00,12,2.466676289005677,1.7219960689544678,1.00E-03,10,1.9194971030244756,1.606069803237915,29020373,0,14,2.915332834146153,2.113187313079834,14,2.9139949700676486,1.8244290351867676,12,2.316799023265346,1.639986515045166
2,2,15625,1.00E+00,379264,0,1.00E+03,14,2.921063240903081,0.35381078720092773,1.00E+00,14,2.9071081103968863,0.12309050559997559,1.00E-03,10,1.7023297431757392,0.08922266960144043,379264,0,14,2.921063240904307,0.35013580322265625,14,2.9071081103968854,0.10355496406555176,10,1.7023297431757387,0.0814828872680664
2,3,50653,1.00E+00,3988958,0,1.00E+03,14,3.0747821203603314,0.46931886672973633,1.00E+00,14,3.0570276572487898,0.23111295700073242,1.00E-03,10,1.774289697614411,0.17026948928833008,3988958,0,14,3.074782120358524,0.4571084976196289,14,3.0570276572487876,0.20327448844909668,10,1.774289697614411,0.16571903228759766
2,4,117649,1.00E+00,19740324,0,1.00E+03,13,3.1388831198434706,0.8714256286621094,1.00E+00,14,3.1567831350791304,0.5629255771636963,1.00E-03,11,1.9167798626287533,0.49560093879699707,17292630,0,13,3.1229733537595585,1.2159292697906494,14,3.133494368159892,0.7593140602111816,10,1.8006210867465204,0.6104011535644531
2,5,226981,1.00E+00,67074389,0,1.00E+03,14,3.1968446232045444,2.093797445297241,1.00E+00,14,3.185783600345847,1.5527677536010742,1.00E-03,10,1.8987062259507792,1.310659408569336,50570373,0,14,3.189520900770564,2.401907205581665,14,3.178357909337816,1.7825376987457275,11,1.9856911700332662,1.5382356643676758
2,6,389017,1.00E+00,181895599,0,1.00E+03,13,3.165904454081449,4.8268046379089355,1.00E+00,13,3.1540621601200165,3.8401753902435303,1.00E-03,10,1.909505987494267,3.395925998687744,117699379,0,14,3.1611946606834156,5.023573160171509,14,3.147541847775467,4.156895637512207,12,2.314092344763894,3.8076977729797363
2,7,614125,1.00E+00,423899553,0,1.00E+03,13,3.118664856174823,12.143709421157837,1.00E+00,13,3.1071467543330336,10.190461158752441,1.00E-03,10,1.9048881148140027,9.241948127746582,236332056,0,14,2.992534374698441,11.434467554092407,14,2.985689782517107,10.074912548065186,13,2.695637746326045,9.661949396133423
\end{filecontents*}
\csvreader[
	head to column names, head to column names prefix=MY,
	tabular=rcccc,
	table head=\toprule 
	& & \multicolumn{3}{c}{$\pafw{0}{Y^0}$} \\
	{$l$} & {$p$} & {$\alpha=10^3$} & {$\alpha=1$} & {$\alpha=10^{-3}$}\\
	\midrule,
	late after line=\ifthenelse{\equal{\MYdegree}{5}}{\\\midrule}{\\},
	filter ifthen=\equal{\MYdegree}{5} \or\equal{\MYdegree}{6} 
	\or\equal{\MYdegree}{7},
	table foot=\bottomrule
	]{results/CG.3d.demkowicz.afw.combined.csv}{}
	{\MYlevel{} & \MYdegree{} & \MYFAiterations{} (\MYAiterations{})  & 
	\MYFBiterations{} (\MYBiterations{}) & \MYFCiterations{} 
	(\MYCiterations{})}
\end{table}

\subsubsection{$\Hcurl$ solvers}

We first consider the subspace decompositions based on vertex patches 
$\pafw{0}{X^1}$ and $\pafw{0}{X^1_{\interface}}+\jacobi{X^1_{\interior}}$ 
using the same symmetric hybrid Schwarz solver described above and in 
\cref{fig:solver-diagram-PAFW}. Note that the vertex patch relaxation of the 
geometric multigrid V-cycle preconditioner on $W^1$ is no longer equivalent 
to a point-Jacobi relaxation. The iteration counts for these two solvers with 
$p \in \{3,5,7\}$ are displayed in the leftmost section of 
\cref{tab:hcurl-unit-cube}. For $\alpha \geq 1$,  $\pafw{0}{X^1}$ and 
$\pafw{0}{X^1_{\interface}}+\jacobi{X^1_{\interior}}$ are within 7 iterations 
of one another, while for $\alpha = 10^{-3}$, 
$\pafw{0}{X^1_{\interface}}+\jacobi{X^1_{\interior}}$ requires about double 
the number of iterations as $\pafw{0}{X^1}$. The robustness of the 
interior-interface splitting for $\alpha \geq 1$ and its degradation for 
$\alpha = 10^{-3}$ is consistent with the stability of the decomposition in 
\cref{thm:generic-subspace-decomp-interface-interior-split}. 

We now consider two Hiptmair--Toselli variants $\ph{X^0}{X^1}$ and 
$\ph{X^0}{X^{1, \rmone}}$ along with their interior-interface splitting 
counterparts 
$\ph{X^0_{\interface}}{X^1_{\interface}}+\jacobi{X^1_{\interior}}$ and 
$\ph{X^0_{\interface}}{X^{1,\rmone}_{\interface}}+\jacobi{X^1_{\interior}}$. 
We use an analogous hybrid Schwarz method preconditioner as for 
$\pafw{0}{X^1}$. For example, the 
$\ph{X^0_{\interface}}{X^{1,\rmone}_{\interface}}+\jacobi{X^1_{\interior}}$ 
solver can be expressed as a symmetric multiplicative Schwarz method as in 
\cref{eq:generic-subspace-splitting-mult} with weighted inexact solvers by 
choosing
\begin{align*}
	X_1^1 = X^1_{\interior}, \quad X^1_2 = X^{1}_{\interface}, \quad 
	\text{and} \quad X^1_3 = W^1.
\end{align*}
As above, $\tilde{a}_1(\cdot,\cdot)$ is point-Jacobi, while 
$\tilde{a}_2(\cdot,\cdot)$ is the additive Schwarz method associated to the 
(Hiptmair--Toselli) decomposition 
\begin{align} \label{eq:hiptmair-smoothing}
	\sum_{V\in \Delta_{0}(\mesh)} \left. \grad X^{0}_{\interface} 
	\right|_{\star V}
	+ \sum_{E\in \Delta_1(\mesh)} \left. X^{1, \rmone}_{\interface} 
	\right|_{\star 
		E} 
\end{align}
with exact solvers. The inexact solver $\tilde{a}_3(\cdot,\cdot)$ is a 
geometric multigrid V-cycle preconditioner with Hiptmair--Jacobi smoothing 
meaning that the smoother is the additive Schwarz method associated to the 
decomposition \cref{eq:hiptmair-smoothing} at lowest order ($X^0_{\interface} 
= W^0$ and $X^{1, \rmone}_{\interface} =  W^1$) with exact solvers. The 
solver is summarized
in \cref{fig:solver-diagram-PH}. The other three hybrid Hiptmair--Toselli 
solvers are defined analogously.

\begin{figure}[htbp]
	\footnotesize
	\centering
	\begin{tikzpicture}[%
		every node/.style={draw=black, thick, anchor=west},
		grow via three points={one child at (-0.5,-0.7) and
			two children at (-0.7,-0.7) and (-0.7,-1.4)},
		edge from parent path={($(\tikzparentnode.south west)+(0.5, 0.0)$) |- 
			(\tikzchildnode.west)}]
		\node {KSP: conjugate gradients}
		child {node {PC: symmetric multiplicative Schwarz}
			child {node {$X^k_\interior$: point-Jacobi on cell interiors}}
			child {node {$X^k_\interface$: additive Hiptmair--Toselli 
					decomposition}
				child {node {$\d^{k-1} Y^{k-1}_\interface$: interface 
				additive 
						patches on $(k-1)$-stars}}
				child {node {$Z^k_\interface$: interface additive patches on 
						$k$-stars}}
			}
			child[missing]{}
			child[missing]{}
			child {node {$W^k$: geometric multigrid V-cycle}
				child {node {Relaxation: Hiptmair--Jacobi}}
				child {node {$h$-coarse: Cholesky}}
			}
		};
	\end{tikzpicture}
	\caption{Solver diagram for 
	$\ph{Y_{\interface}^{k-1}}{Z_{\interface}^{k,\rmone}}+\jacobi{X^k_{\interior}}$
	 used for $\Hcurl$ and $\Hdiv$.}
	\label{fig:solver-diagram-PH}
\end{figure}

The iteration counts for the four Hiptmair--Toselli decompositions are 
displayed 
in the two rightmost sections of \cref{tab:hcurl-unit-cube}. For each fixed 
value 
of $\alpha$, the $\ph{X^0}{X^1}$ and 
$\ph{X^0_{\interface}}{X^1_{\interface}}+\jacobi{X^1_{\interior}}$ solvers 
give 
nearly identical iteration counts which are robust in $h$ and $p$. Moreover, 
the 
iteration counts are robust across values of $\alpha$. 
For $\alpha \geq 1$, the $\ph{X^0}{X^{1, \rmone}}$ and 
$\ph{X^0_{\interface}}{X^{1, \rmone}_{\interface}}+\jacobi{X^1_{\interior}}$ 
solvers also give nearly identical iteration counts that are a few iterations 
less than the $\ph{X^0}{X^1}$ and 
$\ph{X^0_{\interface}}{X^1_{\interface}}+\jacobi{X^1_{\interior}}$ solvers. 
For $\alpha = 10^{-3}$, $\ph{X^0_{\interface}}{X^{1, 
\rmone}_{\interface}}+\jacobi{X^1_{\interior}}$ slightly outperforms 
$\ph{X^0}{X^{1, \rmone}}$, but both require substantially more iterations 
than for $\alpha \geq 1$ and the number of iterations decreases as the mesh 
is refined. Overall, for $\alpha \geq 1$, the finest splitting 
$\ph{X^0_{\interface}}{X^{1, \rmone}_{\interface}}+\jacobi{X^1_{\interior}}$ 
is within 6 iterations of $\pafw{0}{X^1}$, which is the most computationally 
expensive solver and requires the smallest iteration counts, while the 
slightly coarser splitting 
$\ph{X^0_{\interface}}{X^{1}_{\interface}}+\jacobi{X^1_{\interior}}$ requires 
about 3 times the number of iterations as $\pafw{0}{X^1}$ for $\alpha = 
10^{-3}$.

\begin{table}[htbp]
	\centering
	\caption{PCG iteration counts for the $\Hcurl$ Riesz map, where $Y^k = 
		X^k_{\interface} \ (Y^k = X^k)$.\label{tab:hcurl-unit-cube}}
	\begin{filecontents*}{results/N1curl.3d.demkowicz.afw.combined.csv}
level,degree,dofs,beta,nonzeros,memory,Aalpha,Aiterations,Acondest,Atime,Balpha,Biterations,Bcondest,Btime,Calpha,Citerations,Ccondest,Ctime,Fnonzeros,Fmemory,FAiterations,FAcondest,FAtime,FBiterations,FBcondest,FBtime,FCiterations,FCcondest,FCtime
0,2,1314,1.00E+00,125225,0,1.00E+03,11,1.9201970872604888,0.35985636711120605,1.00E+00,11,1.9334874803253823,0.04781699180603027,1.00E-03,9,1.492031762894058,0.04121255874633789,125225,0,11,1.9201970872604408,0.3346235752105713,11,1.933487480325381,0.04498624801635742,9,1.4920317628940574,0.03922271728515625
0,3,3591,1.00E+00,801615,0,1.00E+03,10,1.723819278492931,0.45481324195861816,1.00E+00,10,1.8003948391943785,0.06987690925598145,1.00E-03,9,1.6207248579561386,0.06605768203735352,631047,0,10,1.6980341640504282,0.9510726928710938,10,1.777547288449588,0.09966635704040527,15,3.1251162036672624,0.12481498718261719
0,4,7596,1.00E+00,3357445,0,1.00E+03,10,1.8148040680702655,0.36974358558654785,1.00E+00,10,1.8146751736521163,0.16056180000305176,1.00E-03,10,1.7125088672455608,0.16097021102905273,2058981,0,10,1.8275091988352459,0.786308765411377,11,2.027647594055076,0.2144334316253662,16,3.749213594507752,0.2763550281524658
0,5,13815,1.00E+00,10655336,0,1.00E+03,10,1.8118299630899823,0.808330774307251,1.00E+00,10,1.8134639329335087,0.5325918197631836,1.00E-03,10,1.7617197038828625,0.5721931457519531,5129341,0,11,2.1030691206912238,1.5353074073791504,11,2.149728182156609,0.5522286891937256,17,4.142007428317964,0.7698404788970947
0,6,22734,1.00E+00,27712353,0,1.00E+03,10,1.8234717436076113,2.468743324279785,1.00E+00,10,1.8256824629194457,1.9920263290405273,1.00E-03,10,1.8040904319582622,2.000192403793335,10838829,0,11,2.170586030262741,2.594256639480591,12,2.2383799976306835,1.334524154663086,18,4.330925217014819,1.739208459854126
0,7,34839,1.00E+00,62928358,0,1.00E+03,10,1.8152539843088045,7.9160754680633545,1.00E+00,10,1.8183036987949224,6.22390079498291,1.00E-03,10,1.8255044830745337,6.137047290802002,20322900,0,11,2.0963568917192026,6.2636895179748535,12,2.305778318719884,4.249034643173218,18,4.604591807133687,5.278225660324097
1,2,9324,1.00E+00,924349,0,1.00E+03,14,2.8960489649234593,1.3846805095672607,1.00E+00,14,2.9146891123866396,0.20676207542419434,1.00E-03,10,1.7975762559995427,0.15988850593566895,924349,0,14,2.896048964923416,1.236490249633789,14,2.9146891123866365,0.20600318908691406,10,1.797576255999542,0.16331911087036133
1,3,26298,1.00E+00,6446370,0,1.00E+03,13,2.4995692638310847,0.5756573677062988,1.00E+00,13,2.612561646423351,0.366396427154541,1.00E-03,10,1.8345236365249842,0.32752060890197754,5106657,0,13,3.2377897004267586,0.972506046295166,14,3.2229611200633483,0.4495549201965332,16,3.520409199716998,0.4778475761413574
1,4,56664,1.00E+00,27135337,0,1.00E+03,13,2.7233899075142136,1.4671103954315186,1.00E+00,13,2.7133514872469706,1.0432589054107666,1.00E-03,10,1.856316325004444,0.8601276874542236,16912521,0,15,4.618218006434374,1.9838049411773682,17,4.755611794466449,1.3188135623931885,20,5.37562560677063,1.4988603591918945
1,5,104310,1.00E+00,85714384,0,1.00E+03,13,2.7792651272241806,3.9235517978668213,1.00E+00,13,2.7686663855315317,3.319973945617676,1.00E-03,10,1.862349047541245,2.9286677837371826,42496199,0,17,5.274635974721525,4.534912347793579,19,5.455039463084966,4.2212934494018555,21,6.08048804500162,4.536322832107544
1,6,173124,1.00E+00,223736529,0,1.00E+03,12,2.6201698393218127,13.232312440872192,1.00E+00,12,2.6099958199069135,11.587369918823242,1.00E-03,10,1.8770819638946958,10.958664417266846,89930757,0,17,5.64440373788134,10.694121837615967,19,5.827296970238878,10.578338384628296,22,6.394212467025486,11.73892092704773
1,7,266994,1.00E+00,506223922,0,1.00E+03,13,2.775173895434536,44.94590139389038,1.00E+00,13,2.764388228950455,38.03809452056885,1.00E-03,10,1.878950223395668,34.83893036842346,170050742,0,18,5.880254759572253,32.812156438827515,20,6.145300711209909,35.28279232978821,23,6.786165511267433,38.02065873146057
2,2,70056,1.00E+00,7538909,0,1.00E+03,16,3.705553135064412,1.2484519481658936,1.00E+00,16,3.7498572139333595,0.6402261257171631,1.00E-03,11,2.0183253109546926,0.5242409706115723,7538909,0,16,3.705553135065593,1.2153513431549072,16,3.7498572139333617,0.6335387229919434,11,2.0183253109546917,0.5038576126098633
2,3,200988,1.00E+00,52365395,0,1.00E+03,15,3.5902238155630433,2.440484046936035,1.00E+00,16,3.7784339820998243,1.883934497833252,1.00E-03,11,1.9308255569863355,1.3613228797912598,41708927,0,15,3.5859234032017295,3.130642890930176,15,3.68440132952102,1.8344480991363525,15,3.252631201941014,1.9099400043487549
2,4,437328,1.00E+00,219549709,0,1.00E+03,15,3.7260798202247982,7.548566102981567,1.00E+00,15,3.7331882490374952,6.104076385498047,1.00E-03,10,1.8900839597873544,4.93460488319397,138395869,0,16,4.069787232837765,7.975039005279541,17,4.703740576771023,6.916275501251221,19,4.809702396174518,7.511923789978027
2,5,810180,1.00E+00,688411782,0,1.00E+03,15,3.6848485137112323,24.811477422714233,1.00E+00,15,3.6915379250037352,20.514736652374268,1.00E-03,10,1.8880330729489723,16.312312841415405,347858902,0,17,5.217787434593262,25.216866493225098,18,5.447010594986189,22.9491446018219,20,5.61662560795085,24.583200454711914
2,6,1350648,1.00E+00,1793812823,0,1.00E+03,15,3.670092060646792,84.76515483856201,1.00E+00,15,3.676993710210196,74.37688112258911,1.00E-03,10,1.894730532151912,64.17081189155579,734307275,0,17,5.563316156349031,63.97129726409912,19,5.8332702605092885,64.70502519607544,21,5.9953917343453496,69.83771467208862
2,7,2089836,1.00E+00,4057248280,0,1.00E+03,15,3.7600540122964534,310.4987561702728,1.00E+00,15,3.7673991782209137,242.5157549381256,1.00E-03,10,1.895367628667036,196.02875995635986,1387841715,0,18,5.818611961030466,208.6723690032959,20,6.035081270389055,245.7009196281433,22,6.167234585742321,238.66336154937744
\end{filecontents*}
\csvreader[
	head to column names, head to column names prefix=MY,
	tabular=rlccc|,
	table head=\toprule 
	& & \multicolumn{3}{c|}{$\pafw{0}{Y^1}$} \\
	{$l$} & {$p$} & {$\alpha=10^3$} & {$\alpha=1$} & {$\alpha=10^{-3}$}\\
	\midrule,
	late after line=\ifthenelse{\equal{\MYdegree}{3}}{\\\midrule}{\\},
	filter ifthen=\equal{\MYdegree}{3} \or\equal{\MYdegree}{5} 
	\or\equal{\MYdegree}{7},
	table foot=\bottomrule
	]{results/N1curl.3d.demkowicz.afw.combined.csv}{}
	{\MYlevel{} & \MYdegree{} & \MYFAiterations{} (\MYAiterations{})  & 
	\MYFBiterations{} (\MYBiterations{}) & \MYFCiterations{} 
	(\MYCiterations{})}
	\hspace{-0.7em}
	\begin{filecontents*}{results/N1curl.3d.demkowicz.hiptmair.combined.csv}
level,degree,dofs,beta,nonzeros,memory,Aalpha,Aiterations,Acondest,Atime,Balpha,Biterations,Bcondest,Btime,Calpha,Citerations,Ccondest,Ctime,Fnonzeros,Fmemory,FAiterations,FAcondest,FAtime,FBiterations,FBcondest,FBtime,FCiterations,FCcondest,FCtime
0,2,1314,1.00E+00,56649,0,1.00E+03,11,2.0476111455669663,0.44923973083496094,1.00E+00,13,2.4273844857189215,0.05818939208984375,1.00E-03,18,4.16484747863665,0.07055521011352539,56649,0,11,2.0476111455669104,0.3853447437286377,13,2.427384485718921,0.05199432373046875,18,4.164847478636652,0.06647920608520508
0,3,3591,1.00E+00,467877,0,1.00E+03,16,4.133393070129046,0.4695916175842285,1.00E+00,17,4.243932007340881,0.1036064624786377,1.00E-03,23,6.754997931251514,0.13435935974121094,322077,0,16,4.358438866048845,1.1174774169921875,18,4.472355768483282,0.14902424812316895,22,6.437701971190216,0.17361831665039062
0,4,7596,1.00E+00,2431733,0,1.00E+03,19,6.015993208291946,0.7242927551269531,1.00E+00,21,6.1041554040046835,0.22634315490722656,1.00E-03,26,8.632296469977927,0.297884464263916,1159448,0,17,5.2357707147220784,0.8949253559112549,19,5.294848349822593,0.30030393600463867,24,7.712594913575781,0.34623074531555176
0,5,13815,1.00E+00,8853765,0,1.00E+03,19,6.610941343338527,1.0517909526824951,1.00E+00,21,6.647933111626232,0.6851863861083984,1.00E-03,26,9.061707284316,0.8797750473022461,3092905,0,18,5.870952444126795,2.169229030609131,20,5.899903599206664,0.7827739715576172,25,8.220237993653589,0.9675028324127197
0,6,22734,1.00E+00,25180865,0,1.00E+03,19,6.955915878442713,2.5980963706970215,1.00E+00,22,6.9779790566616775,2.238018751144409,1.00E-03,26,9.160312375721738,2.432060956954956,6812095,0,18,6.244299690391339,5.626602411270142,20,6.270006181682217,1.7993531227111816,25,8.393995034995134,2.0724267959594727
0,7,34839,1.00E+00,62513135,0,1.00E+03,19,7.136479888937449,7.256209135055542,1.00E+00,22,7.20461040168024,6.914708614349365,1.00E-03,25,9.178952203682623,7.009173631668091,13172279,0,18,6.5377093365517265,10.223592042922974,21,6.60292569903614,5.815986633300781,25,8.537467906795893,6.829340696334839
1,2,9324,1.00E+00,346752,0,1.00E+03,17,4.99129841546676,1.4878661632537842,1.00E+00,19,5.468294959749884,0.2927231788635254,1.00E-03,21,5.892133484299378,0.36532020568847656,346752,0,17,4.99129841546655,1.3809161186218262,19,5.468294959749887,0.29093313217163086,21,5.8921334842993724,0.3290102481842041
1,3,26298,1.00E+00,3587144,0,1.00E+03,19,6.333286340943569,0.7479853630065918,1.00E+00,21,6.523320882300314,0.5236926078796387,1.00E-03,26,8.70451680492371,0.5949704647064209,2420744,0,19,6.60656415532823,1.2499098777770996,22,6.862458285868443,0.6221635341644287,26,8.78244130938043,0.7080068588256836
1,4,56664,1.00E+00,19350986,0,1.00E+03,21,8.005492084336675,1.64631986618042,1.00E+00,24,8.162688895936089,1.3914804458618164,1.00E-03,29,10.81043403627557,1.6960043907165527,9035978,0,20,7.146028934414781,2.1440250873565674,23,7.487881974088913,1.5272572040557861,28,10.029464772038391,1.7914056777954102
1,5,104310,1.00E+00,70918983,0,1.00E+03,21,8.492949487277087,4.791685342788696,1.00E+00,24,8.598550383323396,4.348966360092163,1.00E-03,29,11.062563187467422,5.026270627975464,24375391,0,21,7.775645386847354,5.029639482498169,23,7.957731712578873,4.56255316734314,29,10.524920804484568,5.412900924682617
1,6,173124,1.00E+00,201433021,0,1.00E+03,21,8.685750838696094,13.87264633178711,1.00E+00,24,8.846758085271098,13.806548833847046,1.00E-03,29,11.091707018218772,15.494983196258545,53974486,0,21,7.958198966420012,11.64664340019226,23,8.13807554584152,11.67334270477295,29,10.767139865888694,14.074678659439087
1,7,266994,1.00E+00,493430075,0,1.00E+03,21,8.925928376321245,38.97936797142029,1.00E+00,24,9.067613479506928,41.2550995349884,1.00E-03,29,11.09628198741873,46.935932636260986,104759447,0,21,8.266633143374317,34.83637070655823,24,8.492778226953796,35.334293603897095,29,10.908075143729135,43.356924295425415
2,2,70056,1.00E+00,2767838,0,1.00E+03,18,5.372451103022022,2.0306363105773926,1.00E+00,20,6.0954393721707945,1.0270309448242188,1.00E-03,20,5.175822156790262,1.0754780769348145,2767838,0,18,5.372451103029596,1.845400333404541,20,6.095439372170793,0.9992015361785889,20,5.175822156790262,1.0219354629516602
2,3,200988,1.00E+00,28520880,0,1.00E+03,20,6.7614087957182685,3.124408721923828,1.00E+00,22,7.122072726384196,2.239562511444092,1.00E-03,24,7.596029297705258,2.379767656326294,19189680,0,20,7.057511402550876,4.305463790893555,22,7.438143139432865,2.5812876224517822,25,7.7609225348137745,2.845587730407715
2,4,437328,1.00E+00,154681942,0,1.00E+03,22,8.57816480601003,8.863444805145264,1.00E+00,24,8.943472084798378,7.786565542221069,1.00E-03,28,9.83768005576069,8.696601629257202,71861512,0,21,7.680465392370648,9.347707509994507,23,7.9996141989937595,8.029937267303467,26,8.763205667070952,8.891589641571045
2,5,810180,1.00E+00,568432451,0,1.00E+03,22,9.082999082231625,27.895915269851685,1.00E+00,24,9.41626621324197,25.157098054885864,1.00E-03,28,10.139152828723672,27.91436195373535,194321443,0,21,8.256799030120314,27.094651460647583,23,8.545379122209063,26.543710470199585,27,9.249000838639232,30.222848653793335
2,6,1350648,1.00E+00,1617062933,0,1.00E+03,22,9.330603432876641,90.27322149276733,1.00E+00,25,9.779564389504673,89.40520000457764,1.00E-03,28,10.360297369422511,95.89575910568237,431068553,0,21,8.585946551130482,71.62147569656372,24,8.98400432859198,73.96020317077637,27,9.595136786382055,81.16948008537292
2,7,2089836,1.00E+00,3931644883,0,1.00E+03,22,9.570124254607487,297.9142200946808,1.00E+00,25,10.048042761085535,280.67670369148254,1.00E-03,28,10.500476463203748,209.4733920097351,837741010,0,22,8.972857797740971,248.0665135383606,24,9.331466590342185,253.08928966522217,28,9.86610525685894,236.18561959266663
\end{filecontents*}
\csvreader[
	head to column names, head to column names prefix=MY,
	tabular=|ccc,
	table head=\toprule
	\multicolumn{3}{c}{$\ph{Y^0}{Y^1}$}\\ 
	{$\alpha=10^3$} & {$\alpha=1$} & {$\alpha=10^{-3}$}\\
	\midrule,
	late after line=\ifthenelse{\equal{\MYdegree}{3}}{\\\midrule}{\\},
	filter ifthen=\equal{\MYdegree}{3} \or\equal{\MYdegree}{5} 
	\or\equal{\MYdegree}{7},
	table foot=\bottomrule
	]{results/N1curl.3d.demkowicz.hiptmair.combined.csv}{}
	{\MYFAiterations{} (\MYAiterations{})  & \MYFBiterations{} 
	(\MYBiterations{}) & \MYFCiterations{} (\MYCiterations{})}
	%
	%
	
	\vspace{0.3em}
	\begin{filecontents*}{results/N1curl.3d.demkowicz.hiptmair-red.combined.csv}
level,degree,dofs,beta,nonzeros,memory,Aalpha,Aiterations,Acondest,Atime,Balpha,Biterations,Bcondest,Btime,Calpha,Citerations,Ccondest,Ctime,Fnonzeros,Fmemory,FAiterations,FAcondest,FAtime,FBiterations,FBcondest,FBtime,FCiterations,FCcondest,FCtime
0,2,1314,1.00E+00,48450,0,1.00E+03,13,2.568187317620318,0.5038561820983887,1.00E+00,13,2.6190315202834236,0.0908665657043457,1.00E-03,31,12.006174556887652,0.13254523277282715,48450,0,13,2.5681873176211814,0.5016112327575684,13,2.6190315202834213,0.07072210311889648,31,12.006174556887672,0.12878942489624023
0,3,3591,1.00E+00,331005,0,1.00E+03,14,3.3100377636474523,0.35367512702941895,1.00E+00,15,3.4261502916931703,0.09723567962646484,1.00E-03,51,33.008478110802464,0.24584126472473145,220197,0,14,3.298065579389529,1.030188798904419,15,3.433731420789811,0.17205452919006348,47,27.660141332740174,0.41013646125793457
0,4,7596,1.00E+00,1568786,0,1.00E+03,15,3.758075752598065,0.647472620010376,1.00E+00,16,3.8403276929875094,0.1918962001800537,1.00E-03,65,54.0441959760626,0.5438995361328125,712193,0,15,3.8193461392383914,1.0867125988006592,16,3.9220765780908895,0.28517770767211914,61,46.58082913249701,0.8299589157104492
0,5,13815,1.00E+00,5397181,0,1.00E+03,15,4.074654804178854,0.9048542976379395,1.00E+00,16,4.151534859355849,0.5638995170593262,1.00E-03,74,72.5779201911478,1.7426505088806152,1794889,0,15,4.165688961124735,2.1458983421325684,17,4.285568142680461,0.6540091037750244,69,62.88246720309173,2.126694917678833
0,6,22734,1.00E+00,14865110,0,1.00E+03,15,4.266747292531223,1.7186622619628906,1.00E+00,17,4.34405287178511,1.3897967338562012,1.00E-03,81,87.17304896599829,4.646114110946655,3816760,0,15,4.356929582303201,2.874023199081421,17,4.490587587100787,1.417872428894043,76,76.24490269911293,5.082202196121216
0,7,34839,1.00E+00,35064881,0,1.00E+03,15,4.332441175491637,4.278027296066284,1.00E+00,16,4.4588484154506025,3.9742274284362793,1.00E-03,86,99.74115972906658,13.61622142791748,7204415,0,15,4.449801961654649,5.028283596038818,17,4.647398065842803,3.9812676906585693,80,86.197705996028,15.875405073165894
1,2,9324,1.00E+00,286938,0,1.00E+03,18,5.3239356860025815,1.4322583675384521,1.00E+00,19,5.4968013739599915,0.3585078716278076,1.00E-03,34,14.86924894476309,0.6309573650360107,286938,0,18,5.323935686005867,1.4345483779907227,19,5.4968013739599995,0.3735964298248291,34,14.869248944763111,0.624091625213623
1,3,26298,1.00E+00,2526836,0,1.00E+03,18,5.51028341131707,0.7611653804779053,1.00E+00,20,5.695495145031101,0.5558993816375732,1.00E-03,49,30.291669269965258,1.1715786457061768,1640372,0,18,5.604434562063107,1.6343061923980713,20,5.8105562223643705,0.7410478591918945,44,25.272457862066258,1.2669203281402588
1,4,56664,1.00E+00,12392492,0,1.00E+03,18,5.501482249169886,1.3806779384613037,1.00E+00,19,5.762332445847412,1.1822419166564941,1.00E-03,55,40.50364524370787,2.6572561264038086,5567612,0,18,5.674157804700535,2.1890387535095215,20,6.032535466371108,1.5215318202972412,51,34.52009685557373,3.2436890602111816
1,5,104310,1.00E+00,43336739,0,1.00E+03,18,5.654998377899311,3.888709306716919,1.00E+00,19,5.85997326998152,3.3774333000183105,1.00E-03,60,49.058041243108676,8.320318222045898,14251867,0,18,5.7972883305547835,4.181201696395874,20,6.197593768155319,3.9188525676727295,55,41.07821507302019,9.221864700317383
1,6,173124,1.00E+00,119566051,0,1.00E+03,17,5.466756938043509,8.73501181602478,1.00E+00,19,5.759426391908811,9.045920133590698,1.00E-03,63,54.3134211726687,22.634852409362793,30535876,0,18,5.759023220376306,9.353990316390991,20,6.103928694461415,9.637124061584473,58,45.96104285346399,24.479559659957886
1,7,266994,1.00E+00,280702961,0,1.00E+03,17,5.567183025366917,24.27796196937561,1.00E+00,19,5.871526822854941,26.97056555747986,1.00E-03,65,58.88532468329338,65.34384036064148,57960023,0,19,5.979404231450431,27.075551509857178,21,6.320657754112441,28.313053131103516,60,49.771051153951085,70.21819972991943
2,2,70056,1.00E+00,2311898,0,1.00E+03,19,5.0592813485041015,2.255641460418701,1.00E+00,21,6.150375249391742,1.5333163738250732,1.00E-03,25,8.273510292894477,1.7049164772033691,2311898,0,19,5.059281348497975,2.136638879776001,21,6.150375249391749,1.5071618556976318,25,8.273510292894494,1.7141849994659424
2,3,200988,1.00E+00,20172120,0,1.00E+03,19,5.417310318038601,2.804701805114746,1.00E+00,20,6.237598127224971,2.516322612762451,1.00E-03,32,13.292675654767296,3.3819289207458496,13080408,0,19,5.900149257173284,3.809715509414673,20,6.34028195399053,3.159982681274414,30,11.340596874242271,4.050702095031738
2,4,437328,1.00E+00,99564466,0,1.00E+03,18,5.009988501869478,6.887904167175293,1.00E+00,20,6.122301222012848,6.622705936431885,1.00E-03,35,15.572408797642645,9.879562854766846,44530276,0,19,5.700281869277056,8.336378335952759,20,6.351805100817656,7.685723543167114,32,13.600303679351631,10.930355787277222
2,5,810180,1.00E+00,348007175,0,1.00E+03,18,5.577650508267306,20.477570056915283,1.00E+00,20,6.144260099277181,20.315773963928223,1.00E-03,36,17.269224219681757,31.246561288833618,114305467,0,19,6.0021679731161734,22.80456566810608,20,6.472589021171139,22.544703483581543,34,15.253973920955286,34.73837685585022
2,6,1350648,1.00E+00,961042313,0,1.00E+03,18,5.584515672803356,57.41093850135803,1.00E+00,19,6.093553785231679,57.519261837005615,1.00E-03,37,18.615447531809405,92.26628923416138,245488013,0,19,6.010879376280101,58.566489696502686,21,6.52970333170874,62.382620334625244,36,16.58228462449591,97.84830498695374
2,7,2089836,1.00E+00,2256580975,0,1.00E+03,18,5.625695713037274,173.67386960983276,1.00E+00,19,6.140290368162519,186.44203853607178,1.00E-03,38,19.660259082919527,299.83285093307495,466776130,0,19,5.993771723788222,179.65190434455872,21,6.581292609530171,190.0144875049591,37,17.639090016262035,300.0162320137024
\end{filecontents*}
\csvreader[
	head to column names, head to column names prefix=MY,
	tabular=rlccc,
	table head=
	& & \multicolumn{3}{c}{$\ph{Y^0}{Y^{1,\rmone}}$}\\ 
	{$l$} & {$p$} & {$\alpha=10^3$} & {$\alpha=1$} & {$\alpha=10^{-3}$}\\
	\midrule,
	late after line=\ifthenelse{\equal{\MYdegree}{3}}{\\\midrule}{\\},
	filter ifthen=\equal{\MYdegree}{3} \or\equal{\MYdegree}{5} 
	\or\equal{\MYdegree}{7},
	table foot=\bottomrule
	]{results/N1curl.3d.demkowicz.hiptmair-red.combined.csv}{}
	{\MYlevel{} & \MYdegree{} & \MYFAiterations{} (\MYAiterations{})  & 
	\MYFBiterations{} 
	(\MYBiterations{}) & \MYFCiterations{} (\MYCiterations{})}
	\hphantom{	\hspace{-1em}
		\csvreader[
		head to column names, head to column names prefix=MY,
		tabular=|ccc|,
		table head=\toprule
		\multicolumn{3}{c|}{$\ph{Y^0}{Y^1}$}\\ 
		{$\alpha=10^3$} & {$\alpha=1$} & {$\alpha=10^{-3}$}\\
		\midrule,
		late after line=\ifthenelse{\equal{\MYdegree}{3}}{\\\midrule}{\\},
		filter ifthen=\equal{\MYdegree}{3} \or\equal{\MYdegree}{5} 
		\or\equal{\MYdegree}{7},
		table foot=\bottomrule
		]{results/N1curl.3d.demkowicz.hiptmair.combined.csv}{}
		{\MYFAiterations{} (\MYAiterations{})  & \MYFBiterations{} 
			(\MYBiterations{}) & \MYFCiterations{} (\MYCiterations{})}}
\end{table}

\subsubsection{$\Hdiv$ solvers}

We again first consider the vertex patch subspace decompositions 
$\pafw{0}{X^3}$ and 
$\pafw{0}{X_{\interface}^3}+\jacobi{X_{\interior}^3}$ with the same symmetric 
hybrid Schwarz solver described in \cref{fig:solver-diagram-PAFW}. The 
iteration 
counts for these two solvers with $p \in \{3,5,7\}$ are displayed in the 
upper 
left quadrant of \cref{tab:hdiv-unit-cube}. After one or two levels of 
refinement, $\pafw{0}{X_{\interface}^3}+\jacobi{X_{\interior}^3}$ requires 
about 
double the number of iterations as $\pafw{0}{X^3}$, but all iteration counts 
appear to be robust in $\alpha$, $h$, and $p$. For the edge patch subspace 
decompositions $\pafw{1}{X^3}$ and 
$\pafw{1}{X_{\interface}^3}+\jacobi{X_{\interior}^3}$, described in 
\cref{fig:solver-diagram-PAFW} and displayed in the upper right of 
\cref{tab:hdiv-unit-cube}, the iteration counts for $\alpha \geq 1$ differ by 
at 
most 2 and are insensitive to $\alpha$ and $p$. For $\alpha = 10^{-3}$, the 
iteration counts decrease as the mesh is refined and on the finest mesh, the 
iteration counts are also within 2 and appear to be insensitive to $p$.

The iteration counts for the Hiptmair--Toselli decompositions 
$\ph{X^{2}}{X^{3}}$, \\
$\ph{X_{\interface}^{2}}{X_{\interface}^{3}}+\jacobi{X_{\interior}^3}$,
$\ph{X^{2,\rmone}}{X^{3,\rmone}}$, and 
$\ph{X_{\interface}^{2,\rmone}}{X_{\interface}^{3,\rmone}}+\jacobi{X_{\interior}^3}$,
described in \cref{fig:solver-diagram-PH}, are displayed in the bottom half 
of 
\cref{tab:hdiv-unit-cube}. Note that the results are similar to the edge 
patch 
decompositions, with the only exception being that the type-I/type-II 
splitting 
$\ph{X^{2,\rmone}}{X^{3,\rmone}}$ performs worse than $\ph{X^{2}}{X^{3}}$ on 
coarser meshes. Overall, $\pafw{0}{X^3}$ and 
$\pafw{0}{X_{\interface}^3}+\jacobi{X_{\interior}^3}$ perform the best, but 
all solvers with the interior-interface split perform nearly identically for 
$\alpha \geq 1$ and on fine enough meshes for $\alpha = 10^{-3}$.

\begin{table}[htbp]
	\centering
	\caption{PCG iteration counts for the $\Hdiv$ Riesz map, where $Y^k = 
		X^k_{\interface} \ (Y^k = X^k)$.\label{tab:hdiv-unit-cube}}
	\begin{filecontents*}{results/N1div.3d.demkowicz.afw0.combined.csv}
level,degree,dofs,beta,nonzeros,memory,Aalpha,Aiterations,Acondest,Atime,Balpha,Biterations,Bcondest,Btime,Calpha,Citerations,Ccondest,Ctime,Fnonzeros,Fmemory,FAiterations,FAcondest,FAtime,FBiterations,FBcondest,FBtime,FCiterations,FCcondest,FCtime
0,2,1620,1.00E+00,153054,0,1.00E+03,7,1.2735284482780391,0.35399889945983887,1.00E+00,7,1.284772926794729,0.04179668426513672,1.00E-03,6,1.1777215721723975,0.03241610527038574,90090,0,7,1.2730578613790848,0.6979289054870605,7,1.278170053322419,0.06013631820678711,23,6.8678883543137434,0.1460399627685547
0,3,4212,1.00E+00,896940,0,1.00E+03,7,1.3162645446660897,0.200791597366333,1.00E+00,7,1.3177244192661743,0.04914283752441406,1.00E-03,7,1.2063440028850227,0.04909396171569824,329688,0,10,1.7874638033302999,0.6752562522888184,10,1.7881504289520158,0.08103466033935547,27,9.348140181176689,0.1692960262298584
0,4,8640,1.00E+00,3557980,0,1.00E+03,8,1.351752426463463,0.46487927436828613,1.00E+00,8,1.3536927305032287,0.18353056907653809,1.00E-03,7,1.2717367083448126,0.17517447471618652,895860,0,11,2.0189484817688856,0.8185524940490723,11,2.018544618912991,0.3399536609649658,30,11.56067246530439,0.3390681743621826
0,5,15390,1.00E+00,10868535,0,1.00E+03,8,1.3722573232366977,1.1571884155273438,1.00E+00,8,1.3743844478923488,0.7469191551208496,1.00E-03,7,1.3167852991932165,0.8446745872497559,2000070,0,12,2.181021165632618,1.501561164855957,12,2.182012297488156,0.44402337074279785,30,11.778421373623384,0.8489222526550293
0,6,24948,1.00E+00,27770544,0,1.00E+03,8,1.385347165103616,3.655888080596924,1.00E+00,8,1.3875112790936852,2.857832193374634,1.00E-03,8,1.360913105731894,2.8746471405029297,3906648,0,12,2.2523867395464574,2.525845766067505,12,2.2526804764815425,1.3315191268920898,30,12.036283859845195,2.571850061416626
0,7,37800,1.00E+00,62637110,0,1.00E+03,8,1.393738320298331,12.171893119812012,1.00E+00,8,1.396013115579091,10.27650237083435,1.00E-03,8,1.3862439403012012,10.123080253601074,6932790,0,12,2.3249448767290866,5.014179468154907,12,2.3249252747948868,3.4922237396240234,30,12.288872767785728,6.316992998123169
1,2,12312,1.00E+00,1202364,0,1.00E+03,10,1.8794987283147004,1.4940788745880127,1.00E+00,10,1.9024250377693863,0.1967604160308838,1.00E-03,8,1.3542396667209342,0.16131091117858887,701775,0,10,1.8874463551724256,1.4535350799560547,10,1.899975962395137,0.17655324935913086,19,4.864094151071684,0.31144261360168457
1,3,32400,1.00E+00,7260606,0,1.00E+03,10,1.8890435024169858,0.43315672874450684,1.00E+00,10,1.8957257701163743,0.26912379264831543,1.00E-03,8,1.4089209826476006,0.2330019474029541,2729286,0,15,3.258548986940455,0.6785235404968262,15,3.258167335750916,0.4030640125274658,22,6.537443090916306,0.5121653079986572
1,4,66960,1.00E+00,28792494,0,1.00E+03,10,1.8691604024760293,1.4435863494873047,1.00E+00,10,1.8748619917381928,1.0963234901428223,1.00E-03,8,1.449371311463043,0.9989035129547119,7522134,0,18,4.700218299643435,1.362790584564209,18,4.7003117485863815,1.0178656578063965,23,7.098027880757877,1.2619929313659668
1,5,119880,1.00E+00,87796089,0,1.00E+03,10,1.8473239025362729,5.024600505828857,1.00E+00,10,1.852917237322715,4.107877731323242,1.00E-03,8,1.4736712540441945,3.940826892852783,16871319,0,20,5.297877952539951,3.793720245361328,20,5.297348744627724,3.3190243244171143,24,7.561876190854625,3.876176118850708
1,6,195048,1.00E+00,224136675,0,1.00E+03,10,1.8642724389755618,18.611954927444458,1.00E+00,10,1.8699202494097922,15.335853576660156,1.00E-03,8,1.4887601502999013,14.66108751296997,33015591,0,21,5.901477576278157,11.778434991836548,21,5.901217215069848,11.103031158447266,24,7.892280991733837,12.369308233261108
1,7,296352,1.00E+00,503459082,0,1.00E+03,10,1.8668436703593105,63.26119589805603,1.00E+00,10,1.8730404066431747,58.147743225097656,1.00E-03,8,1.4992340917969933,57.151734828948975,58641450,0,22,6.383821487711182,31.778433799743652,22,6.383624328820212,30.4686439037323,24,7.829197772417155,33.74477553367615
2,2,95904,1.00E+00,9830017,0,1.00E+03,11,2.247936039897094,0.9773266315460205,1.00E+00,12,2.2883365072136015,0.4781813621520996,1.00E-03,8,1.4431402413448151,0.3629481792449951,5845195,0,11,2.244153022961111,1.1601636409759521,12,2.279575559839048,0.5011703968048096,12,2.184282133229868,0.49219322204589844
2,3,254016,1.00E+00,58808269,0,1.00E+03,11,2.187504374406742,1.9190778732299805,1.00E+00,11,2.208750602631922,1.1982402801513672,1.00E-03,8,1.4693047511556867,1.0743372440338135,22637413,0,15,3.252465360180206,2.092971086502075,15,3.252453173088198,1.3602097034454346,16,3.5775988287895575,1.4358818531036377
2,4,527040,1.00E+00,232121269,0,1.00E+03,11,2.195269723127196,7.512189626693726,1.00E+00,11,2.216720309052631,6.025895357131958,1.00E-03,9,1.5051541324434532,5.6574366092681885,62340389,0,19,4.75238596381045,6.413851499557495,19,4.752614135404899,5.4351513385772705,19,4.902625237342187,5.527374029159546
2,5,946080,1.00E+00,706491199,0,1.00E+03,11,2.192192550776762,27.939088821411133,1.00E+00,11,2.213522987900447,24.528162479400635,1.00E-03,9,1.5268682746225295,23.032651901245117,139794139,0,20,5.23694684932056,20.672467947006226,20,5.236925460791717,19.626277208328247,20,5.289562069550814,19.744046688079834
2,6,1542240,1.00E+00,1802441062,0,1.00E+03,11,2.191530282099567,116.07730150222778,1.00E+00,11,2.2129423637531427,91.333167552948,1.00E-03,9,1.5375099387620887,87.08341646194458,273548683,0,21,5.887908114934536,69.78619599342346,21,5.887892865735721,67.95631194114685,21,5.923084225567034,70.52277207374573
2,7,2346624,1.00E+00,4046322765,0,1.00E+03,11,2.171814438742177,402.96562123298645,1.00E+00,11,2.1934623556632524,258.57169938087463,1.00E-03,9,1.5473882437024409,234.25445461273193,485864045,0,22,6.379355008901936,206.52605366706848,22,6.379340317373742,203.05500102043152,22,6.40599056038984,238.4102234840393
\end{filecontents*}
\csvreader[
	head to column names, head to column names prefix=MY,
	tabular=rlccc|,
	table head=\toprule 
	&  & \multicolumn{3}{c|}{$\pafw{0}{Y^2}$} \\
	{$l$} & {$p$} & {$\alpha=10^3$} & {$\alpha=1$} & {$\alpha=10^{-3}$}\\
	\midrule,
	late after line=\ifthenelse{\equal{\MYdegree}{3}}{\\\midrule}{\\},
	filter ifthen=\equal{\MYdegree}{3} \or\equal{\MYdegree}{5} 
	\or\equal{\MYdegree}{7}
	]{results/N1div.3d.demkowicz.afw0.combined.csv}{}
	{\MYlevel{} & \MYdegree{} & \MYFAiterations{} (\MYAiterations{})  & 
	\MYFBiterations{} (\MYBiterations{}) & \MYFCiterations{} 
	(\MYCiterations{})}
	\hspace{-0.7em}
	\begin{filecontents*}{results/N1div.3d.demkowicz.afw.combined.csv}
level,degree,dofs,beta,nonzeros,memory,Aalpha,Aiterations,Acondest,Atime,Balpha,Biterations,Bcondest,Btime,Calpha,Citerations,Ccondest,Ctime,Fnonzeros,Fmemory,FAiterations,FAcondest,FAtime,FBiterations,FBcondest,FBtime,FCiterations,FCcondest,FCtime
0,2,1620,1.00E+00,121599,0,1.00E+03,13,2.4356370905094282,0.3618919849395752,1.00E+00,13,2.5651439359196657,0.051122426986694336,1.00E-03,10,1.8381134330835418,0.040068626403808594,54531,0,13,2.412729071366462,0.8300912380218506,13,2.5523841081614744,0.2245941162109375,24,7.013336260440529,0.20624589920043945
0,3,4212,1.00E+00,823140,0,1.00E+03,16,3.3892722078902606,0.2246716022491455,1.00E+00,16,3.3878803185308364,0.0800020694732666,1.00E-03,11,2.067619407298292,0.062033891677856445,187452,0,16,3.3850520605827774,0.8328831195831299,16,3.508225537697628,0.1476728916168213,28,9.518624749276059,0.27339720726013184
0,4,8640,1.00E+00,3577140,0,1.00E+03,17,3.8051345266728682,0.7207202911376953,1.00E+00,17,3.8031437455954205,0.22491884231567383,1.00E-03,13,2.495227406009135,0.18970704078674316,500760,0,17,3.8981164425464403,0.8676514625549316,17,3.9638577446024867,0.21611785888671875,30,11.701037251556329,0.34631991386413574
0,5,15390,1.00E+00,11637855,0,1.00E+03,17,4.094066697110036,1.030602216720581,1.00E+00,17,4.09328313659595,0.7714211940765381,1.00E-03,14,2.880313935139265,0.7805113792419434,1111095,0,18,4.252583859470113,1.425722360610962,18,4.296136551449921,0.5281367301940918,30,11.916849584959811,0.7854094505310059
0,6,24948,1.00E+00,31149297,0,1.00E+03,18,4.302426618378353,3.075195074081421,1.00E+00,18,4.302565580237683,2.6773576736450195,1.00E-03,14,3.21872492796834,2.4466300010681152,2164257,0,18,4.481463880136093,2.721203088760376,18,4.522248702069421,1.5433330535888672,30,12.176950695763878,2.346684455871582
0,7,37800,1.00E+00,75118950,0,1.00E+03,18,4.4463406819843945,16.920656204223633,1.00E+00,18,4.446417653753424,13.848006248474121,1.00E-03,15,3.49797782603591,9.0398690700531,3835206,0,18,4.683773786130016,5.180505275726318,18,4.70196574943962,4.014921426773071,30,12.414940243169934,6.023605823516846
1,2,12312,1.00E+00,913212,0,1.00E+03,20,5.53164382667076,1.461296796798706,1.00E+00,21,5.915832702727217,0.29282331466674805,1.00E-03,14,2.983248538025992,0.22219038009643555,376668,0,20,5.5362467147191134,1.9677855968475342,21,5.840905433298092,0.3652327060699463,19,4.974634117104867,0.3240993022918701
1,3,32400,1.00E+00,6514362,0,1.00E+03,21,5.715822337738219,0.6585257053375244,1.00E+00,21,5.873799718324931,0.4876260757446289,1.00E-03,16,3.684650215622104,0.3961942195892334,1428858,0,21,5.860553573121492,1.153503179550171,21,6.061032676856159,0.6870088577270508,23,6.693262866938603,0.761178731918335
1,4,66960,1.00E+00,28520874,0,1.00E+03,21,5.792988662514829,1.702855110168457,1.00E+00,21,5.933040462919053,1.4223105907440186,1.00E-03,17,4.18995952701844,1.2585015296936035,3909834,0,21,6.037485082920771,2.183617115020752,22,6.226202715331774,1.7186830043792725,23,7.186245324723443,1.2950756549835205
1,5,119880,1.00E+00,92957724,0,1.00E+03,21,5.78906324145823,5.655512809753418,1.00E+00,21,5.88970149026196,4.901672124862671,1.00E-03,18,4.567843671004634,4.578860759735107,8743644,0,22,6.104543856567831,3.7794601917266846,22,6.233619052693141,3.4678902626037598,24,7.645365967409675,3.681690216064453
1,6,195048,1.00E+00,248965668,0,1.00E+03,21,5.906658682568892,17.275254011154175,1.00E+00,21,6.010353051814919,16.26318597793579,1.00E-03,19,4.862807748651715,15.581942081451416,17085348,0,22,6.269847366685284,11.36376166343689,22,6.406491361125853,10.80069875717163,24,7.948259627749295,11.624374151229858
1,7,296352,1.00E+00,588775306,0,1.00E+03,21,5.809052382795971,55.877331256866455,1.00E+00,21,5.889813146941687,53.44322323799133,1.00E-03,19,5.0669683265392615,52.86985492706299,30321018,0,22,6.427887881220087,31.44952893257141,22,6.450213565216798,39.404356718063354,24,7.89497811377043,34.49088644981384
2,2,95904,1.00E+00,7333993,0,1.00E+03,22,6.341304925298517,1.6250724792480469,1.00E+00,23,6.877438903685486,1.0586156845092773,1.00E-03,16,3.5953405209279206,0.8786273002624512,3041641,0,22,6.339817921499453,1.821929931640625,23,6.727093460864963,1.1776530742645264,17,3.799131118910878,0.8835082054138184
2,3,254016,1.00E+00,52107229,0,1.00E+03,22,6.329618060607164,2.7158353328704834,1.00E+00,22,6.60928428505997,2.0176873207092285,1.00E-03,18,4.455588737918163,1.776054859161377,11423197,0,22,6.430509678695839,3.2866358757019043,23,6.754160201149679,2.0302858352661133,20,5.041536265020517,1.8126115798950195
2,4,527040,1.00E+00,228078109,0,1.00E+03,22,6.335133611247368,9.770575761795044,1.00E+00,22,6.593472924424673,8.662758588790894,1.00E-03,19,5.037161714667333,7.857197999954224,31189789,0,23,6.624958594953986,6.998693466186523,23,6.898535278660891,6.151487112045288,21,5.6646920643475545,5.688573837280273
2,5,946080,1.00E+00,743417929,0,1.00E+03,22,6.185068426450923,34.703941106796265,1.00E+00,22,6.377229719281247,31.15486717224121,1.00E-03,20,5.4814296502250155,29.55981469154358,69705289,0,23,6.579158428732432,22.2374484539032,23,6.784332040660994,21.035367488861084,22,6.129169543897585,20.011333227157593
2,6,1542240,1.00E+00,1991217097,0,1.00E+03,22,6.234179798579874,112.80103540420532,1.00E+00,22,6.440441740097068,105.66942501068115,1.00E-03,20,5.747089196350511,136.62302207946777,136174537,0,23,6.697035664511359,71.30064392089844,23,6.8948982430598065,70.13018536567688,22,6.397194305302188,68.16431999206543
2,7,2346624,1.00E+00,4669626845,0,1.00E+03,21,6.2233378493947695,379.45029401779175,1.00E+00,22,6.330417832355697,267.13714027404785,1.00E-03,21,5.9613563301764,246.68033576011658,241643341,0,23,6.745285887802773,204.35257411003113,23,6.854092196242927,203.50735664367676,23,6.623563317410253,222.71417093276978
\end{filecontents*}
\csvreader[
	head to column names, head to column names prefix=MY,
	tabular=|ccc,
	table head=\toprule 
	\multicolumn{3}{c}{$\pafw{1}{Y^2}$} \\
	{$\alpha=10^3$} & {$\alpha=1$} & {$\alpha=10^{-3}$}\\
	\midrule,
	late after line=\ifthenelse{\equal{\MYdegree}{3}}{\\\midrule}{\\},
	filter ifthen=\equal{\MYdegree}{3} \or\equal{\MYdegree}{5} 
	\or\equal{\MYdegree}{7}
	]{results/N1div.3d.demkowicz.afw.combined.csv}{}
	{\MYFAiterations{} (\MYAiterations{})  & \MYFBiterations{} 
	(\MYBiterations{}) & \MYFCiterations{} (\MYCiterations{})}
	
	\begin{filecontents*}{results/N1div.3d.demkowicz.hiptmair.combined.csv}
level,degree,dofs,beta,nonzeros,memory,Aalpha,Aiterations,Acondest,Atime,Balpha,Biterations,Bcondest,Btime,Calpha,Citerations,Ccondest,Ctime,Fnonzeros,Fmemory,FAiterations,FAcondest,FAtime,FBiterations,FBcondest,FBtime,FCiterations,FCcondest,FCtime
0,2,1620,1.00E+00,80757,0,1.00E+03,14,2.882917067008871,0.5474421977996826,1.00E+00,14,2.8857595360253088,0.07100200653076172,1.00E-03,15,3.1769642914493823,0.06324100494384766,53541,0,14,2.8715709836996974,0.9058265686035156,14,2.864578867399224,0.11005544662475586,24,7.263372855736589,0.16824793815612793
0,3,4212,1.00E+00,713061,0,1.00E+03,19,4.678878515311744,0.5121760368347168,1.00E+00,19,4.677938662267967,0.10065197944641113,1.00E-03,18,4.2294465772694245,0.10300469398498535,281979,0,18,4.195282896259808,1.1766159534454346,18,4.197441145635243,0.15369153022766113,28,9.687258111855403,0.21206450462341309
0,4,8640,1.00E+00,3668256,0,1.00E+03,20,5.326732493942882,0.8802385330200195,1.00E+00,20,5.32636772016836,0.2709958553314209,1.00E-03,19,4.845418212862663,0.2668449878692627,960750,0,19,4.902950161983991,1.3065638542175293,19,4.9051201886126705,0.2711961269378662,31,11.876511455556557,0.39507198333740234
0,5,15390,1.00E+00,13283590,0,1.00E+03,21,5.772306377706054,1.8839259147644043,1.00E+00,21,5.772045723660545,0.9341456890106201,1.00E-03,20,5.324092624263259,0.9090662002563477,2488914,0,20,5.377383691724555,2.0207555294036865,20,5.379577061979185,0.679840087890625,30,12.045782912436795,0.942176342010498
0,6,24948,1.00E+00,37852317,0,1.00E+03,22,6.0857592440478925,4.335445404052734,1.00E+00,22,6.08563176524647,3.3828461170196533,1.00E-03,21,5.710657123955589,3.3112523555755615,5360625,0,21,5.723288770799712,4.9154393672943115,21,5.724946442053045,2.4985034465789795,30,12.305997574681228,2.8167624473571777
0,7,37800,1.00E+00,101166309,0,1.00E+03,22,6.328625677780175,79.28771591186523,1.00E+00,22,6.32853830899888,10.239043474197388,1.00E-03,21,5.989349716275722,10.434195518493652,10218051,0,21,6.008570887341735,9.949601411819458,21,6.009995082540283,7.425627708435059,31,12.553172335616118,9.07687520980835
1,2,12312,1.00E+00,573516,0,1.00E+03,22,6.4899401421243885,2.083848476409912,1.00E+00,23,6.831803601090773,0.5608320236206055,1.00E-03,19,4.49841645771377,0.43859219551086426,355788,0,22,6.493508107889976,1.9702632427215576,23,6.827508873042369,0.5076868534088135,20,5.237293564588193,0.4171593189239502
1,3,32400,1.00E+00,5576436,0,1.00E+03,24,7.290765165287504,0.9616115093231201,1.00E+00,24,7.581379237783196,0.6324903964996338,1.00E-03,22,5.998532795064721,0.60959792137146,2127780,0,23,7.002866816779233,1.1817781925201416,24,7.299538731959769,0.6956663131713867,24,7.18768165541694,0.7124054431915283
1,4,66960,1.00E+00,29237166,0,1.00E+03,24,7.559876422762203,2.0902726650238037,1.00E+00,25,7.824300407799796,1.7971384525299072,1.00E-03,23,6.68057758262687,1.6924591064453125,7426782,0,24,7.350010888336036,2.5366415977478027,24,7.590410052682717,1.4636309146881104,24,7.440424776563166,1.4841558933258057
1,5,119880,1.00E+00,106235986,0,1.00E+03,25,7.796791475024985,6.792170286178589,1.00E+00,25,7.975572036648547,6.095799446105957,1.00E-03,23,7.14309958607005,5.833007097244263,19389978,0,24,7.53764152664689,4.921860456466675,25,7.738836540432534,4.490153551101685,25,7.856918592314691,4.519988536834717
1,6,195048,1.00E+00,302581080,0,1.00E+03,25,8.002342162513516,22.649660110473633,1.00E+00,25,8.16711524188289,21.28542733192444,1.00E-03,24,7.453081825246703,20.82619571685791,41918544,0,25,7.8620625248262055,15.701388359069824,25,8.040054286770273,15.12414288520813,25,7.9602106235818635,16.06277847290039
1,7,296352,1.00E+00,808307748,0,1.00E+03,25,8.028552654406313,71.88086557388306,1.00E+00,25,8.128145244759473,76.05791974067688,1.00E-03,24,7.66616892383164,69.55275821685791,80070768,0,25,7.895905497387745,55.61349129676819,25,8.030587633038687,41.765913009643555,25,7.978592113602814,51.11723232269287
2,2,95904,1.00E+00,4563721,0,1.00E+03,24,7.4142466051374205,3.2848384380340576,1.00E+00,25,7.915400609941799,1.8556206226348877,1.00E-03,21,5.580450798586785,1.5404956340789795,2821897,0,24,7.412885075530996,3.576857805252075,25,7.910942603028898,1.8176987171173096,21,5.623780323110128,1.575981616973877
2,3,254016,1.00E+00,44376265,0,1.00E+03,25,8.00917598169369,4.313581705093384,1.00E+00,26,8.470480575394461,3.2420339584350586,1.00E-03,23,6.8283981016840825,2.9607934951782227,16787017,0,24,7.636741809092969,4.916827440261841,25,8.085073503379448,3.0364272594451904,23,6.579955986785091,2.825221538543701
2,4,527040,1.00E+00,233495029,0,1.00E+03,25,8.199897151003833,12.378623485565186,1.00E+00,26,8.571340761806976,10.6086266040802,1.00E-03,24,7.478578728803815,10.22278881072998,58690549,0,25,7.966988834771867,10.728648662567139,26,8.36517780710384,8.367465496063232,24,7.28480249328964,7.904452800750732
2,5,946080,1.00E+00,850114173,0,1.00E+03,25,8.34911890092841,43.269330739974976,1.00E+00,26,8.560281153589328,39.12016177177429,1.00E-03,25,8.010320256089912,38.36125564575195,153432109,0,25,8.122144547324814,29.154276609420776,26,8.38774890328699,27.003806829452515,25,7.817938421512127,26.716354370117188
2,6,1542240,1.00E+00,2424245473,0,1.00E+03,26,8.558223065635636,156.19075179100037,1.00E+00,26,8.686106668330503,141.8033411502838,1.00E-03,25,8.350874521241659,140.61236190795898,332043985,0,26,8.373321945691817,102.66959357261658,26,8.573222041967504,109.4024932384491,25,8.19165197966024,92.33768010139465
2,7,2346624,1.00E+00,6471120889,0,1.00E+03,26,8.739138291260971,380.6249506473541,1.00E+00,26,8.788834583748162,318.61137795448303,1.00E-03,25,8.62614761542833,327.82077741622925,634715137,0,26,8.559806099728064,302.96173191070557,26,8.65681087290251,213.14753651618958,26,8.497184799321676,203.27830934524536
\end{filecontents*}
\csvreader[
	head to column names, head to column names prefix=MY,
	tabular=rlccc|,
	table head= \toprule
	&  & \multicolumn{3}{c|}{$\ph{Y^1}{Y^2}$} \\
	{$l$} & {$p$} & {$\alpha=10^3$} & {$\alpha=1$} & {$\alpha=10^{-3}$}\\
	\midrule,
	late after line=\ifthenelse{\equal{\MYdegree}{3}}{\\\midrule}{\\},
	filter ifthen=\equal{\MYdegree}{3} \or\equal{\MYdegree}{5} 
	\or\equal{\MYdegree}{7},
	table foot=\bottomrule
	]{results/N1div.3d.demkowicz.hiptmair.combined.csv}{}
	{\MYlevel{} & \MYdegree{} & \MYFAiterations{} (\MYAiterations{})  & 
	\MYFBiterations{} (\MYBiterations{}) & \MYFCiterations{} 
	(\MYCiterations{})}
	\hspace{-0.7em}
	\begin{filecontents*}{results/N1div.3d.demkowicz.hiptmair-red.combined.csv}
level,degree,dofs,beta,nonzeros,memory,Aalpha,Aiterations,Acondest,Atime,Balpha,Biterations,Bcondest,Btime,Calpha,Citerations,Ccondest,Ctime,Fnonzeros,Fmemory,FAiterations,FAcondest,FAtime,FBiterations,FBcondest,FBtime,FCiterations,FCcondest,FCtime
0,2,1620,1.00E+00,55980,0,1.00E+03,13,2.4181044911795055,0.6086950302124023,1.00E+00,13,2.465605340877807,0.08907103538513184,1.00E-03,24,7.253168573339854,0.11007452011108398,39672,0,13,2.4112732475363376,0.9071762561798096,13,2.4045808172151384,0.10226249694824219,24,7.315715616678894,0.1691887378692627
0,3,4212,1.00E+00,370719,0,1.00E+03,16,3.3892788813796395,0.5511026382446289,1.00E+00,16,3.394720099590518,0.1068727970123291,1.00E-03,30,10.898146688774693,0.15453433990478516,158553,0,16,3.384725635823619,0.9173634052276611,16,3.385509413072728,0.14871597290039062,28,9.771284717739336,0.2323153018951416
0,4,8640,1.00E+00,1673739,0,1.00E+03,17,3.80513939820989,0.8827402591705322,1.00E+00,17,3.808286826737278,0.20217466354370117,1.00E-03,32,13.191976928825001,0.3146233558654785,454905,0,17,3.8980082848494852,1.147209882736206,17,3.8986787854325167,0.2481684684753418,31,11.983963566358065,0.40322399139404297
0,5,15390,1.00E+00,5629856,0,1.00E+03,17,4.0940699565615155,1.3533399105072021,1.00E+00,17,4.096691826846675,0.48447561264038086,1.00E-03,33,13.67650637946003,0.7892844676971436,1060488,0,18,4.252514824766973,1.5916635990142822,18,4.253051651202841,0.5891985893249512,31,12.148820356176447,0.9419541358947754
0,6,24948,1.00E+00,15311502,0,1.00E+03,18,4.302428497831607,2.1511220932006836,1.00E+00,18,4.304623253305493,1.3059241771697998,1.00E-03,32,13.546971073647988,1.9975323677062988,2111274,0,18,4.481400248593635,2.5799167156219482,18,4.4818060074939226,1.6451647281646729,30,12.385454583171514,2.47369647026062
0,7,37800,1.00E+00,35900127,0,1.00E+03,18,4.446341844576036,5.5183424949646,1.00E+00,18,4.447872057671647,3.8014397621154785,1.00E-03,32,13.533284231513589,5.589282989501953,3800367,0,18,4.683749980795361,6.246482610702515,18,4.684140651819498,4.180715560913086,30,12.600103026925087,6.235746145248413
1,2,12312,1.00E+00,383886,0,1.00E+03,20,5.52423093645436,2.0539610385894775,1.00E+00,21,5.7975054780069994,0.6186251640319824,1.00E-03,23,6.789609688883008,0.6060936450958252,253854,0,20,5.536055295002799,1.683941125869751,21,5.779962948424315,0.39270758628845215,21,5.603002869614689,0.4178917407989502
1,3,32400,1.00E+00,2884248,0,1.00E+03,21,5.715851801231593,0.7801108360290527,1.00E+00,21,5.906123289724564,0.5869467258453369,1.00E-03,26,8.346115419880482,0.7244336605072021,1187352,0,21,5.860548665043079,0.9573662281036377,21,6.065758887709794,0.5804579257965088,23,6.814781576381969,0.6138801574707031
1,4,66960,1.00E+00,13259160,0,1.00E+03,21,5.793015011478242,1.5459916591644287,1.00E+00,21,5.962898565351052,1.2719128131866455,1.00E-03,26,8.552181681843656,1.489353895187378,3523176,0,21,6.037500693341441,1.6616835594177246,22,6.248660077724512,1.3030693531036377,24,7.286719324134874,1.4353702068328857
1,5,119880,1.00E+00,45150630,0,1.00E+03,21,5.789080922020908,3.7560923099517822,1.00E+00,21,5.909896705820675,3.2941384315490723,1.00E-03,26,8.601729754446666,4.236478328704834,8297694,0,22,6.104555362615137,4.2744972705841064,22,6.249985404962751,3.8638503551483154,24,7.722580462651355,4.106680393218994
1,6,195048,1.00E+00,123058350,0,1.00E+03,21,5.90667683863765,9.70292615890503,1.00E+00,21,6.031580874451078,9.002458810806274,1.00E-03,26,8.518251930608228,10.516662836074829,16592958,0,22,6.269861907432793,12.120806455612183,22,6.429330217055313,11.703897953033447,25,8.057545162407855,12.853827476501465
1,7,296352,1.00E+00,287197434,0,1.00E+03,21,5.80906551336341,25.633382081985474,1.00E+00,21,5.90645975523706,24.175615787506104,1.00E-03,25,8.469007265218304,26.555858612060547,29929824,0,22,6.427889438352198,32.82071924209595,22,6.45724635794159,35.43008089065552,24,7.9589484285738505,41.359140396118164
2,2,95904,1.00E+00,3080485,0,1.00E+03,22,6.339919194254781,3.0182719230651855,1.00E+00,23,6.7062765422464325,2.0278375148773193,1.00E-03,22,6.027106685868481,1.9975976943969727,2041957,0,22,6.339786455430676,2.0105113983154297,23,6.697789102666787,1.2794170379638672,20,5.446889842140612,1.0740385055541992
2,3,254016,1.00E+00,23019985,0,1.00E+03,22,6.329628813667898,3.527465581893921,1.00E+00,22,6.621522329136687,2.926175117492676,1.00E-03,23,6.948861653493604,3.0324389934539795,9446545,0,22,6.430513170634164,3.0556423664093018,23,6.759178962487663,2.3498194217681885,21,5.783154691435371,2.1675188541412354
2,4,527040,1.00E+00,106353649,0,1.00E+03,22,6.33514353980195,8.383046388626099,1.00E+00,22,6.605023174678432,7.049619674682617,1.00E-03,24,7.586655520842291,7.688027620315552,28011313,0,23,6.624967382294068,8.086679935455322,23,6.909490149928326,7.162973403930664,22,6.095354165857061,6.8537373542785645
2,5,946080,1.00E+00,361960353,0,1.00E+03,22,6.185074555382103,23.331510066986084,1.00E+00,22,6.385501332840545,21.298228979110718,1.00E-03,25,8.07388151402134,23.211268424987793,65964133,0,23,6.579163491009715,24.99238657951355,23,6.7921520444965955,23.7206552028656,22,6.2767833431011955,22.939653873443604
2,6,1542240,1.00E+00,987557653,0,1.00E+03,22,6.234185714097356,63.68155026435852,1.00E+00,22,6.450307043101216,59.55971384048462,1.00E-03,25,8.34701981808469,65.36006951332092,131948245,0,23,6.697040492590882,77.96449136734009,23,6.904404441677067,79.14694833755493,22,6.4890382010122325,73.36915731430054
2,7,2346624,1.00E+00,2305632277,0,1.00E+03,21,6.223339627367318,177.71249508857727,1.00E+00,22,6.336586452438169,169.45179390907288,1.00E-03,25,8.562604759167671,216.32945227622986,238046257,0,23,6.7452869752452544,234.6039206981659,23,6.858562376691355,220.55348563194275,23,6.668198885274865,219.0473599433899
\end{filecontents*}
\csvreader[
	head to column names, head to column names prefix=MY,
	tabular=|ccc,
	table head= \toprule
	\multicolumn{3}{c}{$\ph{Y^{1,\rmone}}{Y^{2,\rmone}}$} \\
	{$\alpha=10^3$} & {$\alpha=1$} & {$\alpha=10^{-3}$}\\
	\midrule,
	late after line=\ifthenelse{\equal{\MYdegree}{3}}{\\\midrule}{\\},
	filter ifthen=\equal{\MYdegree}{3} \or\equal{\MYdegree}{5} 
	\or\equal{\MYdegree}{7},
	table foot=\bottomrule
	]{results/N1div.3d.demkowicz.hiptmair-red.combined.csv}{}
	{\MYFAiterations{} (\MYAiterations{})  & \MYFBiterations{} 
	(\MYBiterations{}) & \MYFCiterations{} (\MYCiterations{})}
\end{table}

\subsubsection{Summary remarks}
\label{sec:rm-summary-remarks}

For all of the Riesz maps considered, each space decomposition with and 
without 
the interior-interface splitting have the same behavior as $h\to 0$, 
$p\to\infty$, and 
$h^2\beta/\alpha\to 0$, which is consistent with 
\cref{thm:generic-subspace-decomp-interface-interior-split}. The degradation 
of the splitting as $h^2\beta/\alpha \to \infty$ present in 
\Cref{thm:generic-subspace-decomp-interface-interior-split} is confirmed
by a jump in iteration counts for $l=0$ and $\alpha = 10^{-3}$ between
$\pafw{0}{X^1}$ and
$\pafw{0}{X^1_{\interface}}+\jacobi{X^1_{\interior}}$,
between $\pafw{0}{X^2}$ and
$\pafw{0}{X^2_{\interface}}+\jacobi{X^2_{\interior}}$, 
between $\pafw{1}{X^2}$ and
$\pafw{1}{X^2_{\interface}}+\jacobi{X^2_{\interior}}$, and
between 
$\ph{X^1}{X^2}$ and
$\ph{X^1_{\interface}}{X^2_{\interface}}+\jacobi{X^2_{\interior}}$.
That the remaining solvers appear to be unaffected by $h^2\beta/\alpha$ 
warrants further theoretical investigation. Also note that the effect of 
changing shape regularity is observed as there is a slight increase in the 
iteration counts 
after one level of mesh refinement. Moreover, 
\cref{lem:ph-typeI-typeII-stable} seems to partially capture the lack of 
$p$-robustness of the type-I/type-II splitting in as observed in 
$\ph{X^0}{X^{1,\rmone}}$ and 
$\ph{X_{\interface}^0}{X_{\interface}^{1,\rmone}}+\jacobi{X_{\interior}^1}$ 
when $\alpha=10^{-3}$. The apparent 
$p$-robustness of all the Type-I/Type-II splittings for $\alpha \geq 1$ and 
for $\ph{X^{1,\rmone}}{X^{2,\rmone}}$ for $\alpha = 10^{-3}$ is not explained 
by our current theory.

We again emphasize that the cost of one iteration for the standard solvers, 
which include all interior degrees of freedom in the patch solves, is 
significantly more expensive than the solvers with the interior-interface 
splittings. In the case $h^2 \beta /\alpha \leq 1$, all of the iteration 
counts observed are nearly identical between solvers with and without the 
interior-interface splitting. Moreover, all solvers with the 
interior-interface splitting considered for a particular Riesz map gave 
nearly identical iteration counts. Thus, in the case that $h^2 \beta /\alpha 
\leq 1$, we recommend solvers for each of the Riesz maps as follows. For 
$\Hgrad$, there is only one choice, 
$\pafw{0}{X^0_{\interface}}+\jacobi{X^0_{\interior}}$. For $\Hcurl$, we take 
the finest splitting 
$\ph{X_{\interface}^0}{X_{\interface}^{1,\rmone}}+\jacobi{X^1_{\interior}}$. 
For $\Hdiv$, we take $\pafw{1}{X^2_{\interface}}+\jacobi{X^2_{\interior}}$ 
because this decomposition involves patch problems that are only marginally 
more expensive than the patch problems in the finest splitting $\ph{X^{1, 
\rmone}}{X^{2,\rmone}}+\jacobi{X^2_{\interior}}$, as discussed in 
\cref{sec:hdiv-decompositions}, but avoids the need for an auxiliary space. 
We will use these decompositions in the next section to precondition the 
Hodge--Laplace problem.

\subsection{Hodge--Laplacians} \label{sec:results-hodge-laplace}

The Riesz maps are very useful in the construction of block preconditioner 
for coupled
systems of partial differential equations \cite{mardal11}. Here, we employ 
block diagonal preconditioners based on our developed space decomposition
methods for the Riesz maps to solve  Hodge--Laplacians \cite{Arnold2006} on 
the Fichera corner $\Omega = (0,1)^3 \setminus (0.5, 1)^3$ with $\Gamma_N = 
\partial \Omega$ as a prototypical example.
For $k \in 1:3$, the problem is to find
$(\sigma, u) \in X^{k-1} \times X^k$ such that
\begin{subequations}
	\label{eq:hodge-laplace}
	\begin{alignat}{2}
		-(\sigma, \tau) + (u, \d^{k-1} \tau) &= 0 \qquad & &\forall  \tau \in 
		X^{k-1}
		\\
		(\d^{k-1} \sigma, v)  + (\d^k u, \d^{k} v) &= F(v) \qquad & &\forall 
		v \in X^k,
	\end{alignat}	
\end{subequations}
where we recall that $d^3 := 0$. The cases $k \in 1:2$ correspond to a mixed 
formulation of the vector Poisson problem, while the case $k=3$ 
yields the mixed formulation of the scalar Poisson problem in $\Hdiv \times 
\Ltwo$. We consider a single unstructured mesh, pictured in 
\cref{fig:fichera}. The mesh is refined towards the re-entrant vertex and the 
edges abutting the vertex, and the diameter of the refined tetrahedra are 
about $1/8$ the diameter of the largest tetrahedra. We prescribe a constant 
source term $F(v) = (f, v)$ for a constant
$f = 1$ or $f=(1, 1, 1)^\top$.

Since problem \cref{eq:hodge-laplace} is well-posed in $H(\d^{k-1}) \times 
H(\d^k)$ \cite{Arnold2006}, we consider the following block diagonal 
(weighted) Riesz map preconditioner:
\begin{equation} \label{eq:augmented-lagrangian}
	(\sigma, \tau) + (\d^{k-1}\sigma, \gamma\d^{k-1}\tau) + 
	(u, \gamma^{-1}v) + (\d^{k}u, \d^{k}v).
\end{equation}
The case $\gamma = 1$ corresponds to standard operator preconditioning 
\cite{mardal11}, while $k=3$ and $\gamma \gg 1$ corresponds to augmented 
Lagrangian preconditioning \cite{FortinGlow83,Hiptmair96}. For $k \in 1:2$ 
and $\gamma \neq 1$, the weighted Riesz map preconditioner 
\cref{eq:augmented-lagrangian} appears to be novel. We show in 
\cref{thm:augmented-lagrangian-hl} that if one uses 
\cref{eq:augmented-lagrangian} to precondition \cref{eq:hodge-laplace}, then 
all of the eigenvalues of the preconditioned system converge to $\pm 1$ as 
$\gamma \to \infty$.

Of course, the preconditioner \cref{eq:augmented-lagrangian} is too expensive 
to use directly, so we replace each of the weighted Reisz maps with the 
hybrid Schwarz method preconditioners we used above. Since we only encounter 
the case $\beta/\alpha \leq 1$, we restrict the discussion to the subspace 
decompositions we recommend in \cref{sec:rm-summary-remarks} and the 
corresponding decompositions that do not split the interior-interface degrees 
of freedom:
\begin{itemize}
	\item $\CG_p$: $\pafw{0}{X^0_{\interface}}+\jacobi{X^0_{\interior}}$ (and 
	$\pafw{0}{X^0}$),
	
	\item $\Ned_p$: 
	$\ph{X_{\interface}^0}{X_{\interface}^{1,\rmone}}+\jacobi{X^1_{\interior}}$
	 (and $\ph{X^0}{X^{1,\rmone}}$),
	
	\item $\RT_p$:  $\pafw{1}{X^2_{\interface}}+\jacobi{X^2_{\interior}}$ 
	(and $\pafw{1}{X^2}$),
	
	\item $\DG_{p-1}$: $\jacobi{X^3}$ (and $\jacobi{X^3}$), an exact solver 
	as the basis is orthogonal.
\end{itemize}

\input{./figures/table_hodge_fichera.tex}

The preconditioned MINRES iteration counts for $p \in 1:6$ and $\gamma \in 
\{1,10^3\}$ with a relative tolerance of $10^{-8}$ are displayed in 
\Cref{tab:hodge-fichera}. First note that for $p=1$, the solver corresponds 
to 
using the exact weighted Riesz map preconditioner 
\cref{eq:augmented-lagrangian} 
as the multigrid V-cycle preconditioner reduces to an exact solver. Here, we 
observe that taking $\gamma = 10^{3}$ reduces the iteration counts from 5 or 
6 to 
3. For $p \geq 2$, we generally see that the solvers are robust in $p$ and 
that 
taking $\gamma$ large reduces the iteration counts at no extra expense in the 
solver. The effect is minimal for the $\Hcurl\times\Hdiv$ Hodge--Laplacian 
for 
high-order Schwarz solvers owing to difficulty of solving the individual 
Riesz 
maps. Suppose instead that we apply \cref{eq:augmented-lagrangian} as the 
preconditioner by performing an inner PCG iteration with the Schwarz solvers in 
\cref{sec:preconditioning} as the preconditioners. Then, we would expect 
between 3 and 6 outer iterations and between 20 and 25 inner iterations 
per outer iteration from 
\cref{tab:hgrad-unit-cube,tab:hcurl-unit-cube,tab:hdiv-unit-cube}. 
For the $\Hcurl\times\Hdiv$ Hodge--Laplacian, 
the total number of Schwarz solver applications is in line with our 
observed iteration counts in \cref{tab:hodge-fichera}. 
We also note that the interior-interface splitting again gives 
comparable 
iteration counts within 4 for the $\Hcurl \times \Hdiv$ and $\Hdiv \times 
\Ltwo$ 
Hodge--Laplacians and within 16 for the $\Hgrad \times \Hcurl$ 
Hodge--Laplacian.

\appendix

\section{Energy stable splittings}

By construction, the space $X^k$ admits direct sum decompositions based on 
splitting the interior and interface basis functions and based on splitting 
the 
type-$\rmone$ and type-$\rmtwo$ basis functions:
\begin{align} \label{eq:xk-interior-interface-typeI-typeII-split}
	X^k = X^{k}_{\interface} \oplus 
	X^k_{\interior} \quad \text{and} \quad 
	X^{k} = X^{k, \rmone} \oplus X^{k, \rmtwo}.
\end{align}
We now examine the energy stability of these two decompositions. Given a 
domain 
$\omega$, let
\begin{align*}
	(u, v)_{H(\d^k; \omega)} := (u, v)_{\omega} + (\d^k u, \d^k v)_{\omega} 
	\quad 
	\text{and} \quad \|u\|_{H(\d^k; \omega)}^2 := (u, u)_{H(\d^k; \omega)}.
\end{align*}

\subsection{$H(\d^k)$-stability on the reference cell}

We begin by analyzing the stability of the decompositions on the reference 
cell:
\begin{align}
	X^k(\That) = X^{k}_{\interface}(\That) \oplus 
	X^k_{\interior}(\That) \quad \text{and} \quad 
	X^{k}(\That) = X^{k, \rmone}(\That) \oplus X^{k, \rmtwo}(\That),
\end{align}
where $X^{k}_{\interface}(\That)$, etc.~denotes the space analogous to 
\cref{eq:global-spaces-typeI,eq:global-spaces-typeII,%
	eq:global-space-interior-interface} on the reference cell. In particular, 
	given 
$u \in X^k(\That)$, there exist unique $u_{\interface} \in  
X^{k}_{\interface}(\That)$, $u_{\interior} \in  X^{k}_{\interior}(\That)$, 
$u_{\rmone} \in X^{k, \rmone}(\That)$, and $u_{\rmtwo} \in X^{k, 
\rmtwo}(\That)$ 
such that
\begin{align}
	u = u_{\interface} + u_{\interior} \quad \text{and} \quad u = u_{\rmone} 
	+ 
	u_{\rmtwo},
\end{align}
and the mappings $u \mapsto (u_{\interface}, u_{\interior})$ and $u \mapsto 
(u_{\rmone}, u_{\rmtwo})$ are linear. The first result shows that the 
interior-interface decomposition is uniformly 
stable with respect to the polynomial degree.
\begin{theorem}	
	\label{thm:spectral-equiv-reference-stiffness}
	There exists a constant $C$ depending only on the dimension $d$ such that 
	for 
	all $u \in X^k(\That)$, there holds
	\begin{align}
		\label{eq:stable-decomp-reference-mass}
		\frac{1}{2} \|u\|_{\That}^2 \leq \| u_{\interface} \|_{\That}^2 
		+ \| u_{\interior} \|_{\That}^2 \leq  C \|u\|_{H(\d^k; \That)}^2
	\end{align}
	and
	\begin{align}
		\label{eq:stable-decomp-reference-dk}
		\|\d^k u\|_{\That}^2 = \| \d^k u_{\interface} \|_{\That}^2 
		+ \| \d^k u_{\interior} \|_{\That}^2. 
	\end{align}
\end{theorem}
\begin{proof}
	We prove the case $d=3$; the case $d=2$ follows from analogous arguments.
	
	\noindent \textbf{Step 1: $\d^k$ stability. } The degrees of freedom 
	\cref{eq:type-I-dofs,eq:type-II-dofs} show that
	$u_{\interior}$ satisfies
	\begin{subequations}
		\label{proof:eq:interior-orthog-conditions}
		\begin{alignat}{2}
			\label{proof:eq:interior-orthog-conditions-1}
			(u_{\interior}, \d^{k-1} y)_{\That} &= (u, \d^{k-1} y)_{\That} 
			\qquad 
			& &\Forall y 
			\in 
			\mathring{X}^{k-1}(\That), \\
			\label{proof:eq:interior-orthog-conditions-2}
			(\d^{k} u_{\interior}, \d^k v)_{\That} &= (\d^k u, \d^k 
			v)_{\That} 
			\qquad & &\Forall 
			v 
			\in 
			\mathring{X}^k(\That),
		\end{alignat}
	\end{subequations}
	and in particular $\d^k u_{\interior}$ is the $L^2(\That)$-projection of 
	$\d^k u$ 
	onto 
	$\d^k \mathring{X}^k(\That)$. \Cref{eq:stable-decomp-reference-dk} 
	immediately 
	follows.
	
	\noindent \textbf{Step 2: $L^2$-stability. } Suppose first that $k = 0$. 
	Then, 
	Poincar\'{e}'s inequality gives $\|u_{\interior}\|_{\That} \leq C 
	\|u\|_{H(\d^k; \That)}$. Now suppose that $k \in \{1,2\}$. Then, 
	conditions 
	\cref{proof:eq:interior-orthog-conditions} ensure that there exists $q 
	\in 
	\d^k 
	\mathring{X}^k(\That)$ such that
	\begin{subequations}
		\begin{alignat}{2}
			(u_{\interior}, v)_{\That} + (\d^k v, q)_{\That} &= (u, 
			v)_{\That} 
			\qquad & &\Forall v \in 
			\mathring{X}^k(\That), \\
			(\d^k u_{\interior}, r)_{\That} &= (\d^k u, r)_{\That} \qquad & 
			&\Forall r \in \d^k 
			\mathring{X}^k(\That).
		\end{alignat}
	\end{subequations}
	Applying standard estimates for saddle point systems gives
	\begin{align*}
		\| u_{\interior} \|_{H(\d^k; \That)} \leq C (\beta^k)^{-2} 
		\|u\|_{H(\d^k; 
			\That)},
	\end{align*}
	where $\beta^k$ is the inf-sup constant
	\begin{align*}
		\beta^k := \inf_{0 \neq r \in \d^k \mathring{X}^k(\That) } \sup_{0 
		\neq v 
			\in 
			\mathring{X}^k(\That) } \frac{ (\d^k v, r)_{\That} }{ 
			\|v\|_{H(\d^k, 
				\That)} 
			\|r\| }.
	\end{align*}
	
	We now show that $\beta^k \geq \beta_0 > 0$ for some $\beta_0$ 
	independent of 
	$p$. Let $r \in \d^k \mathring{X}^k(\That)$ be given. Thanks to 
	\cite[Theorem 
	5.1]{Ern2024}, there exists $v \in \mathring{X}^k(\That)$ such that $\d^k 
	v = 
	r$ and $\|v\|_{\That} \leq C \|r\|_{\That}$, where $C$ is independent of 
	$p$ 
	and 
	$r$. Then, $\|v\|_{H(\d^k, \That)} \leq C \|r\|_{\That}$, and so $\beta^k 
	\geq 
	C^{-1} > 0$. Consequently, for $k \in 0:2$, there holds $\|u_{\interior} 
	\|_{H(\d^k, 
		\That)} \leq C \|u\|_{H(\d^k, \That)}$. Inequality 
	\cref{eq:stable-decomp-reference-mass} then follows.
\end{proof}

We now turn to the type-$\rmone$/type-$\rmtwo$ decomposition. Let $\{ u_j^k,
\omega_{j, p}^{k, \rmone/\rmtwo} : j \in 1:\dim X^k(\That) \} \subset 
X^k(\That) 
\times \mathbb{R}$ denote the eigenfunctions and eigenvalues 
(in increasing order) of the following 
eigenvalue problem: 
\begin{subequations}
	\label{eq:typeI-typeII-eigenproblem}
	\begin{alignat}{2}
		(u_i^k, u_j^k)_{\That} &= \omega_{i, p}^{k, \rmone/\rmtwo} 
		\delta_{ij} 
		\qquad & &\Forall i, j \in 1:\dim 
		X^k(\That), \\
		(u_{i, \rmone}^k, u^k_{j, \rmone})_{\That} + (u_{i, 
			\rmtwo}^k, u_{j, \rmtwo}^{k})_{\That} &= \delta_{ij} 
		\qquad & &\Forall i, j \in 1:\dim X^k(\That).
	\end{alignat}
\end{subequations}

\noindent Then, \cref{lem:d-preserves-basis-functions} and
standard arguments show that the $L^2(\Khat)$ stability of the 	 
type-$\rmone$/type-$\rmtwo$ decomposition is determined by the minimal 
eigenvalue 
$\omega_{1, p}^{k, \rmone/\rmtwo}$.
\begin{lemma} \label{lem:stable-decomp-one-two-ref}
	For all $u \in X^k(\That)$, $\d^k u = \d^k u_{\rmone}$ and there holds
	\begin{align} \label{eq:stable-decomp-one-two-ref}
		\frac{1}{2} \|u\|_{\That}^2 \leq \|u_{\rmone}\|_{\That}^2 + 
		\|u_{\rmtwo}\|_{\That}^2 \leq \frac{1}{\omega_{1, p}^{k, 
		\rmone/\rmtwo}} 
		\|u\|_{\That}^2.
	\end{align}
\end{lemma}

\subsection{Energy stability on physical cell}

Given an element in $T \in \mesh$ and $u \in X^k$, define
\begin{align*}
	a_{T}^k(u, v) := (u, \beta v)_{T} + (\d^k u, \alpha \d^k v)_T  
	\ \ \text{and} \ \ 
	 \|u\|_{a, T}^2 := a_T^k(u, u) \qquad \Forall u, 
	v \in X^k(T).
\end{align*}
We again use the notation $u_{\interface}$, $u_{\interior}$, $u_{\rmone}$, 
and 
$u_{\rmtwo}$ to denote the decomposition of $u \in X^k$ into its components 
in 
the decompositions \cref{eq:xk-interior-interface-typeI-typeII-split}. 
Standard scaling arguments applied to 
\cref{thm:spectral-equiv-reference-stiffness} and 
\cref{lem:stable-decomp-one-two-ref} show that
these decompositions are energy stable.

\begin{corollary} \label{cor:stable-decomp-interior-interface}
	There exist constants $C_1, C_2 > 0$ depending only on shape 
	regularity, 
	$d$, and $k$ such that for all $u \in X^k_p(\mesh)$ 
	and $T \in \mesh$, there holds 
	\begin{alignat}{2}
		\label{eq:stable-decomp-physical-stiffness}
		\frac{1}{2} \|u\|_{a, T}^2 &\leq \| u_{\interface} \|_{a, T}^2 + 
		\| 		
		u_{\interior} \|_{a, T}^2 & &\leq C_1 \max\left\{ 1, 
		\frac{\beta}{\alpha} h_T^2 \right\}  \|u\|_{a, T}^2, \\
		\label{eq:stable-decomp-one-two-physical-stiffness}
		\frac{1}{2} \|u\|_{a, T}^2 &\leq \| u_{\rmone} \|_{a, T}^2 + 
		\| 		
		u_{\rmtwo} \|_{a, T}^2 & &\leq \frac{C_2}{\omega_{1, p}^{k, 
				\rmone/\rmtwo}}   \|u\|_{a, T}^2,
	\end{alignat}
	where $\omega_{1, p}^{k, \rmone/\rmtwo}$ is defined in 
	\cref{eq:typeI-typeII-eigenproblem}. 
\end{corollary}
\begin{proof}	
	Let $T \in \mesh$ and $\hat{u} := (\pullback_T^k)^{-1} 
	u$ so that $\d^k u = \pullback_T^{k+1} (\d^k \hat{u})$. 
	
	\noindent \textbf{Step 1: $H(\d^k; T)$-stability of $u_{\interior}$. } 
	Note 
	that 
	$\pullback^k_T$ is of the form  $\pullback^k_T \hat{v} = G_T^k \hat{v} 
	\circ 
	F_T^{-1}$ for some constant or matrix $G^k_T$. Thus, there holds
	\begin{align*}
		\|u_{\interior}\|_{T}^2 
			= |\det J_T| \cdot 
			\| G^k_T  \hat{u}_{\interior} \|_{\That}^2 
		&\leq |\det J_T| \cdot |G^k_T|^2  
			\|\hat{u}_{\interior} \|_{\That}^2 \\
		&\leq   |\det J_T| \cdot |G^k_T|^2  \left( \|\hat{u}\|_{\That}^2 + 
		\|\d^k 
		\hat{u} \|_{\That}\|^2 \right),
	\end{align*}
	where $|G^k_T|$ also denotes the spectral norm on matrices. Now,
	\begin{align*}
		\|\hat{u}\|_{\That}^2 = |\det J_T^{-1}| \cdot \| (G^k_T)^{-1} 
		u\|_{T}^2 
		\leq 
		|\det J_T|^{-1} |(G^k_T)^{-1}|^2 \|u\|_{T}^2
	\end{align*}
	and
	\begin{align*}
		\|\d^k \hat{u} \|_{\That}^2 = |\det J_T^{-1}| \cdot \| 
		(G^{k+1}_T)^{-1} 
		\d^k 
		u\|_{T}^2 \leq  |\det J_T|^{-1} |(G^{k+1}_T)^{-1}|^2 \| \d^k 
		u\|_{T}^2,
	\end{align*}
	and so
	\begin{align*}
		\|u_{\interior}\|_{T}^2 \leq |G^k_T|^2 \left( |(G^k_T)^{-1}|^2 
		\|u\|_{T}^2 +  |(G^{k+1}_T)^{-1}|^2 \| \d^k u\|_{T}^2 \right).
	\end{align*}
	One may then verify that $|G^k_T| \cdot |(G^k_T)^{-1}| \leq C$ and 
	$|G^k_T| 
	\cdot 
	|(G^{k+1}_T)^{-1}| \leq C h_T$ for some $C > 0$ 
	depending only on shape regularity, $d$, and $k$. Thus,
	\begin{align*}
		\| u_{\interior} \|_{T}^2 \leq C \left( \|u\|_{T}^2 + h_T^2 \|\d^k 
		u\|_{T}^2 \right).
	\end{align*}
	Similar scaling arguments applied to \cref{eq:stable-decomp-reference-dk} 
	show that
	\begin{align}
		\label{proof:eq:interior-dk-bound}
		\|\d^k u_{\interior} \|_{T} \leq C \| \d^k u\|_{T}.
	\end{align}
	
	\noindent \textbf{Step 2: \cref{eq:stable-decomp-physical-stiffness}. } 
	We 
	then obtain
	\begin{align*}
		(u_{\interior}, \beta u_{\interior})_{T} &\leq C \beta \left\{  
		(u, u)_{T} + h_T^2 (\d^k u, 
		\d^k u)_{T} \right\} \\
		&\leq C \left\{  (u, \beta u)_{T} + 
		\frac{\beta}{\alpha} h_T^2 (\d^k u, \alpha \d^k u)_{T} \right\} 
		\leq C \max\left\{ 1, \frac{\beta}{\alpha} h_T^2 \right\} a^k_{T}(u, 
		u),
	\end{align*}
	and so \cref{proof:eq:interior-dk-bound} gives 
	\begin{align*}
		a^k_{T}(u_{\interior}, u_{\interior}) \leq C \max\left\{ 1, 
		\frac{\beta}{\alpha} h_T^2 \right\} 
		a^k_{T}(u, u).
	\end{align*}
	Inequality \cref{eq:stable-decomp-physical-stiffness} now follows from 
	the 
	bound $\|u_{\interface}\|_{a, T}^2 \leq 2 \left( \|u\|_{a, T}^2 + 
	\|u_{\interior}\|_{a, T}^2 	\right)$.

	\noindent \textbf{Step 3: 
	\cref{eq:stable-decomp-one-two-physical-stiffness}. 
	} Inequality \cref{eq:stable-decomp-one-two-physical-stiffness} follows 
	from 
	\cref{lem:stable-decomp-one-two-ref} using analogous arguments from Steps 
	1 
	and 2 on noting that $\omega_{1, p}^{k, \rmone/\rmtwo} \leq 2$ by 
	\cref{eq:stable-decomp-one-two-ref}.
\end{proof}

We may further decompose the interior component into individual basis 
functions, 
and this decomposition is stable for any choice of $\alpha$ and $\beta$.
\begin{lemma}
	\label{lem:stable-decomp-physical-interior}
	Let $N_{K, \interior} := \dim \mathring{X}^k(T)$ and let $\{ \phi_{T, 
	i}^k : 
	i \in 1:N_{K, \interior} \}$ denote interior basis functions supported on 
	$T$.
	Then, there exists a constant $C > 0$ depending only on shape regularity 
	such 
	that if $u_{\interior} = \sum_{i=1}^{N_{K, \interior}} u_{\interior, i} 
	\phi_{T, i}^k$, 
	then 
	\begin{align}
		\label{eq:stable-decomp-physical-interior}
		C^{-1} \| u_{\interior}\|_{a, T}^2 \leq  \sum_{i=1}^{N_{K, \interior}}
		|u_{\interior, i}|^2 \| \phi_{T, i}^{k}\|_{a, T}^2 \leq C 
		\|u_{\interior}\|_{a, T}^2.
	\end{align}
\end{lemma}
\begin{proof}
	Thanks to \cref{lem:ref-element-matrix-properties}, 
	\cref{eq:stable-decomp-physical-interior} holds in the case that $T = 
	\That$ 
	is the reference cell and $(\alpha, \beta) \in \{  (0, 1), (1, 0) \}$. A 
	standard scaling argument similar to the one in Step 1 of the 
	proof of \cref{cor:stable-decomp-interior-interface} shows that 
	\cref{eq:stable-decomp-physical-interior} holds for all $T \in \mesh$ and 
	$(\alpha, \beta) \in \{  (0, 1), (1, 0) \}$, where 
	$C$ only depend on shape regularity. The case of general $\alpha$ and 
	$\beta$ 
	follows from linearity.	
\end{proof}

\subsection{Proof of \cref{lem:physical-element-decouple}}
\label{sec:proof-physical-element-decouple}

The result follows from  
\cref{thm:generic-subspace-decomp-interface-interior-split} applied on to a 
single element with $X_0^k = X^k$. \hfill \qedsymbol

\section{Further discussion of $\ph{X^{k-1,\rmone}}{X^{k,\rmone}}$} 

We first prove \cref{lem:ph-typeI-typeII-stable} and then discuss the 
$p$-dependence of the constant $C_{\rmone/\rmtwo}$.

\subsection{Proof of \cref{lem:ph-typeI-typeII-stable}}
\label{sec:proof-ph-typeI-typeII-stable}

Let $u \in X^k$ and let $u_0 \in W^k$, $u_{J} \in \left. \d^{k-1} X^{k-1} 
\right|_{\star J}$, $J \in \Delta_{k-1}(\mesh)$, and $u_{L} \in \left. X^k 
\right|_{\star L}$, $L \in \Delta_k(\mesh)$, be such that
\begin{align*}
	u_0 + \sum_{J \in \Delta_{k-1}(\mesh)} u_J + \sum_{L \in 
		\Delta_{k}(\mesh)} u_L &= u, \\
	\|u_0\|_{a}^2 + \sum_{J \in \Delta_{k-1}(\mesh)} \|u_J\|_{a}^2 + \sum_{L 
	\in 
		\Delta_{k}(\mesh)} \|u_L\|_{a}^2 &\leq C_{\mathrm{HT}} \|u\|_{a}^2,
	%
\end{align*}
where $\|\cdot\|_{a}$ again denotes the energy norm.
Thanks to \cref{eq:dk-image-decomp}, $\left. 
\d^{k-1} X^{k-1} \right|_{\star J} = \left. \d^{k-1} X^{k-1, \rmone} 
\right|_{\star J}$, and so $u_{J} \in \left. \d^{k-1} X^{k-1, \rmone} 
\right|_{\star J}$ for all $J \in \Delta_{k-1}(\mesh)$.

For each $L \in \Delta_k(\mesh)$, there exist unique $u_{L, \rmone} \in 
\left. 
X^{k, \rmone} \right|_{\star L}$ and $u_{L, \rmtwo} \in \left. X^{k, \rmtwo}  
\right|_{\star L}$ such that $u_{L} = u_{L, \rmone} + u_{L, \rmtwo}$ thanks 
to 
\cref{eq:typeI-typeII-direct-sum} restricted to $\star L$. Moreover, summing 
\cref{eq:stable-decomp-one-two-physical-stiffness} over the cells in $\star 
L$, 
we obtain
\begin{align}\label{eq:proof:typI-typeII-patch-stability}
	\| u_{L, \rmone} \|_{a, \star L}^2 + \|u_{L, \rmtwo} \|_{a, \star L}^2 
	&\leq 
	\frac{C}{\omega_{1, p}^{k, \rmone/\rmtwo}} 
	\|u_L\|_{a, \star L}^2.
\end{align}
Thanks to \cite[Corollary 2]{Licht2017complexes}, the complex 
\begin{equation*}
	\begin{tikzcd}[ampersand replacement=\&]
		\left. X^{k-1} \right|_{\star L} \arrow[r, "\d^{k-1}"] \& \left. 
		X^{k} 
		\right|_{\star L} 
		\arrow[r, "\d^k"] \& 
		\d^k \left. X^{k} \right|_{\star L} \arrow[r] \& 0
	\end{tikzcd}
\end{equation*}
is exact  since $\star L$ is contractible, and so $u_{L, \rmtwo} \in \d^{k-1} 
\left. X^{k-1} \right|_{\star L} = \d^{k-1}	\left. X^{k-1, \rmone} 
\right|_{\star L}$.

For $J \in \Delta_{k-1}(\mesh)$, let $\Delta_{k}(J)$ denote that subsimplices 
of 
$J$ of dimension $k$ and define
\begin{align*}
	\tilde{u}_J = u_J + \frac{1}{|\Delta_k(J)|} \sum_{L \in \Delta_k(J)} 
	u_{L, 
		\rmtwo}.
\end{align*}
Then, $\tilde{u}_J \in \d^{k-1}	\left. X^{k-1, \rmone} 
\right|_{\star L}$ and
\begin{align*}
	\| \tilde{u}_J \|_a^2 \leq 2 \left( \|u_J\|_a^2 + \frac{1}{|\Delta_k(J)|} 
	\sum_{L 
		\in \Delta_k(J)} \| u_{L,\rmtwo}\|_a^2 \right).
\end{align*}
Consequently, there holds
\begin{align*}
	u_0 +  \sum_{J \in \Delta_{k-1}(\mesh)} \tilde{u}_J + \sum_{L \in 
		\Delta_{k}(\mesh)} u_{L, \rmone} = u_0 + \sum_{J \in 
		\Delta_{k-1}(\mesh)} 
	u_J + \sum_{L \in \Delta_{k}(\mesh)} u_L = u,
\end{align*}
and applying \cref{eq:proof:typI-typeII-patch-stability} gives
\begin{align*}
	&\|u_0\|_a^2 + \sum_{J \in \Delta_{k-1}(\mesh)} \|\tilde{u}_J\|_a^2 + 
	\sum_{L 
		\in 
		\Delta_{k}(\mesh)} \|u_{L, \rmone}\|_a^2 \\
	&\qquad \leq 2 \left( \|u_0\|_a^2 + \sum_{J 
		\in \Delta_{k-1}(\mesh)} \|\tilde{u}_J\|_a^2 + \sum_{L \in 
		\Delta_{k}(\mesh)} \left( \|u_{L, \rmone}\|_a^2 + \|u_{L, 
		\rmtwo}\|_a^2\right) \right) \\
	&\qquad \leq C_{\mathrm{HT}} \frac{C}{\omega_{1, p}^{k, \rmone/\rmtwo}}   
	\|u\|_{a}^2,
\end{align*}
which completes the proof with $C_{\rmone/\rmtwo} = C/\omega_{1, p}^{k, 
	\rmone/\rmtwo}$. \hfill \qedsymbol

\subsection{The $p$-dependence of $C_{\rmone/\rmtwo}$}
\label{sec:typeI-typeII-constant-p}

The previous section shows that we may take $C_{\rmone/\rmtwo} = C/\omega_{1, 
	p}^{k, \rmone/\rmtwo}$ in 
\cref{lem:ph-typeI-typeII-stable}. We show the values of $\omega_{1, p}^{k, 
\rmone/\rmtwo}$ for $\Ned_p$ and $\RT_p$ with $p \in 1:10$ in 
\cref{fig:decoupling}. It appears that $\omega_{1, p}^{1, \rmone/\rmtwo} \sim 
\mathcal{O}(p^{-3})$, while $\omega_{1, p}^{2, \rmone/\rmtwo} \sim 
\mathcal{O}(p^{-2})$. Thus, the type-I/type-II splitting approach may incur 
an algebraic factor of $p$ in the condition number estimate for the Riesz map 
problems. However, we only observed such a drastic growth in terms of 
iteration counts for the solvers $\ph{X^0}{X^{1, \rmone}}$ and 
$\ph{X_{\interface}^0}{X_{\interface}^{1, \rmone}}+\jacobi{X^1_{\interior}}$ 
when $\alpha = 10^{-3}$ and $\beta = 1$, as shown in 
\cref{tab:hcurl-unit-cube}.

\begin{figure}
	\centering
	\input{./figures/plot_decoupling.tex}
	\caption{
		Type based decoupling constant $\omega^{k,\rmone/\rmtwo}_{1,p}$ 
		on an equilateral reference tetrahedron.
	}
	\label{fig:decoupling}
\end{figure}

\section{Augmented Lagrangian preconditioning for Hodge Laplacians}

We first recall a few properties of closed Hilbert complexes from 
\cite{Arnold2010}. 
Let $(\{W^k\}, \{\d^k\})$ be a closed Hilbert complex with inner-product 
$\langle \cdot, \cdot \rangle$ inducing the norm $\|\cdot\|$, 
and let $(\{V^k\}, \{\d^k\})$ be the 
associated domain complex:
\begin{equation}\label{eq:abstract-short-sequence}
	\begin{tikzcd}[ampersand replacement=\&]
		\cdots \arrow[r, "\d^{k-2}"] \& V^{k-1} \arrow[r, "\d^{k-1}"] \& 
		V^{k} 
		\arrow[r, "\d^k"] \& V^{k+1} \arrow[r, "\d^{k+1}"] \& \cdots,
	\end{tikzcd}
\end{equation}
where the inner-product on $V^k$ is $\langle \cdot, \cdot \rangle + 
\langle \d^k \cdot, \d^k \cdot \rangle$. 
Two important subspaces are the space of harmonic forms $\mathfrak{H}^k$ 
and the orthogonal complement of the kernel of $\d^k$:
\begin{align*}
	\mathfrak{H}^k &:= \{ \mathfrak{h} \in V^k : \d^{k} \mathfrak{h} = 0 
	\text{ and } \langle \mathfrak{h}, \d^{k-1} 
	\sigma \rangle = 0 \ \forall \sigma \in V^{k-1} \}, \\
	\mathfrak{Z}^{k, \perp} &:= \{ v \in V^{k} : \langle v, z \rangle = 0 \ 
	\forall z \in V^k \text{ with } \d^k z = 0 \}.
\end{align*}
In particular, since the complex is closed, we have the following 
orthogonal decomposition:
\begin{align*}
	V^{k} = \mathfrak{Z}^{k-1, \perp} \oplus \mathfrak{H}^k 
		    \oplus \mathfrak{Z}^{k, \perp}
\end{align*}

The mixed formulation of the Hodge Laplace problem associated to 
\cref{eq:abstract-short-sequence} reads as follows:
Find $(\sigma, u, p) \in V^{k-1} \times V^k \times \mathfrak{H}^{k}$ such that
\begin{subequations}\label{eq:abstract-hodge-laplace}
	\begin{alignat}{2}
		-\langle \sigma, \tau \rangle + \langle u, \d^{k-1} \tau \rangle &= 0 
		\qquad & &\forall \tau \in V^{k-1}, \\
		\langle \d^{k-1} \sigma, v \rangle + \langle p, v \rangle + \langle 
		\d^k u, \d^k v \rangle &= \langle f, v \rangle \qquad & &\forall v 
		\in 
		V^{k}, \\
		\langle u, q \rangle &= 0 \qquad & &\forall q \in \mathfrak{H}^{k}.
	\end{alignat}
\end{subequations}
We may combine the bilinear forms in \cref{eq:abstract-hodge-laplace} into 
one 
large symmetric bilinear form
\begin{align*}
	A(\sigma, u, p; \tau, v, q) := \langle \d^k u, \d^k v \rangle + \langle 
	\d^{k-1} \sigma, v \rangle + \langle u, \d^{k-1} \tau \rangle - \langle 
	\sigma, \tau \rangle + \langle p, v \rangle + \langle u, q \rangle.
\end{align*} 
Given $\gamma > 0$, we also introduce the following weighted inner-product on 
$V^{k-1} \times V^k \times \mathfrak{H}^{k}$:
\begin{align*}
	B_{\gamma}(\sigma, u, p; \tau, v, q) &:= \gamma^{-1} \langle u - 
	u_{\mathfrak{H}}, v - v_{\mathfrak{H}} \rangle + \langle 
	u_{\mathfrak{H}}, 
	v_{\mathfrak{H}} \rangle + \langle \d^k u, \d^k v \rangle
	\\
	&\qquad + \langle \sigma, \tau \rangle + \gamma \langle \d^{k-1} \sigma, 
	\d^{k-1} \tau \rangle + \langle p, q \rangle.
\end{align*}

\begin{theorem} \label{thm:augmented-lagrangian-hl}
	Let $\gamma > 0$ and suppose that $V^{k-1}$ and $V^{k}$ are finite 
	dimensional. Consider the 
	following symmetric eigenvalue problem: Find $(\sigma, u, p) \in V^{k-1} 
	\times V^k \times \mathfrak{H}^{k}$ and $\lambda \in \mathbb{R}$ such that
	\begin{align} \label{eq:hodge-laplace-evals}
		A(\sigma, u, p; \tau, v, q) = \lambda B_{\gamma}(\sigma, u, p; \tau, 
		v, 
		q) \qquad 
		\forall (\tau, v, q) \in V^{k-1} \times V^k \times \mathfrak{H}^{k}.
	\end{align}
	All of the eigenpairs of \cref{eq:hodge-laplace-evals} are classified as 
	follows.
	\begin{enumerate}
		\item[(i)] For any $\sigma \in V^{k-1}$, $(\sigma, -\gamma \d^{k-1} 
		\sigma, 0)$ and $\lambda = -1$ is an eigenpair of 
		\cref{eq:hodge-laplace-evals}.
		
		\item[(ii)] For any $p \in \mathfrak{H}^k$, $(0, \lambda p, p)$ and 
		$\lambda = \pm 1$ is an eigenpair of \cref{eq:hodge-laplace-evals}.
		
		\item[(iii)] Let $u_{\perp} \in \mathfrak{Z}^{k, \perp}$ and $\mu > 
		0$ be an eigenpair of the following eigenvalue problem:
		\begin{align} \label{eq:hodge-laplace-u-evals-aux}
			\langle \d^k u_{\perp}, \d^k v_{\perp} \rangle = \mu \langle 
			u_{\perp}, v_{\perp} \rangle \qquad 
			\forall v_{\perp} \in \mathfrak{Z}^{k, \perp}.
		\end{align}
		Then, $(0, u_{\perp}, 0)$ and $\lambda = \frac{\mu \gamma}{\mu \gamma 
		+ 
			1}$ is an eignpair of \cref{eq:hodge-laplace-evals}.
		
		\item[(iv)] Let $\sigma_{\perp} \in \mathfrak{Z}^{k-1, \perp}$ and 
		$\nu > 0$ be an eigenpair of the following eigenvalue problem: 
		\begin{align} \label{eq:hodge-laplace-sigma-evals-aux}
			\langle \d^{k-1} \sigma_{\perp}, \d^{k-1} \tau_{\perp} \rangle = 
			\nu 
			\langle \sigma_{\perp}, \tau_{\perp} \rangle \qquad \forall 
			\tau_{\perp} \in 
			\mathfrak{Z}^{k-1, \perp}.
		\end{align}
		Then, the following is an eigenpair of \cref{eq:hodge-laplace-evals}:
		\begin{align*}
			(\sigma_{\perp}, \gamma \lambda^{-1} \d^{k-1} \sigma_{\perp}, 0) 
			\quad \text{and} 
			\quad \lambda = \frac{\nu \gamma }{\nu \gamma + 1}.	
		\end{align*}
	\end{enumerate}
\end{theorem}
\begin{proof}
	(i-iv) may be verified by direct computation. We only show the details 
	for 
	(iii) and (iv). Let $u_{\perp}$, $\mu$, and $\lambda$ be as in 
	(iii) and $(\tau, v, q) \in V^{k-1} \times V^k \times \mathfrak{H}^k$. 
	Then,
	\begin{align*}
		A(0, u_{\perp}, 0; \tau, v, p) &= \langle \d^k u_{\perp}, \d^k v 
		\rangle 
		+ \langle u_{\perp}, \d^{k-1} \tau \rangle + \langle u_{\perp}, q 
		\rangle, 
		=  \langle \d^k u_{\perp}, \d^k v_{\perp} \rangle \\
		B_{\gamma}(0, u_{\perp}, 0; \tau, v, p) &= \frac{1}{\gamma} \langle 
		u_{\perp}, v - 
		v_{\mathfrak{H}} \rangle + \langle \d^k u_{\perp}, \d^k v \rangle = 
		\frac{1}{\gamma} \langle u_{\perp}, v_{\perp} \rangle + \langle \d^k 
		u_{\perp}, \d^k v_{\perp} \rangle.
	\end{align*}
	Consequently, there holds
	\begin{align*}
		\lambda B_{\gamma}(0, u_{\perp}, 0; \tau, v, p) = \frac{\mu 
		\gamma}{\mu 
			\gamma + 1} \left( \frac{1}{\gamma} + \mu \right) \langle 
			u_{\perp}, 
		v_{\perp} \rangle = \mu \langle u_{\perp}, v_{\perp} \rangle = A(0, 
		u_{\perp}, 0; \tau, v, p).
	\end{align*}
	Now let $\sigma_{\perp}$, $\nu$, and $\lambda$ be as in (iv). Then,
	\begin{align*}
		A(\sigma_{\perp}, \gamma \lambda^{-1} \d^{k-1} \sigma_{\perp}, 0; 
		\tau, v, q) &= \langle 
		\d^{k-1} \sigma_{\perp}, v \rangle + \gamma \lambda^{-1} \langle 
		\d^{k-1} \sigma_{\perp}, 
		\d^{k-1} \tau \rangle - \langle \sigma_{\perp}, \tau \rangle \\
		&= \langle \d^{k-1} \sigma_{\perp}, v_{\perp} \rangle + \gamma 
		\lambda^{-1} 
		\langle \d^{k-1} \sigma_{\perp}, \d^{k-1} \tau \rangle - \langle 
		\sigma_{\perp}, \tau 
		\rangle, 
	\end{align*}
	and
	\begin{align*}
		&B_{\gamma}(\sigma_{\perp}, \gamma \lambda^{-1} 
		\d^{k-1}\sigma_{\perp}, 0; \tau, v, q) \\
		&\qquad = 
		\lambda^{-1} \langle \d^{k-1} \sigma_{\perp}, v - v_{\mathfrak{H}} 
		\rangle + 
		\langle \sigma_{\perp}, \tau \rangle +  \gamma \langle \d^{k-1} 
		\sigma_{\perp}, 
		\d^{k-1} \tau \rangle \\
		&\qquad = \lambda^{-1} \langle \d^{k-1} \sigma_{\perp}, v_{\perp} 
		\rangle + 
		\langle \sigma_{\perp}, \tau \rangle +  \gamma \langle \d^{k-1} 
		\sigma_{\perp}, \d^{k-1} \tau \rangle \\
		&\qquad = \lambda^{-1} \langle \d^{k-1} \sigma_{\perp}, v_{\perp} 
		\rangle + (1 + \nu \gamma) \langle \sigma_{\perp}, \tau \rangle.
	\end{align*}
	Consequently, there holds
	\begin{align*}
		\lambda B_{\gamma}(\sigma_{\perp}, \gamma \lambda^{-1} 
		\d^{k-1}\sigma_{\perp}, 0; \tau, 
		v, q) &= \langle \d^{k-1} \sigma_{\perp}, v_{\perp} \rangle + 
		\lambda \left(1 + \nu \gamma \right) \langle \sigma_{\perp}, \tau 
		\rangle \\
		&= \langle \d^{k-1} \sigma_{\perp}, v_{\perp} \rangle + 
		\nu \gamma \langle \sigma_{\perp}, \tau \rangle \\
		&= \langle \d^{k-1} \sigma_{\perp}, v_{\perp} \rangle +  (\nu \gamma  
		\lambda^{-1} 
		- 1) \langle \sigma_{\perp}, \tau \rangle \\
		&= A(\sigma, \gamma \lambda^{-1} \d^{k-1} \sigma_{\perp}, 0; \tau, v, 
		q).
	\end{align*}
	
	We now verify that (i-iv) contain $\dim V^{k-1} + \dim V^{k} + \dim 
	\mathfrak{H}^k$ linearly independent (L.I.) eigenvectors, and so (i-iv) 
	classify all eigenpairs of \cref{eq:hodge-laplace-evals}. It is immediate 
	that (i) contains $\dim V^{k-1}$ L.I.~eigenfunctions and (ii) contains 
	$\dim 
	\mathfrak{H}^k$ L.I.~eigenfunctions for each value of $\lambda \in \{-1, 
	1\}$. Since 
	\cref{eq:hodge-laplace-u-evals-aux,eq:hodge-laplace-sigma-evals-aux} are 
	well-defined, SPD eigenvalue problems, there are $\dim \mathfrak{Z}^{k, 
		\perp}$ L.I.~eigenfunctions in (iii) and $\dim \mathfrak{Z}^{k-1, 
		\perp}$ 
	L.I.~eigenfunctions in (iv). Thus, there are
	\begin{multline*}
		\dim V^{k-1} + \dim \mathfrak{H}^k + \dim \mathfrak{Z}^{k, \perp} + 
		\dim 
		\mathfrak{Z}^{k-1, \perp} + \dim \mathfrak{H}^k \\ = \dim V^{k-1} + 
		\dim 
		V^{k} + \dim \mathfrak{H}^k
	\end{multline*}
	L.I.~eigenfunctions, which completes the proof.
\end{proof}


\input{paper-mcom.bbl}

\end{document}

%% file: figures/plot_complexity.tex
\newcommand{\flops}{\thisrow{SNESSolve_flops}*1E-9}
\newcommand{\memory}{\thisrow{memory}/(2^30)}
\newcommand{\RowTitle}{$\Hgrad$}

\pgfplotstableread[col sep=comma,]{flops/flops.CG.3d.demkowicz.csv}\tableA
\pgfplotstableread[col sep=comma,]{flops/flops.CG.3d.demkowicz.csv}\tableB
\pgfplotstableread[col sep=comma,]{flops/flops.CG.3d.default.csv}\tableC
\pgfplotstableread[col sep=comma,]{flops/flops.CG.3d.default.csv}\tableD

\subfloat[Flop count, \RowTitle]{\resizebox{0.32\textwidth}{!}{
\begin{tikzpicture}
    \begin{loglogaxis}[
         max space between ticks=20,
         ylabel=Gflop,
         xlabel=$p$,
         xtick=data,
         log x ticks with fixed point,
         x tick label style={/pgf/number format/1000 sep=\,},
         legend cell align={left},
         legend style={at={(0.05,0.95)},anchor=north west},
         mark size=2pt, line width=0.75pt]
   \addplot[color=paired2,mark=square*,mark options={fill=paired1}] table[x=degree, y expr=\flops] {\tableA};
   \addplot[color=paired6,mark=triangle*,mark options={fill=paired5}] table[x=degree, y expr=\flops] {\tableC};
   
   \addplot [color=black,mark=none,dashed,domain=5:10] {30*pow(x/10, 6)} 
      coordinate [pos=0.80] (A) 
      coordinate [pos=0.95] (B)
   ;
   \draw (A) -| (B)  
      node [pos=0.75, anchor=west] {6} 
      node [pos=0.25, anchor=north] {1}
   ;

   \addplot [color=black,mark=none,dashed,domain=5:10] {1E2*pow(x/10, 9)} 
      coordinate [pos=0.80] (E) 
      coordinate [pos=0.95] (F)
   ;
   \draw (E) |- (F)  
      node [pos=0.25, anchor=east] {9} 
      node [pos=0.75, anchor=south] {1}
   ;
   \legend{
   $\CG_p(\mathrm{facet})$\\
   $\CG_p(\mathrm{full})$\\
   }
   \end{loglogaxis}
\end{tikzpicture}
}}
\hfill
\subfloat[Peak memory, \RowTitle]{\resizebox{0.32\textwidth}{!}{
\begin{tikzpicture}
    \begin{loglogaxis}[
         max space between ticks=20,
         ylabel=Gbyte,
         xlabel=$p$,
         xtick=data,
         log x ticks with fixed point,
         x tick label style={/pgf/number format/1000 sep=\,},
         legend cell align={left},
         legend style={at={(0.1,0.9)},anchor=north west},
         mark size=2pt, line width=0.75pt]
   \addplot[color=paired2,mark=square*,mark options={fill=paired1}] table[x=degree, y expr=\memory] {\tableA};
   \addplot[color=paired6,mark=triangle*,mark options={fill=paired5}] table[x=degree, y expr=\memory] {\tableC};

   \addplot [color=black,mark=none,dashed,domain=5:10] {0.25*pow(x/10, 4)} 
      coordinate [pos=0.80] (A) 
      coordinate [pos=0.95] (B)
   ;
   \draw (A) -| (B)  
      node [pos=0.75, anchor=west] {4} 
      node [pos=0.25, anchor=north] {1}
   ;
   \addplot [color=black,mark=none,dashed,domain=5:10] {1*pow(x/10, 6)} 
      coordinate [pos=0.80] (E) 
      coordinate [pos=0.95] (F)
   ;
   \draw (E) |- (F)  
      node [pos=0.25, anchor=east] {6} 
      node [pos=0.75, anchor=south] {1}
   ;
   \end{loglogaxis}
\end{tikzpicture}
}}
\hfill
\subfloat[Nonzeros, \RowTitle]{\resizebox{0.32\textwidth}{!}{
\begin{tikzpicture}
    \begin{loglogaxis}[
         max space between ticks=20,
         ylabel=\# nonzeros,
         xlabel=$p$,
         xtick=data,
         log x ticks with fixed point,
         x tick label style={/pgf/number format/1000 sep=\,},
         legend cell align={left},
         legend style={at={(0.1,0.9)},anchor=north west},
         mark size=2pt, line width=0.75pt]
   \addplot[color=paired2,mark=square*,mark options={fill=paired1}] table[x=degree, y=nonzeros] {\tableA};
   \addplot[color=paired6,mark=triangle*,mark options={fill=paired5}] table[x=degree, y=nonzeros] {\tableC};

   \addplot [color=black,mark=none,dashed,domain=5:10] {1.75E7*pow(x/10, 4)} 
      coordinate [pos=0.80] (A) 
      coordinate [pos=0.95] (B)
   ;
   \draw (A) -| (B)  
      node [pos=0.75, anchor=west] {4} 
      node [pos=0.25, anchor=north] {1}
   ;
   \addplot [color=black,mark=none,dashed,domain=5:10] {8E7*pow(x/10, 6)} 
      coordinate [pos=0.80] (E) 
      coordinate [pos=0.95] (F)
   ;
   \draw (E) |- (F)  
      node [pos=0.25, anchor=east] {6} 
      node [pos=0.75, anchor=south] {1}
   ;
   \end{loglogaxis}
\end{tikzpicture}
}}

\renewcommand{\RowTitle}{$\Hcurl$}

\pgfplotstableread[col sep=comma,]{flops/flops.N1curl.3d.demkowicz.csv}\tableA
\pgfplotstableread[col sep=comma,]{flops/flops.N2curl.3d.demkowicz.csv}\tableB
\pgfplotstableread[col sep=comma,]{flops/flops.N1curl.3d.default.csv}\tableC
\pgfplotstableread[col sep=comma,]{flops/flops.N2curl.3d.default.csv}\tableD

\subfloat[Flop count, \RowTitle]{\resizebox{0.32\textwidth}{!}{
\begin{tikzpicture}
    \begin{loglogaxis}[
         max space between ticks=20,
         ylabel=Gflop,
         xlabel=$p$,
         xtick=data,
         log x ticks with fixed point,
         x tick label style={/pgf/number format/1000 sep=\,},
         legend cell align={left},
         legend style={at={(0.05,0.95)},anchor=north west},
         mark size=2pt, line width=0.75pt]
   \addplot[color=paired2,mark=square*,mark options={fill=paired1}] table[x=degree, y expr=\flops] {\tableA};
   \addplot[color=paired4,mark=*,mark options={fill=paired3}] table[x=degree, y expr=\flops] {\tableB};
   \addplot[color=paired6,mark=triangle*,mark options={fill=paired5}] table[x=degree, y expr=\flops] {\tableC};
   \addplot[color=paired8,mark=diamond*,mark options={fill=paired7}] table[x=degree, y expr=\flops] {\tableD};

   \addplot [color=black,mark=none,dashed,domain=5:10] {2.5E2*pow(x/10, 6)} 
      coordinate [pos=0.80] (A) 
      coordinate [pos=0.95] (B)
   ;
   \draw (A) -| (B)  
      node [pos=0.75, anchor=west] {6} 
      node [pos=0.25, anchor=north] {1}
   ;
   \addplot [color=black,mark=none,dashed,domain=5:10] {1E3*pow(x/9, 9)} 
      coordinate [pos=0.80] (E) 
      coordinate [pos=0.95] (F)
   ;
   \draw (E) |- (F)  
      node [pos=0.25, anchor=east] {9} 
      node [pos=0.75, anchor=south] {1}
   ;
   \legend{
   $\Ned_p(\mathrm{facet})$\\
   $\NedTwo_p(\mathrm{facet})$\\
   $\Ned_p(\mathrm{full})$\\
   $\NedTwo_p(\mathrm{full})$\\
   }
   \end{loglogaxis}
\end{tikzpicture}
}}
\hfill
\subfloat[Peak memory, \RowTitle]{\resizebox{0.32\textwidth}{!}{
\begin{tikzpicture}
    \begin{loglogaxis}[
         max space between ticks=20,
         ylabel=Gbyte,
         xlabel=$p$,
         xtick=data,
         log x ticks with fixed point,
         x tick label style={/pgf/number format/1000 sep=\,},
         legend cell align={left},
         legend style={at={(0.1,0.9)},anchor=north west},
         mark size=2pt, line width=0.75pt]
   \addplot[color=paired2,mark=square*,mark options={fill=paired1}] table[x=degree, y expr=\memory] {\tableA};
   \addplot[color=paired4,mark=*,mark options={fill=paired3}] table[x=degree, y expr=\memory] {\tableB};
   \addplot[color=paired6,mark=triangle*,mark options={fill=paired5}] table[x=degree, y expr=\memory] {\tableC};
   \addplot[color=paired8,mark=diamond*,mark options={fill=paired7}] table[x=degree, y expr=\memory] {\tableD};

   \addplot [color=black,mark=none,dashed,domain=5:10] {1*pow(x/10, 4)} 
      coordinate [pos=0.80] (A) 
      coordinate [pos=0.95] (B)
   ;
   \draw (A) -| (B)  
      node [pos=0.75, anchor=west] {4} 
      node [pos=0.25, anchor=north] {1}
   ;
   \addplot [color=black,mark=none,dashed,domain=5:10] {10*pow(x/10, 6)} 
      coordinate [pos=0.80] (E) 
      coordinate [pos=0.95] (F)
   ;
   \draw (E) |- (F)  
      node [pos=0.25, anchor=east] {6} 
      node [pos=0.75, anchor=south] {1}
   ;
   \end{loglogaxis}
\end{tikzpicture}
}}
\hfill
\subfloat[Nonzeros, \RowTitle]{\resizebox{0.32\textwidth}{!}{
\begin{tikzpicture}
    \begin{loglogaxis}[
         max space between ticks=20,
         ylabel=\# nonzeros,
         xlabel=$p$,
         xtick=data,
         log x ticks with fixed point,
         x tick label style={/pgf/number format/1000 sep=\,},
         legend cell align={left},
         legend style={at={(0.1,0.9)},anchor=north west},
         mark size=2pt, line width=0.75pt]
   \addplot[color=paired2,mark=square*,mark options={fill=paired1}] table[x=degree, y=nonzeros] {\tableA};
   \addplot[color=paired4,mark=*,mark options={fill=paired3}] table[x=degree, y=nonzeros] {\tableB};
   \addplot[color=paired6,mark=triangle*,mark options={fill=paired5}] table[x=degree, y=nonzeros] {\tableC};
   \addplot[color=paired8,mark=diamond*,mark options={fill=paired7}] table[x=degree, y=nonzeros] {\tableD};

   \addplot [color=black,mark=none,dashed,domain=5:10] {5E7*pow(x/10, 4)} 
      coordinate [pos=0.80] (A) 
      coordinate [pos=0.95] (B)
   ;
   \draw (A) -| (B)  
      node [pos=0.75, anchor=west] {4} 
      node [pos=0.25, anchor=north] {1}
   ;
   \addplot [color=black,mark=none,dashed,domain=5:10] {2E8*pow(x/8, 6)} 
      coordinate [pos=0.80] (E) 
      coordinate [pos=0.95] (F)
   ;
   \draw (E) |- (F)  
      node [pos=0.25, anchor=east] {6} 
      node [pos=0.75, anchor=south] {1}
   ;
   \end{loglogaxis}
\end{tikzpicture}
}}

\renewcommand{\RowTitle}{$\Hdiv$}

\pgfplotstableread[col sep=comma,]{flops/flops.N1div.3d.demkowicz.csv}\tableA
\pgfplotstableread[col sep=comma,]{flops/flops.N2div.3d.demkowicz.csv}\tableB
\pgfplotstableread[col sep=comma,]{flops/flops.N1div.3d.default.csv}\tableC
\pgfplotstableread[col sep=comma,]{flops/flops.N2div.3d.default.csv}\tableD

\subfloat[Flop count, \RowTitle]{\resizebox{0.32\textwidth}{!}{
\begin{tikzpicture}
    \begin{loglogaxis}[
         max space between ticks=20,
         ylabel=Gflop,
         xlabel=$p$,
         xtick=data,
         log x ticks with fixed point,
         x tick label style={/pgf/number format/1000 sep=\,},
         legend cell align={left},
         legend style={at={(0.05,0.95)},anchor=north west},
         mark size=2pt, line width=0.75pt]
   \addplot[color=paired2,mark=square*,mark options={fill=paired1}] table[x=degree, y expr=\flops] {\tableA};
   \addplot[color=paired4,mark=*,mark options={fill=paired3}] table[x=degree, y expr=\flops] {\tableB};
   \addplot[color=paired6,mark=triangle*,mark options={fill=paired5}] table[x=degree, y expr=\flops] {\tableC};
   \addplot[color=paired8,mark=diamond*,mark options={fill=paired7}] table[x=degree, y expr=\flops] {\tableD};

   \addplot [color=black,mark=none,dashed,domain=5:10] {3E2*pow(x/10, 6)} 
      coordinate [pos=0.80] (A) 
      coordinate [pos=0.95] (B)
   ;
   \draw (A) |- (B)  
      node [pos=0.25, anchor=east] {6} 
      node [pos=0.75, anchor=south] {1}
   ;
   \addplot [color=black,mark=none,dashed,domain=5:10] {2E3*pow(x/10, 9)} 
      coordinate [pos=0.80] (E) 
      coordinate [pos=0.95] (F)
   ;
   \draw (E) |- (F)  
      node [pos=0.25, anchor=east] {9} 
      node [pos=0.75, anchor=south] {1}
   ;
   \legend{
   $\RT_p(\mathrm{facet})$\\
   $\BDM_p(\mathrm{facet})$\\
   $\RT_p(\mathrm{full})$\\
   $\BDM_p(\mathrm{full})$\\
   }
   \end{loglogaxis}
\end{tikzpicture}
}}
\hfill
\subfloat[Peak memory, \RowTitle]{\resizebox{0.32\textwidth}{!}{
\begin{tikzpicture}
    \begin{loglogaxis}[
         max space between ticks=20,
         ylabel=Gbyte,
         xlabel=$p$,
         xtick=data,
         log x ticks with fixed point,
         x tick label style={/pgf/number format/1000 sep=\,},
         legend cell align={left},
         legend style={at={(0.1,0.9)},anchor=north west},
         mark size=2pt, line width=0.75pt]
   \addplot[color=paired2,mark=square*,mark options={fill=paired1}] table[x=degree, y expr=\memory] {\tableA};
   \addplot[color=paired4,mark=*,mark options={fill=paired3}] table[x=degree, y expr=\memory] {\tableB};
   \addplot[color=paired6,mark=triangle*,mark options={fill=paired5}] table[x=degree, y expr=\memory] {\tableC};
   \addplot[color=paired8,mark=diamond*,mark options={fill=paired7}] table[x=degree, y expr=\memory] {\tableD};

   \addplot [color=black,mark=none,dashed,domain=5:10] {0.5*pow(x/10, 4)} 
      coordinate [pos=0.80] (A) 
      coordinate [pos=0.95] (B)
   ;
   \draw (A) |- (B)  
      node [pos=0.25, anchor=east] {4} 
      node [pos=0.75, anchor=south] {1}
   ;
   \addplot [color=black,mark=none,dashed,domain=5:10] {10*pow(x/10, 6)} 
      coordinate [pos=0.80] (E) 
      coordinate [pos=0.95] (F)
   ;
   \draw (E) |- (F)  
      node [pos=0.25, anchor=east] {6} 
      node [pos=0.75, anchor=south] {1}
   ;
   \end{loglogaxis}
\end{tikzpicture}
}}
\hfill
\subfloat[Nonzeros, \RowTitle]{\resizebox{0.32\textwidth}{!}{
\begin{tikzpicture}
    \begin{loglogaxis}[
         max space between ticks=20,
         ylabel=\# nonzeros,
         xlabel=$p$,
         xtick=data,
         log x ticks with fixed point,
         x tick label style={/pgf/number format/1000 sep=\,},
         legend cell align={left},
         legend style={at={(0.1,0.9)},anchor=north west},
         mark size=2pt, line width=0.75pt]
   \addplot[color=paired2,mark=square*,mark options={fill=paired1}] table[x=degree, y=nonzeros] {\tableA};
   \addplot[color=paired4,mark=*,mark options={fill=paired3}] table[x=degree, y=nonzeros] {\tableB};
   \addplot[color=paired6,mark=triangle*,mark options={fill=paired5}] table[x=degree, y=nonzeros] {\tableC};
   \addplot[color=paired8,mark=diamond*,mark options={fill=paired7}] table[x=degree, y=nonzeros] {\tableD};

   \addplot [color=black,mark=none,dashed,domain=5:10] {4E7*pow(x/10, 4)} 
      coordinate [pos=0.80] (A) 
      coordinate [pos=0.95] (B)
   ;
   \draw (A) |- (B)  
      node [pos=0.25, anchor=east] {4} 
      node [pos=0.75, anchor=south] {1}
   ;
   \addplot [color=black,mark=none,dashed,domain=5:10] {8E8*pow(x/10, 6)} 
      coordinate [pos=0.80] (E) 
      coordinate [pos=0.95] (F)
   ;
   \draw (E) |- (F)  
      node [pos=0.25, anchor=east] {6} 
      node [pos=0.75, anchor=south] {1}
   ;
   \end{loglogaxis}
\end{tikzpicture}
}}

%% file: figures/plot_conditioning.tex
\pgfplotstableread[col sep=comma,]{results/beuchler.2d.csv}\tableA
\pgfplotstableread[col sep=comma,]{results/demkowicz.2d.csv}\tableB
\pgfplotstableread[col sep=comma,]{results/beuchler.3d.csv}\tableC
\pgfplotstableread[col sep=comma,]{results/demkowicz.3d.csv}\tableD

\subfloat[2D conditioning\label{fig:cond2d}]{\resizebox{0.49\textwidth}{!}{
\centering
\begin{tikzpicture}[scale=0.75]
    \begin{loglogaxis}[
         xlabel={Degree, $p$}, 
         ylabel={$\kappa_2(\diag(A)^{-1}A)$}, 
         ymax=2E3,
         log x ticks with fixed point,
         x tick label style={/pgf/number format/1000 sep=\,},
         xtick={2,4,6,8,10,12,16,20,24},
         legend cell align={left},
         legend style={at={(0.05,0.95)},anchor=north west}
        ]
        \addplot table [x=degree,y=condA, col sep=comma] {\tableA};
        \addlegendentry{Hierarchical, $K_{\interior \interior}$}
        \addplot table [x=degree,y=condB, col sep=comma] {\tableA};
        \addlegendentry{Hierarchical, $M_{\interior \interior}$}
        \addplot table [x=degree,y=condA, col sep=comma] {\tableB};
        \addlegendentry{This work, $K_{\interior \interior}$}
        \addplot table [x=degree,y=condB, col sep=comma] {\tableB};
        \addlegendentry{This work, $M_{\interior \interior}$}
        \addplot [domain=12:24] {10*pow(x/12,2)} node[above, yshift=2pt, midway, anchor=south west] {$p^2$};
        \addplot [domain=12:24] {30*pow(x/10,4)} node[above, xshift=-2pt, midway, anchor=east] {$p^4$};
    \end{loglogaxis}
\end{tikzpicture}
}}
\subfloat[3D conditioning\label{fig:cond3d}]{\resizebox{0.49\textwidth}{!}{
\centering
\begin{tikzpicture}[scale=0.75]        
    \begin{loglogaxis}[
         xlabel={Degree, $p$},
         ylabel={$\kappa_2(\diag(A)^{-1}A)$}, 
         ymax=2E3,
         log x ticks with fixed point,
         x tick label style={/pgf/number format/1000 sep=\,},
         xtick={4,6,8,10,12,14},
         legend cell align={left},
         legend style={at={(0.05,0.95)},anchor=north west}
        ]
        \addplot table [x=degree,y=condA, col sep=comma] {\tableC};
        \addlegendentry{Hierarchical, $K_{\interior \interior}$}
        \addplot table [x=degree,y=condB, col sep=comma] {\tableC};
        \addlegendentry{Hierarchical, $M_{\interior \interior}$}
        \addplot table [x=degree,y=condA, col sep=comma] {\tableD};
        \addlegendentry{This work, $K_{\interior \interior}$}
        \addplot table [x=degree,y=condB, col sep=comma] {\tableD};
        \addlegendentry{This work, $M_{\interior \interior}$}
        \addplot [domain=8:14] {12*pow(x/10,4)} node[above, yshift=2pt, midway, anchor=south west] {$p^4$};
        \addplot [domain=8:14] {40*pow(x/10,5)} node[above, xshift=-2pt, midway, anchor=east] {$p^5$};
    \end{loglogaxis}
\end{tikzpicture}
}}

%% file: figures/table_hodge_fichera.tex
\begin{table}[htbp]
	\centering
	\caption{Preconditioned MINRES iteration counts to solve the Hodge 
	Laplacians.\label{tab:hodge-fichera}}
	\begin{filecontents*}{results/hodge.N1curl.3d.demkowicz.combined.csv}
level,degree,dofs,beta,nonzeros,memory,Agamma,Aiterations,Acondest,Aruntime,Bgamma,Biterations,Bcondest,Bruntime,Fnonzeros,Fmemory,FAiterations,FAcondest,FAruntime,FBiterations,FBcondest,FBruntime
0,1,6818,1.00E+00,1196871,0,1.00E+00,6,1.122734099552692,0.19539904594421387,1.00E+03,3,1.0000764862078246,0.03241395950317383,1196871,0,6,1.1227340995526829,0.19199728965759277,3,1.0000764862093467,0.030255556106567383
0,2,37168,1.00E+00,1982384,0,1.00E+00,41,3.7407055730708656,2.7889482975006104,1.00E+03,37,3.399569175668679,1.611835241317749,1982384,0,41,3.7407055730708545,2.7341468334198,37,3.399569175651859,1.5892002582550049
0,3,108903,1.00E+00,11225430,0,1.00E+00,48,5.206352258327109,4.951739311218262,1.00E+03,41,6.380705356509834,3.486271619796753,8172488,0,46,4.994321042741186,7.094547748565674,40,4.63366261964508,5.203881025314331
0,4,239875,1.00E+00,51941254,0,1.00E+00,50,5.791576165417283,16.63634181022644,1.00E+03,44,5.299369134673579,13.721510410308838,27427023,0,52,5.778565863302436,24.384047746658325,44,5.2287332059028895,19.750856399536133
0,5,447936,1.00E+00,179455817,0,1.00E+00,51,6.179410875301693,47.02169704437256,1.00E+03,44,5.442610656375472,39.84471869468689,71747403,0,59,6.21381062409535,69.46558427810669,51,5.472320956650543,58.86964797973633
0,6,750938,1.00E+00,492324336,0,1.00E+00,52,6.322533672070028,117.67738485336304,1.00E+03,44,5.563411431180267,99.02978944778442,157037516,0,68,6.6284985963833645,202.77958035469055,57,5.827412150541485,168.4854221343994
\end{filecontents*}
\csvreader[
	head to column names, head to column names prefix=MY,
	tabular=r@{\hskip 0.75em}c@{\hskip 0.75em}c@{\hskip 0.75em}c@{\hskip 0.75em}c|,
	table head=\toprule 
	& \multicolumn{2}{c|}{$\CG_p\times \Ned_p$} \\
	{$p$} & {$\gamma=1$} & {$\gamma=10^3$}\\
	\midrule,
	table foot=\bottomrule
	]{results/hodge.N1curl.3d.demkowicz.combined.csv}{}
	{\MYdegree{} & \MYFAiterations{} (\MYAiterations{})  & \MYFBiterations{} (\MYBiterations{})}
	\hspace{-0.7em}
	\begin{filecontents*}{results/hodge.N1div.3d.demkowicz.combined.csv}
level,degree,dofs,beta,nonzeros,memory,Agamma,Aiterations,Acondest,Aruntime,Bgamma,Biterations,Bcondest,Bruntime,Fnonzeros,Fmemory,FAiterations,FAcondest,FAruntime,FBiterations,FBcondest,FBruntime
0,1,15175,1.00E+00,1889629,0,1.00E+00,6,1.0779133065169102,0.2573878765106201,1.00E+03,3,1.0000764867732348,0.04609322547912598,1889629,0,6,1.0779133065169175,0.2670907974243164,3,1.0000764867697969,0.04555654525756836
0,2,71735,1.00E+00,5412795,0,1.00E+00,60,3.985977127438799,3.432939291000366,1.00E+03,58,3.8055258013982645,2.1861231327056885,3565032,0,60,3.9563924982222396,4.026395082473755,56,3.7917524272895426,2.224029779434204
0,3,196458,1.00E+00,32377231,0,1.00E+00,82,5.33930112669486,7.9960596561431885,1.00E+03,80,5.306092042644434,6.620595455169678,11811178,0,80,5.333535252219968,10.924816131591797,78,5.230004080424464,8.668140649795532
0,4,416122,1.00E+00,142057742,0,1.00E+00,86,5.813014234154041,33.74170231819153,1.00E+03,86,5.768454678487331,32.32911515235901,33858134,0,90,6.003219088655818,41.2843451499939,88,5.868350936458783,38.584872245788574
0,5,757505,1.00E+00,471014307,0,1.00E+00,90,6.1489048188263515,105.86063766479492,1.00E+03,90,6.109469656908772,102.91251945495605,80464392,0,94,6.300618013767786,135.60515332221985,92,6.178990466610536,130.19646191596985
0,6,1247385,1.00E+00,1271745797,0,1.00E+00,94,6.352864991579674,289.01665925979614,1.00E+03,92,6.322097849561058,277.6457667350769,165527617,0,96,6.560308128040137,366.4895272254944,94,6.431392619661497,352.0643768310547
\end{filecontents*}
\csvreader[
	head to column names, head to column names prefix=MY,
	tabular=|c@{\hskip 0.75em}c@{\hskip 0.75em}c|,
	table head=\toprule
	\multicolumn{2}{c|}{$\Ned_p \times \RT_p$}\\ 
	{$\gamma=1$} & {$\gamma=10^3$}\\
	\midrule,
	table foot=\bottomrule
	]{results/hodge.N1div.3d.demkowicz.combined.csv}{}
	{\MYFAiterations{} (\MYAiterations{})  & \MYFBiterations{} (\MYBiterations{})}
	\hspace{-0.7em}
	\begin{filecontents*}{results/hodge.DG.3d.demkowicz.combined.csv}
level,degree,dofs,beta,nonzeros,memory,Agamma,Aiterations,Acondest,Aruntime,Bgamma,Biterations,Bcondest,Bruntime,Fnonzeros,Fmemory,FAiterations,FAcondest,FAruntime,FBiterations,FBcondest,FBruntime
0,1,13795,1.00E+00,776036,0,1.00E+00,5,1.0244868309458361,0.284423828125,1.00E+03,3,1.0000239859393887,0.018648147583007812,776036,0,5,1.0244868309458357,0.29256105422973633,3,1.000023985939389,0.018573284149169922
0,2,59237,1.00E+00,3679905,0,1.00E+00,25,4.154139975424115,1.146599531173706,1.00E+03,16,3.6721555741070366,0.349900484085083,1832142,0,23,4.079481672394113,1.6724474430084229,14,3.5410565564376215,0.3823409080505371
0,3,154178,1.00E+00,22954215,0,1.00E+00,28,5.169991526433924,2.044867992401123,1.00E+03,17,4.618204987318503,0.9101767539978027,5441079,0,26,5.0840855619419605,2.583207845687866,17,4.491490021931282,0.9409024715423584
0,4,316470,1.00E+00,98705317,0,1.00E+00,28,5.587839275039278,7.381855010986328,1.00E+03,18,5.674494468346491,4.656381845474243,13952047,0,29,5.669208990026076,6.880747556686401,19,5.051054607097179,3.9751572608947754
0,5,563965,1.00E+00,320543427,0,1.00E+00,29,5.94763984159631,22.547679901123047,1.00E+03,19,5.462201362438125,14.93855357170105,30535662,0,31,6.0235845784038204,26.30301785469055,19,5.448634839838638,16.3425874710083
0,6,914515,1.00E+00,857683974,0,1.00E+00,29,6.089688075749077,63.12453103065491,1.00E+03,19,5.649254379209121,45.411237478256226,59155194,0,31,6.264644013256889,66.08153438568115,20,5.7703178340200125,44.348308801651
0,7,1385972,1.00E+00,2013755583,0,1.00E+00,29,6.193526172297445,160.1161162853241,1.00E+03,19,5.786496080898817,110.7159628868103,104566567,0,34,8.368894652434573,161.35593032836914,20,6.274131543204602,99.18183159828186
\end{filecontents*}
\csvreader[
	head to column names, head to column names prefix=MY,
	tabular=|c@{\hskip 0.75em}c@{\hskip 0.75em}c,
	table head=\toprule
   \multicolumn{2}{c}{$\RT_p\times \DG_{p-1}^{\phantom{1}}$}\\ 
	{$\gamma=1$} & {$\gamma=10^3$}\\
	\midrule,
	filter ifthen=\not\equal{\MYdegree}{7},
	table foot=\bottomrule
	]{results/hodge.DG.3d.demkowicz.combined.csv}{}
	{\MYFAiterations{} (\MYAiterations{})  & \MYFBiterations{} (\MYBiterations{})}
\end{table}

%% file: figures/plot_decoupling.tex
\pgfplotstableread[col sep=comma,]{./results/decoupling.N1curl.3d.csv}\tableA
\pgfplotstableread[col sep=comma,]{./results/decoupling.N1div.3d.csv}\tableB
\begin{tikzpicture}[scale=0.75]        
    \begin{loglogaxis}[
         xlabel={Degree, $p$},
         ylabel={$\omega^{k, \rmone/\rmtwo}_{1,p}$},
         ymax=2E0,
         log x ticks with fixed point,
         x tick label style={/pgf/number format/1000 sep=\,},
         xtick={1,2,4,6,8,10},
         legend cell align={left},
         legend style={at={(0.95,0.95)},anchor=north east}
        ]
        \addplot table [x=degree,y=omega0, col sep=comma] {\tableA};
        \addlegendentry{$\Ned_p$}
        \addplot table [x=degree,y=omega0, col sep=comma] {\tableB};
        \addlegendentry{$\RT_p$}
        \addplot [domain=4:10] {1/pow(x,2)} node[above, yshift=-2pt, midway, anchor=south west] {$p^{-2}$};
        \addplot [domain=4:8] {1/pow(x,3)}; 
        \addplot [domain=8:10] {1/pow(x,3)} node[above, yshift=-2pt, midway, anchor=south west] {$p^{-3}$};
    \end{loglogaxis}
\end{tikzpicture}

%% file: paper-mcom.bbl
\providecommand{\bysame}{\leavevmode\hbox to3em{\hrulefill}\thinspace}
\providecommand{\MR}{\relax\ifhmode\unskip\space\fi MR }
\providecommand{\MRhref}[2]{%
  \href{http://www.ams.org/mathscinet-getitem?mr=#1}{#2}
}
\providecommand{\href}[2]{#2}